\documentclass[12pt,letterpaper]{report}
\usepackage{mathtools}
\usepackage{amsmath}
\usepackage{amsthm}
\usepackage[inline]{enumitem}
\usepackage{amssymb}
\usepackage{mathrsfs}
\usepackage[top=1in,bottom=1in,left=1.3in,right=1.3in]{geometry}
\usepackage{tikz}
\usepackage{tikz-cd}
\usetikzlibrary{decorations.markings}
\usepackage{import}
\usepackage{pdfpages}
\usepackage{transparent}
\usepackage{xcolor}
\usepackage[colorlinks=true]{hyperref}
\usepackage{cleveref}

\newtheorem{THM}{Theorem}
\newtheorem{COR}{Corollary}
\newtheorem{thm}{Theorem}[section]
\newtheorem{lem}[thm]{Lemma}
\newtheorem{prop}[thm]{Proposition}
\newtheorem{cor}[thm]{Corollary}
\theoremstyle{definition}
\newtheorem{dfn}[thm]{Definition}
\newtheorem{question}{Question}
\newtheorem{rk}[thm]{Remark}
\newtheorem{ex}[thm]{Example}
\crefname{ex}{example}{examples}

\DeclareMathOperator{\stab}{stab}
\DeclareMathOperator{\aut}{Aut}
\DeclareMathOperator{\out}{Out}
\DeclareMathOperator{\inn}{Inn}
\DeclareMathOperator{\gl}{GL}
\DeclareMathOperator{\Mod}{Mod}
\newcommand{\cato}{\operatorname{CAT}(0)}
\DeclareMathOperator{\st}{st}
\DeclareMathOperator{\ad}{ad}

\usepackage{etoolbox}
\apptocmd{\sloppy}{\hbadness 10000\relax}{}{}

\begin{document}
\pagenumbering{roman}
\begin{titlepage}
	\large
		\hfill\vspace{2em}

	\begin{center}

	{\LARGE\bfseries
		Train Tracks 
		on Graphs of Groups \\[0.4em]
		and Outer Automorphisms  
		of \\[0.5em] Hyperbolic Groups
	}
	\vspace{1em}

	\emph{a dissertation submitted by}
	\vspace{1em}

	{\huge\bfseries Rylee Alanza Lyman }
	\vspace{1em}
	
	\emph{in partial fulfillment of the requirements for the degree of}
	\vspace{1em}

	{\LARGE\scshape Doctor of Philosophy}

	\emph{in}

	{\LARGE\scshape Mathematics}
	\vspace{1em}

	{\Large\scshape  Tufts University}

	May 2020
	\vspace{1em}

	Advisor: Kim Ruane
	\end{center}
\end{titlepage}

\setcounter{page}{1}
\begin{abstract}
	\thispagestyle{plain}
	Stallings remarked that an outer automorphism
	of a free group may be thought of as a subdivision of a graph
	followed by a sequence of folds. In this thesis, we prove that automorphisms
	of fundamental groups of graphs of groups satisfying this condition
	may be represented by irreducible train track maps in the sense of
	Bestvina--Handel (we allow collapsing invariant subgraphs).
	Of course, we construct relative train track maps as well.
	Along the way, we give a new exposition of the Bass--Serre theory
	of groups acting on trees, morphisms of graphs of groups,
	and foldings thereof.
	We produce normal forms for automorphisms of free products
	and extend an argument of Qing--Rafi to show that they are not quasi-geodesic.
	As an application, we answer affirmatively a question of Paulin:
	outer automorphisms of finitely-generated word hyperbolic groups
	satisfy a dynamical trichotomy generalizing the Nielsen--Thurston
	``periodic, reducible or pseudo-Anosov.'' At the end of the thesis
	we collect some open problems we find interesting.
\end{abstract}

\setcounter{page}{2}
\pagebreak
{

	\vspace*{6em}
	\centering
	\itshape
	To my parents,  who always believed in me
	and challenged me to pursue excellence.
	\vspace{4em}

	To Jeffrey, Jeremy, Kay and Mara; 
	it is a gift to be seen and heard. 
	\vspace{2em}

	-- \& --
	\vspace{2em}

	To Kim,  who struck the impossible balance
	between guidance and independence.
}

\chapter*{Acknowledgments}
\addcontentsline{toc}{chapter}{Acknowledgments}

The mere existence of a PhD thesis speaks to decades of
assistance, intentional and unintentional,
without which this document would not exist.
Considered from this perspective, my gratitude overwhelms me;
I owe far more thanks than I will be able to write here.

I first learned that math could be beautiful from my parents,
for whom my gratitude could fill a page by itself.
I particularly want to thank my mother for two moments:
for advocating for me with my sixth grade math teacher,
and for not letting me give up math when I was overwhelmed by
its difficulty my first year of college.
Her favorite math course in college was abstract algebra;
it is perhaps no surprise that I would fall in love with group theory.

Five years is a long time to work at something.
For listening, supporting me, and being a fellow-traveler,
I want to thank Jeffrey Lyman, Jeremy Budd, Kay Gabriel,
Mara Knoecklein and Max Shinn.
The community of math grad students at Tufts is truly something special;
in particular I want to mention
Ayla S\'anchez, Burns Healy, Curtis Heberle, Daniel Keliher,
Michael Ben-Zvi, Remy Bohm (who might as well have been a grad student),
Nariel Monteiro, Nathan Fisher, Peter Ohm and Sunrose Shrestha.
Noah Barrientos, Sarah Rich and Tyler Mitchell 
went above and beyond in looking out for me.

Being able to talk mathematics with someone is a rare delight. 
Jason Behrstock and Tim Susse introduced me to geometric group theory.
Robert Kropholler was the first to talk with me as a peer
at the edge of knowledge.
Daniel Keliher was always game to shoot for the moon
or listen to a half-baked observation.
Santana Afton tolerated my over-enthusiasm with grace.
Abdul Zalloum, Francesco Cellarosi, Dave Futer, and Thomas Ng
were excellent hosts.
Besides these I also want to mention
Autumn Kent, Corey Bregman, Dan Margalit, Elizabeth Field, Emily Stark,
Ian Runnels, Justin Lanier,
Kasra Rafi, Kathryn Mann, Lee Mosher, Marissa Loving, Mark Hagen,
Matt Clay, Matt Zaremsky, Mladen Bestvina, Olivia Borghi, Paul Plummer, Radhika Gupta,
Rose Morris-Wright, Sami Douba, Tarik Aougab, Tyrone Ghaswala and Yvon Verbene.

Even high school students seem to be aware that good mathematical writing
is hard to find. 
Dan Margalit, Emily Riehl, Ian Stewart, Martin Gardner, Mladen Bestvina,
and William Thurston are a pleasure to read and to emulate.

For helpful comments on the work in this thesis and closely-related
work in \cite{TrainTracksOrbigraphs}, 
I want to thank
Genevieve Walsh, Kim Ruane, Mark Hagen, Mladen Bestvina,
Robert Kropholler, and an anonymous referee.

Finally I want to thank my advisor, Kim Ruane,
for masterful guidance in my development as a thinker.
She was patient when I made no apparent progress,
brought me back when I wandered too far into the weeds,
offered new directions when I was stuck,
and encouraged my independence.
It is truly a gift to have such a mentor.

\tableofcontents
\pagebreak
\setcounter{page}{0}
\pagenumbering{arabic}
\hfill\vspace{1.5in}

\thispagestyle{empty}
\begin{center}
{\Huge\bfseries
		Train Tracks 
		on Graphs of Groups \\[0.35em]
		and Outer Automorphisms  
		of \\[0.5em] Hyperbolic Groups
}
\end{center}
\pagebreak

\chapter*{Introduction}
\label{introduction}
\addcontentsline{toc}{chapter}{\nameref{introduction}}

The Nielsen--Thurston classification of 
elements of the mapping class group of a surface
has spurred a massive development in the theory of mapping class groups,
and by analogy, outer automorphisms of free groups.
In 1992, Bestvina and Handel \cite{BestvinaHandel}
constructed \emph{relative train track maps}
representing each outer automorphism of a free group.

Relative train track maps play a similar role to Thurston's
normal form for mapping classes,
and led Paulin \cite{Paulin} to ask
whether all outer automorphisms of finitely-generated word hyperbolic groups 
satisfy a dynamical trichotomy generalizing
the Nielsen--Thurston ``periodic, reducible, or pseudo-Anosov.''

The main contribution of this thesis
is the construction of relative train track maps
representing those automorphisms of a group
that may be \emph{realized} on some graph of groups splitting
as a subdivision followed by a sequence of folds.
As an application we answer Paulin's question in the affirmative.

We think it useful to digress to discuss the source of this project;
we will discuss the results of this thesis
in more detail in \hyperref[results]{the following section}
of this introduction,
to which the impatient reader is invited to skip ahead.

\paragraph{}
This thesis, and the author's graduate career, 
begins with the outer automorphism groups $\out(W_n)$, where
\[ W_n = \underbrace{C_2 * \dotsb * C_2}_{n \text{ factors}} 
= \langle a_1,\dotsc,a_n \mid a_1^2 = \dotsb = a_n^2 = 1 \rangle, \]
that is, $W_n$ is 
the free product of $n$ copies of the cyclic group of order $2$.
As usual, the group $\out(W_n)$ is the quotient 
\[ \out(W_n) = \aut(W_n)/\inn(W_n), \]
where $\inn(W_n)$ denotes the group of inner automorphisms,
arising from the conjugation action of $W_n$ on itself.
The groups $W_n$ are the simplest infinite Coxeter groups.

The groups $\out(W_n)$ share some of the properties of
the braid groups and some of the properties of
$\out(F_{n-1})$, the outer automorphism group of a free group
of rank $n-1$.
Indeed, $\out(W_n)$ contains the quotient of 
the $n$-strand braid group by its (infinite  cyclic) center
as a subgroup (of infinite index once $n \ge 4$),
and  $\out(W_n)$ is isomorphic to the quotient of
a subgroup of $\out(F_{n-1})$ by a finite normal subgroup.
In particular $\out(W_1)$ is trivial, $\out(W_2)$ is cyclic of order $2$,
and $\out(W_3)$ is isomorphic to the projective linear group
$\operatorname{PGL}_2(\mathbb{Z})$.
A particularly striking connection is the sporadic isomorphism
\[ \aut(B_4) \cong \aut(F_2) \cong \aut(W_3), \]
where $B_4$ denotes the $4$-strand braid group.
For $n \ge 4$, the groups $\out(W_n)$,
and outer automorphisms of free products
more generally are much less well-understood than
either the braid groups or $\out(F_{n-1})$.

Brady and McCammond \cite{BradyMcCammond}
construct a contractible metric simplicial complex on
which the $n$-strand braid group acts geometrically.
For $n \le 6$, this metric is known to satisfy Gromov's
non-positive curvature condition $\cato$,
and conjecturally the condition is satisfied for all $n$.
On the other hand, for $n \ge 3$,
$\out(F_n)$ cannot act geometrically on a $\cato$ metric space.
This suggests the following question of Kim Ruane.

\begin{question}\label{outWnquestion}
	For $n \ge 4$, does $\out(W_n)$ act geometrically
	on a $\cato$ metric space?
\end{question}

Charles Cunningham in his thesis \cite{Cunningham} showed that
a candidate simplicial complex constructed by
McCullough and Miller in \cite{McCulloughMiller}
equipped with a simplicial action of $\out(W_n)$
with finite stabilizers and finite quotient
does not support an $\out(W_n)$-equivariant $\cato$ metric.
The result is reminiscent of Bridson's thesis,
where the same result is shown for the spine of
Culler--Vogtmann's Outer Space.

Gersten \cite{GerstenFreeByZ} gave the first proof
that $\out(F_n)$ cannot act geometrically on a $\cato$ metric space
when $n \ge 4$
by constructing a group $G_\Psi$
which cannot be a subgroup of a group acting geometrically on
a $\cato$ metric space, but \emph{is} a subgroup of $\aut(F_n)$
for $n \ge 3$ and of $\out(F_n)$ when $n \ge 4$.

The group $G_\Psi$ fits in a short exact sequence
\[\begin{tikzcd}
	1 \ar[r]  & F_3 \ar[r] & G_\Psi \ar[r] & \mathbb{Z} \ar[r] & 1,
\end{tikzcd}\]
and is thus called a \emph{free-by-cyclic} group.
The sequence splits; $G_\Psi$ decomposes as a semidirect product
$F_3\rtimes_\Psi\mathbb{Z}$, where abusing notation
we identify $\Psi\colon \mathbb{Z} \to \aut(F_3)$ with the image
$\Psi\colon F_3 \to F_3$ of the generator $1$.
Automorphisms of free groups are determined by their action
on a free basis;
writing $F_3 = \langle a,b,c \rangle$,
Gersten's automorphism $\Psi$ is described by its action as
\[
	\Psi \begin{dcases}
		a \mapsto a \\
		b \mapsto ba \\
		c \mapsto caa.
	\end{dcases}
\]
A simple computation reveals that word lengths of elements
of $F_3$ grow at most linearly under repeated application of $\Psi$:
we say $\Psi$ is \emph{polynomially-growing.}
To answer \Cref{outWnquestion} in the negative,
the author set out to find a Gersten-type example
$W_n\rtimes_\Phi \mathbb{Z}$.
Surprisingly, the following theorem implies no example exists.
\begin{THM}[Lyman '19]
	Suppose $\Phi\colon  W_n \to W_n$ is a polynomially-growing
	automorphism. There exists $k \ge 1$ such that
	$W_n\rtimes_{\Phi^k}\mathbb{Z}$ acts
	geometrically on a $\cato$ $2$-complex.
	The power $k$ is bounded by
	Landau's function $g(n) < n!$.
\end{THM}

(In fact, the theorem holds for more classes of automorphisms
of free products. Details appear in \cite{SomeNewCAT0FreeByCyclicGroups}.)
Automorphisms of $W_n$ are closely related to 
automorphisms of $F_{n-1}$ which send generators
of some fixed free basis $x_1,\dotsc,x_n$ 
to \emph{palindromes} in the $x_i$---words
spelled the same forwards and backwards.
Let $\iota\colon F_n \to F_n$ denote the automorphism
inverting each element of our fixed free basis.
It is a pleasant exercise to show that
palindromic automorphisms $\Phi\colon F_n \to F_n$
are precisely those automorphisms centralizing $\iota$
\cite[Section 2]{GloverJensen}.

\begin{COR}[Lyman '19]
	If $\Phi \colon F_n \to F_n$ is a polynomially-growing palindromic automorphism,
	there exists an integer $k \ge 1$ such that $F_n\rtimes_{\Phi^k}\mathbb{Z}$
	acts geometrically on a $\cato$ $2$-complex.
	The power $k$ is bounded by Landau's function $g(n+1) < (n+1)!$.
\end{COR}

A direct construction proving the corollary
is easiest to give for those palindromic
automorphisms $\Phi$ of $F_n = \langle x_1,\dotsc,x_n \rangle$
which are \emph{upper-triangular}, in the sense
that for each $i$, the element $\Phi(x_i)$ may be 
written as a word using only those $x_j$ satisfying $j \le i$.
In fact, the corollary is proven
by arguing that in fact every polynomially-growing,
palindromic automorphism $\Phi\colon F_n \to F_n$
has a power $k \ge 1$ for which $\Phi^k$ is upper-triangular
with respect to some free basis.

The demonstration of this latter statement
requires some understanding of automorphisms
$\Phi\colon F_n \to F_n$---or better, their outer classes
$\varphi \in \out(F_n)$---independent 
from their expression in a given basis.
The main tools for this understanding are
\emph{relative train track maps,} originally introduced 
by Bestvina and Handel
in \cite{BestvinaHandel} to prove a conjecture by Scott
that the rank of the fixed subgroup
\[ \{ x \in F_n : \Phi(x) = x \} \]
is at most $n$ for all automorphisms $\Phi\colon F_n \to F_n$.

As we learn in algebraic topology,
free groups are fundamental groups of graphs.
Since every graph $\Gamma$ is an Eilenberg--Mac Lane space,
every automorphism of a free group may be represented
by a basepoint-preserving homotopy equivalence of a graph.
The (free) homotopy class of a homotopy equivalence
determines an outer automorphism of the fundamental group.
A \emph{relative train track map} is a homotopy equivalence
of a graph
with specified extra structure that allows for
exactly the kind of basis-independent reasoning we hoped for.

Thus to complete the proof of the corollary,
one would hope for a ``palindromic'' relative train track map
of some kind.
It is a theorem of Culler \cite{Culler} building on
earlier work of Karrass, Pietrowski and Solitar \cite{KPS}
that finite groups of (outer) automorphisms of free groups
may be realized as homeomorphisms of graphs,
and Bass--Serre theory allows one to consider
a kind of ``orbifold quotient'' 
of the action \cite{Bass}---this is a \emph{graph of groups}
(with, as it  turns out, fundamental group $W_n$
in the case of $\iota$).
One way to guarantee that a relative train track map
respects the action of the outer automorphism
would be to work directly in the quotient graph of groups.
This is the work of  this thesis.

\section*{Statement of Results}
\label{results}
\addcontentsline{toc}{section}{\nameref{results}}

We briefly describe the contents of this thesis,
referring the curious reader to the appropriate section
for more details.

\paragraph{\Cref{graphsofgroups}: Graphs of Groups.}
We give an exposition of Bass--Serre theory.
All the results in the chapter are contained in
or follow easily from \cite{Trees} or \cite{Bass}, but the proofs are new.
The central theorem is \Cref{BassSerreThm} establishing
a correspondence between graphs of groups $(\Gamma,\mathscr{G})$
and actions of the associated fundamental group $\pi_1(\Gamma,\mathscr{G},p)$
without inversion on a tree $\tilde\Gamma$.

Our proof of \Cref{BassSerreThm} first gives
a geometric construction of the tree $\tilde\Gamma$ from the data
of the graph of groups $(\Gamma,\mathscr{G})$.
Our definition of the fundamental group $\pi_1(\Gamma,\mathscr{G},p)$
avoids the use of the \emph{path group,}
preferring instead to directly develop a homotopy theory
of edge paths in graphs of groups.
This approach eases the transition in \Cref{traintrackchapter}
to a more directly topological approach.

Graphs of groups form a \emph{category,}
in particular there is a notion, essentially due to Bass,
of a \emph{morphism} between  graphs of groups.
Let us give the brief explanation.
Each graph $\Gamma$ determines a \emph{small category,}
also called $\Gamma$,
with \begin{enumerate*}[label=(\roman*)] \item objects the vertices
of the first barycentric subdivision of the graph
and \item arrows from barycenters of edges to the vertices incident to the edge.
\end{enumerate*}

A graph of groups $(\Gamma,\mathscr{G})$ is a connected graph $\Gamma$
and a \emph{functor}
$\mathscr{G}\colon \Gamma \to \operatorname{Group^{mono}}$
to a subcategory of the category of groups.
A \emph{morphism} of graphs $f\colon \Lambda \to \Gamma$
is a functor of the associated small categories,
and thus a morphism of graphs of groups 
$f\colon (\Lambda,\mathscr{L}) \to (\Gamma,\mathscr{G})$ ought
to be a morphism of functors---i.e. a natural transformation
$f\colon \mathscr{L} \Longrightarrow \mathscr{G}f$.

We develop this perspective beginning in \Cref{morphisms}.
It turns out that the right notion of a morphism is instead a
\emph{pseudonatural} transformation;
certain diagrams only commute up to
conjugation by specified elements of the target vertex groups.
This is essentially the definition given by Bass in \cite{Bass}
except for two small  differences,
the first of which is that our  morphisms may collapse edges to vertices.
In \Cref{diagramfill1} and \Cref{diagramfill2},
we prove a correspondence between morphisms of graphs of groups 
\[ f\colon (\Lambda,\mathscr{L}) \to (\Gamma,\mathscr{G}) \]
and equivariant  morphisms of Bass--Serre trees
\[ \tilde f \colon \tilde\Lambda \to \tilde\Gamma \]
preserving the distinguished lift of the basepoint.
It is a consequence of the second difference in our definitions
that Bass's definition ignores basepoints.
Bass's definition is aimed at understanding
the group of equivariant automorphisms of the tree $\tilde\Gamma$,
where it is convenient to allow inner automorphisms,
but seems inappropriate for most other applications.

Let us remark that the correspondence just mentioned
allows one to understand the pseudonaturality condition:
the identification of a group $G$  acting without inversions
on a tree $T$ with the fundamental group of the corresponding
graph of groups involves a choice of fundamental domain for the action of $G$ on $T$.
The conjugating elements in the definition of 
$f\colon (\Lambda,\mathscr{L}) \to (\Gamma,\mathscr{G})$ arise
to measure the difference between the image of the fundamental domain
for $\tilde\Lambda$ and the fundamental domain for $\tilde\Gamma$.

Bass's paper leaves implicit the classification of covering spaces
of a graph of groups.
The result follows quickly from \Cref{diagramfill1}
and \Cref{diagramfill2}; we prove it as \Cref{Galoiscorrespondence}.

\paragraph{\Cref{foldingchapter}: Folding.}
Stallings foldings, introduced in \cite{Stallings}
are a fundamental tool in studying
outer automorphisms of free groups and
maps of graphs more generally.
Stallings subsequently gave an extension in \cite{StallingsGTrees}
to morphisms between Bass--Serre trees,
and the theory was later refined and extended in 
\cite{BestvinaFeighnGTrees,Dunwoody,KapovichWeidmannMiasnikov}.
Most of the first four sections are expository,
giving a uniform treatment of results in the previously mentioned papers.
Our one contribution in these early sections
is a study of two kinds of collapsing that can occur:
collapsing of \emph{subgraphs}, where the fundamental group
$\pi_1(\Gamma,\mathscr{G})$ is preserved but the topology of $\Gamma$
changes, and collapsing of \emph{stabilizers,}
where the graph $\Gamma$ is preserved,
but the graph of groups structure $\mathscr{G}$ is changed.
The result is \Cref{collapsing}, which serves 
both as a preliminary step for folding a morphism 
$f\colon (\Lambda,\mathscr{L}) \to (\Gamma,\mathscr{G})$
and, as Dunwoody observed, an intermediary step for folding
in the case where $\pi_1(\Lambda,\mathscr{L})$ and
$\pi_1(\Gamma,\mathscr{G})$ differ.

In \Cref{applications}, we use folding to 
study automorphism groups of free products.
Stallings remarked that an automorphism of a free group
may be thought of as a subdivision of a graph followed
by a sequence of folds.
For us, a \emph{topological realization} of an endomorphism
$\pi_1(\Gamma,\mathscr{G}) \to \pi_1(\Gamma,\mathscr{G})$
is a continuous map $f\colon (\Gamma,\mathscr{G}) \to (\Gamma,\mathscr{G})$
which is a subdivision of the domain followed by a morphism
(in the sense of the previous chapter) from the subdivided graph of groups
to $(\Gamma,\mathscr{G})$.
Topologically realizable automorphisms are exactly those for which
Stallings's observation applies, and folding may be used to study them.

We give a conceptually simple proof in \Cref{fouxerabinovitch}
of a result of Fouxe-Rabinovitch finding a generating set
for the automorphism group of a free product.
Actually, Fouxe-Rabinovitch gives a \emph{presentation;} we do not.
By carefully analyzing the process of folding,
we give as \Cref{combing} a method for determining whether
a topologically realizable endomorphism 
$\Phi\colon \pi_1(\Gamma,\mathscr{G}) \to \pi_1(\Gamma,\mathscr{G})$
is an automorphism, and if so give a normal form for $\Phi$
in the Fouxe-Rabinovitch generators.
Every automorphism of a free product is topologically realizable
on a (Grushko) splitting we call the ``thistle.''

The method given above is algorithmic in principle,
although is not in full generality a true algorithm
because the factors in the free product may be arbitrary.

In the case of a free group $F_n$, 
the methods in the previous paragraph are well-known,
but their extension to free  products appears to be new.
In the case of the free group, Qing and Rafi
show in \cite{QingRafi} that any normal forms that rely solely on folding
are not quasi-geodesic in the Cayley graph of $\aut(F_n)$ 
once the rank of the free group is at least $3$.
We show in \Cref{notquasigeodesic} that our normal forms
are likewise not quasi-geodesic once the Kurosh rank of the free product
$A_1*\dotsb*A_n*F_k$ satisfies $n+k \ge 4$ or $k\ge 3$,
where the $A_i$ are freely indecomposable and not infinite cyclic.

Let us remark that folding techniques may be used to study
automorphisms of virtually free groups.
It would be interesting, for instance,
to know whether a uniform method can be given for computing
a presentation for the subgroup of $\aut(G)$ or $\out(G)$
that can be topologically realized on a given splitting.
We collect a number of such questions at the end of this thesis.

\paragraph{\Cref{traintrackchapter}: Train Track Maps}
In \cite{BestvinaHandel}, Bestvina and Handel
construct \emph{relative train track maps}
for outer automorphisms of free groups.
A \emph{train track map} is an efficient topological realization
of a given outer automorphism; it minimizes a certain
algebraic integer $\lambda \ge 1$ called its \emph{stretch factor.}
The idea of the proof is to associate a nonnegative integral matrix
to a topological realization and use folding and some auxiliary moves
to minimize the associated Perron--Frobenius eigenvalue.
A \emph{relative} train track map preserves a filtration
\[\varnothing = \Gamma_0 \subset \Gamma_1 \subset \dotsb \subset \Gamma_m = \Gamma \]
of the graph of groups $(\Gamma,\mathscr{G})$ into subgraphs
such that within each stratum 
\[H_k = \overline{\Gamma_k \setminus \Gamma_{k-1}}, \]
our topological realization $f\colon(\Gamma,\mathscr{G}) \to (\Gamma,\mathscr{G})$
``looks like'' a train track map.
We make this precise in \Cref{ttfromrtt}
by showing that if $H_k$ is an irreducible stratum,
then the topological realization obtained by first restricting to
$\Gamma_k$ and then collapsing each component of $\Gamma_{k-1}$
satisfies the train track property (although \emph{a priori}
the resulting graph of groups is disconnected.)

The main result of the chapter is an extension of Bestvina--Handel's
methods to topological realizations of automorphisms of 
fundamental groups of graphs of groups.
In \Cref{traintrackthm},
we show that if an automorphism $\Phi\colon G \to G$ 
can be topologically realized, then it can be realized  by a train track map.
In the proof, we \emph{collapse} any invariant subgraphs that occur,
which explains why we get a train track map and not a relative train track map.

If one repeats the argument on any invariant subgraphs that occur,
the result is a \emph{hierarchy} of irreducible train track maps,
\[ f_1\colon (\Lambda_1,\mathscr{L}_1) \to (\Lambda_1,\mathscr{L}_1),
\dotsc,f_m\colon (\Lambda_m,\mathscr{L}_m) \to (\Lambda_m,\mathscr{L}_m), \]
where each $(\Lambda_i,\mathscr{L}_i)$ corresponds to a vertex of
$(\Lambda_j,\mathscr{L}_j)$ with vertex group  $\pi_1(\Lambda_i,\mathscr{L}_i)$
for some $j > i$.
In the case where each graph of groups has trivial edge groups,
we prove a partial converse to \Cref{ttfromrtt} as \Cref{rttfromtt},
namely that one can graft the $(\Lambda_i,\mathscr{L}_i)$ together
to yield a topological realization
$f\colon(\Gamma,\mathscr{G}) \to (\Gamma,\mathscr{G})$
that has no valence-one or valence-two vertices, only irreducible strata,
and each irreducible stratum satisfies
a weakening of the relative train track properties
due to Gaboriau, Jaeger, Levitt and Lustig \cite{GaboriauJaegerLevittLustig}
called a \emph{partial train track property.}
Of course, one may use techniques of Bestvina--Handel 
to modify $f\colon(\Gamma,\mathscr{G}) \to (\Gamma,\mathscr{G})$,
yielding a relative train track map with the same exponentially-growing strata,
which we record as \Cref{rttexistence}

As an application of \Cref{rttexistence},
we answer affirmatively as \Cref{trichotomy} a question of Paulin,
who asks whether each automorphism of a hyperbolic group
satisfies a generalization of the Nielsen--Thurston
``periodic, reducible, or pseudo-Anosov'' trichotomy.

\paragraph{}
We expect \Cref{traintrackthm} and a future strengthening of it
constructing the \emph{completely split relative train track maps}
of \cite{FeighnHandel} on graphs of groups to be a useful tool
in studying outer automorphisms of free products,
and subgroups of outer automorphisms of free groups preserving 
the conjugacy classes of certain free factors.
For example, we have used \Cref{traintrackthm} in the case
of $W_n$, the free product of $n$ copies of the cyclic group of order $2$,
to show that \emph{fully irreducible} elements $\varphi \in \out(W_n)$
are either realizable as pseudo-Anosov homeomorphisms of once-punctured orbifolds,
or $W_n\rtimes_\phi\mathbb{Z}$ is word-hyperbolic.

Perhaps most noticeably absent from this thesis is the connection
to deformation spaces of trees, where a train track map corresponds
to an axis for the outer automorphism in the (asymmetric) Lipschitz metric.
This is how Francaviglia and Martino \cite{FrancavigliaMartino}
prove the existence of relative train
track maps for outer automorphisms of free products.
We collect some questions and directions for future work at the end of this thesis.

\chapter{Graphs of Groups}
\label{graphsofgroups}

Bass--Serre theory~\cite{Trees} studies the structure 
of groups acting on trees.
Suppose a group $G$ acts on a tree $T$. The idea is to treat
the quotient $G\backslash T$ as a kind of ``orbifold,''
in the sense that understanding the topology of the quotient space
should be in some way equivalent to understanding the action of $G$ on $T$.
The original definition is combinatorial. Scott and Wall~\cite{ScottWall}
gave a topological interpretation by introducing \emph{graphs of spaces.}
In this thesis we need both the combinatorial 
clarity of graphs of groups
and the ability to reason about our objects 
using the tools of one-dimensional topology.

A reader familiar with Bass--Serre theory may skim this chapter
up until \Cref{morphisms}.
Our exposition differs from that in~\cite{Trees} only in its
preference for combinatorial topology. We construct the tree first,
then analyze the resulting action of the fundamental group on it.

This decision pays off in the latter part of the chapter
as it makes clear the mysterious appearance of a \emph{pseudonaturality}
condition appearing in the definition of a morphism of a graph of groups.
Many authors mention casually that vertex and edge groups of a graph of groups
are really only well defined ``up to inner automorphism.'' Indeed,
a careful reading of \Cref{inducedgraphofgroups} below already shows
why this might be the case. One way to anticipate the definition is to
imagine how a different choice of fundamental domain for the action of
a group on a tree might yield a different, yet isomorphic 
graph of groups structure on the quotient.

Beginning with \Cref{morphisms}, we develop the covering space theory
of graphs of groups. Our main source is Bass \cite{Bass}. Bass's
ultimate goal appears to be understanding the group of
$\pi_1$-equivariant isomorphisms of the Bass--Serre tree
$T$ in \cite{BassJiang}. Our goal is to develop Stallings foldings
and to understand the subgroups of $\aut(\pi_1)$ or $\out(\pi_1)$
which may be realized as \emph{homotopy equivalences} 
of a given graph of groups.
These goals are at odds with each other, hence our definitions
and treatment diverge somewhat from those of Bass.

\section{Two Notions of Graphs}

For us, a \emph{graph} is a 1-dimensional CW complex.
Its $0$-cells are \emph{vertices,} and its $1$-cells are \emph{edges.}
We think of an edge $e$ as coming with a choice of orientation,
and write $\bar{e}$ for $e$ with its orientation reversed.

There is a convenient combinatorial definition of graphs due to Gersten.
In this definition a graph is a \emph{set} $\Gamma$ with
an involution $e \mapsto \bar e$ on $\Gamma$,
and a retraction $\tau\colon \Gamma \to \Gamma$ onto
the fixed point set $V(\Gamma)$ of the involution.
Thus for all $e \in \Gamma$ we have $\tau(e) \in V(\Gamma)$,
and $\tau^2 = \tau$. In this definition the fixed point set $V(\Gamma)$
is the set of vertices of $\Gamma$,
and its complement $E(\Gamma) = \Gamma - V(\Gamma)$
is the set of oriented edges, with the involution $e \mapsto \bar e$
reversing orientations. The map $\tau$ sends an edge to its terminal vertex.
One recovers the topological definition by constructing a CW complex
with $0$-skeleton $V(\Gamma)$ and a $1$-cell for each orbit $\{e,\bar e\}$ in
$E(\Gamma)$, which is attached to $\{\tau(\bar e),\tau(e)\}$.

The advantage of this definition is that there is an obvious
\emph{category of graphs,}
with objects triples $(\Gamma,e\mapsto\bar e,\tau)$ as above.
Let $\Gamma$ and $\Gamma'$ be graphs.
A \emph{morphism} $f \colon \Gamma \to \Gamma'$ 
is a map of the underlying sets compatible with the involution and retraction
in the following sense. 
Abusing notation
by writing $e \mapsto \bar e$ and $\tau$ for the involution and retraction,
respectively, of both $\Gamma$ and $\Gamma'$,
a map $f\colon \Gamma \to \Gamma'$
is a morphism if the following diagrams commute
\[\begin{tikzcd}
	\Gamma \arrow[r,"f"] \arrow[d,"\tau"] & \Gamma' \arrow[d,"\tau"] \\
	\Gamma \arrow[r,"f"] & \Gamma'
\end{tikzcd}
\qquad\qquad\qquad
\begin{tikzcd}
	\Gamma \arrow[r,"f"] \arrow[d,"\bar{\cdot}"] & \Gamma' \arrow[d,"\bar\cdot"] \\
	\Gamma \arrow[r,"f"] & \Gamma'.
\end{tikzcd}\]
Topologically, morphisms define maps of graphs that send vertices to vertices,
and send edges either homeomorphically to edges or collapse them to vertices.

We will pass back and forth continuously
between the topological and combinatorial notions of graphs.
By an ``edge,'' we will always mean a $1$-cell; we will
write ``oriented edge'' and $E(\Gamma)$ when we need to consider
$e$ to be distinct from $\bar e$.

A simply-connected graph is a \emph{forest,}
and a forest is called a \emph{tree} if it is connected.
We remark that trees are also naturally simplicial complexes.

\section{Graphs of Groups}
Here is the abstract nonsense. A graph $\Gamma$ determines
a \emph{small category} (let us also call it $\Gamma$)
whose objects are the vertices and edges of $\Gamma$,
with morphisms recording the incidence structure of vertices and edges.
Thus if $e$ is an edge, there are morphisms 
$e \to \tau(e)$ and $e \to \tau(\bar e)$. We have $e$ and $\bar e$
equal as objects in the small category $\Gamma$.
A graph of groups structure on $\Gamma$ is a diagram
of groups and monomorphisms in the shape  of $\Gamma$.
That is, it is a functor 
$\mathscr{G}\colon \Gamma \to \operatorname{Grp^{mono}}$
from the small category $\Gamma$ 
to the category of groups with arrows monomorphisms.

Let us be more concrete.

\begin{dfn} A \emph{graph of groups} is a pair
$(\Gamma,\mathscr{G})$, namely a connected graph $\Gamma$, 
together with a diagram
$\mathscr{G}$ of groups and homomorphisms between them.
For each vertex $v$ and edge $e$, $\mathscr{G}$ contains groups 
$G_v$, and $G_e$, which are called \emph{vertex groups}
and \emph{edge groups,} respectively. 
For each edge $e$, there are two monomorphisms,
$\iota_e\colon G_e \to G_{\tau(e)}$
and $\iota_{\bar e}\colon G_e \to  G_{\tau(\bar e)}$.
\end{dfn}

\begin{ex}
	Every graph is naturally a graph of groups by declaring each
	group in $\mathscr{G}$ to be the trivial group.
\end{ex}

Here is the fundamental example
capturing the situation we would like to understand.
Suppose $G$ is a group acting (on the left) by automorphisms on a tree $T$.
We assume further that the action is \emph{without inversions in edges.}
That is, if $e$ is an edge of $T$, then no element of $G$ sends
$e$ to $\bar e$. Topologically this requirement says that 
if an element of $G$ sends an edge to itself,
then it fixes each incident vertex. 
The advantage of this definition is that the graph structure on $T$
naturally defines a graph structure on the quotient $G\backslash T$.
An action of a group $G$ on a tree $T$ by automorphisms always yields an action
on the barycentric subdivision of $T$ without inversions in edges.

\begin{ex}\label{inducedgraphofgroups}
	Suppose $G$ acts on a tree $T$
	without inversions in edges.
	We say $T$ is a \emph{$G$-tree.}
	The quotient $G\backslash T$ is a graph, which will serve as $\Gamma$.
	Choose a maximal subtree (called a spanning tree)
	$T_0 \subset \Gamma$ and a lift
	$\tilde T_0$ that is sent bijectively to $T_0$ under
	the natural projection $\pi\colon T \to G\backslash{T}$.

	We define the graph of groups structure $\mathscr{G}$
	first on $T_0$.  For a vertex $v$ or an edge $e$ of $T_0$, write
	$\tilde v$ or $\tilde e$, respectively,
	for its preimage in $\tilde T_0$.
	Define the groups $G_v$ and $G_e$ to be the stabilizers
	of $\tilde v$ and $\tilde e$ under the action of $G$, respectively.
	If $e$ is an edge of $T_0$, the assumption that $G$ acts on $T$
	without inversions in edges says that the edge group $G_e$
	is naturally a subgroup of the vertex groups
	$G_{\tau(e)}$ and $G_{\tau(\bar e)}$.
	These inclusions define the required monomorphisms.

	Suppose now that $e'$ is an edge of the complement 
	$\Gamma\setminus T_0$.
	Let $\tilde e'$ be a lift of $e'$ to $T$ 
	such that $\tau(\tilde e') \in \tilde T_0$.
	As before, let $G_{e'}$ be the stabilizer of $\tilde e'$,
	which is naturally a subgroup of $G_{\tau(e')}$ as above.
	We now describe the monomorphism $G_{e'} \to G_{\tau(\bar e')}$.
	Let $\tilde w$ be the lift of $\tau(\bar e')$ incident to $e'$.
	Since $T_0$ is a spanning tree of $\Gamma$, $\tilde T_0$
	contains a vertex $\tilde v'$ in the $G$-orbit of $\tilde w$.
	Let $g$ be an element of $G$ with $g.\tilde w = \tilde v'$.
	Observe that if $h$ stabilizes $\tilde e'$, then
	$ghg^{-1}$ stabilizes $\tilde v'$.
	Therefore define the corresponding monomorphism 
	$\iota_{\bar e'}\colon G_{e'} \to G_{\tau(\bar e')}$
	by $\iota_{\bar e'}(h) = ghg^{-1}$.

	This collection of groups and monomorphisms
	defines a graph of groups structure $\mathscr{G}$ on $\Gamma$.
	We say that $\mathscr{G}$ is \emph{induced} by
	the action of $G$ on $T$.
\end{ex}

In the course of the above example, we built a closed, connected subtree 
$F \subset T$ containing $\tilde T_0$ and a lift $\tilde e'$ for each edge $e'$
of $\Gamma\setminus T_0$. The orbit $GF$ is all of $T$, and the
natural projection $\pi\colon T \to G\backslash T$ restricts to an embedding
on the interior of $F$.
Such a subset $F$ is a \emph{fundamental domain}
for the action of $G$ on $T$.
It is compact if $G\backslash T$ is compact.

\begin{rk}
The graph of groups structure defined above is not quite unique.
Different choices of spanning tree, 
for instance, or different choices of the conjugating elements $g$
could change the description slightly while not changing $\Gamma$ or
the isomorphism type of any of the groups $G_v$ and $G_e$.
\end{rk}

\begin{rk}\label{starsofverts}
	The local structure of $T$
	is faithfully reflected in the induced graph of groups structure.
	If $v$ is a vertex of a graph $\Gamma$, define the \emph{star} of $v$,
	written $\st(v)$, to be the set of oriented edges incident to $v$:
	\[ \st(v) = \tau^{-1}(v) = \{ e \in E(\Gamma) : \tau(e) = v\}. \]
	The \emph{valence} of $v$ is the cardinality of $\st(v)$.

	In the above example, if $v$ is a vertex of $\Gamma$
	and $\tilde v \in \pi^{-1}(v)$, the projection induces a surjection
	$\pi_{\tilde v}\colon \st(\tilde v) \to \st(v)$.
	The stabilizer of $\tilde v$ is conjugate to $G_v$ as a subgroup of $G$,
	so there is an action of $G_v$ on $\st(\tilde v)$ which respects
	$\pi_{\tilde v}$. 
	The action is well-defined up to an inner automorphism of $G_v$.
	If $e$ is an oriented edge of $\st(v)$,
	the action of $G_v$ on $\st(\tilde v)$ puts the fiber $\pi^{-1}(e)$
	in $G_v$-equivariant bijection with the set of cosets 
	$G_v/\iota_e(G_e)$.
\end{rk}

The following theorem states that this latter observation allows
one to rebuild both $T$ and $G$ up to equivariant isomorphism from
the graph of groups structure $(\Gamma,\mathscr{G})$.

\begin{thm}[I.5.4, Theorem 13~\cite{Trees}]\label{BassSerreThm}
	If $(\Gamma,\mathscr{G})$ is a graph of groups, then there
	exists a tree $T$ and a group $G$ acting on $T$
	without inversions in edges
	such that $G\backslash T \cong \Gamma$, and the action of
	$G$ on $T$ induces the graph of groups structure 
	$\mathscr{G}$ on $\Gamma$.
	If $G'$ and $T'$ are another such group and tree inducing
	the graph of groups $(\Gamma,\mathscr{G})$,
	then $G\cong G'$, and there is a $G$-equivariant isomorphism $T \to T'$.
\end{thm}

The proof of \Cref{BassSerreThm} 
occupies the next several sections until \Cref{morphisms}. 
We will adopt the notation of the theorem statement
until the completion of the proof.

\section{Construction of the Bass--Serre Tree}\label{construction}

In the construction we will give, the group $G$ is called
the \emph{fundamental group of the graph of groups}
$\pi_1(\Gamma,\mathscr{G})$,
for reasons that will become clear---in particular
there is a choice of basepoint we are temporarily suppressing. 
The tree $T$ is called the \emph{universal cover} 
of the graph of groups $(\Gamma,\mathscr{G})$,
although more commonly it is called the \emph{Bass--Serre tree.}

\paragraph{The case of $\Gamma$ a tree.}
Suppose $(\Gamma,\mathscr{G})$ is a graph of groups where $\Gamma$ is a tree.
We will build the tree $T$ inductively from copies of $\Gamma$,
which we will call \emph{puzzle pieces.}
Each puzzle piece has a natural simplicial homeomorphism
onto $\Gamma$, and the projection $\pi\colon T \to \Gamma$
will be assembled from these homeomorphisms.

We can build $T$ inductively. Begin with $T$ 
a single puzzle piece and $\pi\colon T \to \Gamma$ a homeomorphism.
We would like to make the observation in \Cref{starsofverts} hold.
Namely, given a vertex $v \in \Gamma$ and a vertex $\tilde v$ in the fiber
$\pi^{-1}(v)$, the projection $\pi\colon T \to \Gamma$ induces a map
$\pi_{\tilde v}\colon\st(\tilde v)\to \st(v)$, which we would like to
induce a $G_v$-equivariant bijection
\[ \st(\tilde v) \cong \coprod_{e \in \st(v)} (G_v/\iota_e(G_e)\times \{e\}). \]

For the inductive step, suppose the above statement does not hold. 
If $G_e$ were trivial for each edge $e \in \st(v)$, we could 
form a tree $T_{\tilde v}$ as the one-point union
\[ T_{\tilde v} = (G_v\times \Gamma) / (G_v \times \{v\}), \] 
and then form a new tree $T'$ from $T\cup T_{\tilde v}$
by identifying $\{1\}\times \Gamma$ in $T_{\tilde v}$
with the (unique, by induction) puzzle piece containing $\tilde v$ in $T$.
The projection of $T_{\tilde v}$ onto the $\Gamma$ factor
allows us to extend the projection $\pi\colon T \to \Gamma$
to a projection $T' \to \Gamma$.
The induction continues on $T'$.

If some $G_e$ is nontrivial, the outline is the same,
but a more sophisticated construction of the tree $T_{\tilde v}$
is required. Observe that since $\Gamma$ is a tree, each edge in $\st(v)$
determines a unique component of $\Gamma\setminus\{v\}$, whose closure is a
subtree denoted $\Gamma_e \subset \Gamma$. The vertex $v$
is a leaf of $\Gamma_e$.
The tree $T_{\tilde v}$ in this case is formed from the disjoint union
\[ \coprod_{e\in\st(v)} (G_v/\iota_e(G_e))\times \Gamma_e \] 
by identifying $(g\iota_e(G_e),v)$ with $(h\iota_{e'}(G_{e'}),v)$
for each pair of group elements $g$ and $h \in G_v$,
and each pair of oriented edge $e$ and $e'\in \st(v)$. 
Note that for each $g \in G_v$, the image in $T_{\tilde v}$ of the disjoint union
\[ \coprod_{e\in\st(v)} \{g\iota_e(G_e)\}\times \Gamma_e \]
is homeomorphic to $\Gamma$. 
These are the puzzle pieces contained in $T_{\tilde v}$.
The action of $G_v$ on the disjoint union by permuting the labels
descends to an action on $T_{\tilde v}$ permuting the puzzle pieces,
which may intersect.
The inductive step completes as in the previous paragraph.

Continue this process inductively, perhaps transfinitely, and pass
to the direct limit of the trees constructed. The limiting tree is 
the \emph{Bass--Serre tree} $T$. Each puzzle piece in $T$
is sent homeomorphically to $\Gamma$ under the projection $\pi\colon T \to \Gamma$.

\paragraph{The general case.} If we drop the assumption that $\Gamma$
is a tree we can build $T$ in two stages. In the first stage, we build
the universal cover $\tilde\Gamma$ of $\Gamma$ as an ordinary graph.
Let $p\colon\tilde\Gamma \to \Gamma$ be the covering map.
We define a graph of groups structure $\tilde{\mathscr{G}}$ on $\tilde\Gamma$.
If $\tilde v$ is a vertex of $\tilde\Gamma$,
set $G_{\tilde v} \cong G_{p(\tilde v)}$,
and do the same for edges.
A monomorphism $G_e \to G_{\tau(e)}$ induces a monomorphism
$G_{\tilde e}\to G_{\tau(\tilde e)}$ for each edge $\tilde e \in p^{-1}(e)$;
these will be the monomorphisms of $\mathscr{\tilde G}$.

Now follow the preceding construction to produce the Bass--Serre tree $T$
and projection $p'\colon T \to \tilde \Gamma$
for the graph of groups $(\tilde\Gamma,\mathscr{\tilde G})$.
The tree $T$ will also be the Bass--Serre tree for 
$(\Gamma,\mathscr{G})$, and the projection
$\pi$ will be $p'p\colon T \to \Gamma$.

\section{Fundamental Group of a Graph of Groups}

To construct the group $G$ acting on $T$, we introduce
$\pi_1(\Gamma,\mathscr{G},p)$, \emph{the fundamental group 
of the graph of groups based at the point $p\in \Gamma$,}
and define an action of $\pi_1(\Gamma,\mathscr{G},p)$ on $T$
inducing the graph of groups structure.

\paragraph{Paths and edge paths.}
Since the tree $T$ is simply connected,
any two maps $\gamma,\gamma'\colon [0,1] \to T$ are homotopic
rel endpoints if and only if their endpoints are equal.
It will be more convenient to work with paths that are normalized
in the following way.

By a \emph{path} in $T$ we mean an \emph{edge path,}
a finite sequence $e_1\dotsb e_k$
of oriented edges such that
$\tau(e_i) = \tau(\bar  e_{i+1})$ for  $1\le i \le k-1$.
More generally, we allow paths to be points (i.e. $k = 0$),
or for $e_1$ and $e_k$ to be segments---connected, closed
subsets of edges, which must contain the appropriate vertex if $k > 1$.
For any path there is a map $\gamma\colon[0,1]\to T$
which is an embedding on each edge or segment of an edge,
and whose image is the specified edge path.
Since every map $[0,1]\to T$ is homotopic rel endpoints
to such a path $\gamma$, we will work only with edge paths.
A path is \emph{tight} if it is a point,
or if $\bar{e_i} \neq e_{i+1}$ for $1 \le i\le k$.
Paths  may be \emph{tightened} by a homotopy rel endpoints
to produce tight paths by deleting subpaths of the form
$\sigma\bar\sigma$. If $\gamma$ is a path,
let $\bar\gamma$ denote  $\gamma$ with its orientation reversed.

Let $\tilde\gamma$ and $\tilde\gamma'$ be paths in $T$.
We would like a notion of paths 
$\gamma$ and $\gamma'$ in $(\Gamma,\mathscr{G})$
and a notion of homotopy
such that $\gamma$ and $\gamma'$ are homotopic rel endpoints
only if their lifts $\tilde\gamma$ and $\tilde\gamma'$ are.

Recall from our construction that $T$ is built from
puzzle pieces, which are copies of $\Gamma$ or its universal cover.
Vertices $\tilde v \in T$ come equipped with actions of
their associated vertex groups $G_{\pi(\tilde v)}$ permuting
the set of puzzle pieces containing $\tilde v$.
Whenever a path $\tilde\gamma$ meets a lift of a vertex
with nontrivial vertex group,
it has the opportunity to cross into a new puzzle piece.
That is, suppose $ee'$ is a path with $e$ and $e'$ contained
in the puzzle pieces $P$ and $P'$, respectively.
If $v$ is the common vertex of $e$ and $e'$,
then there is a group element $g \in G_{\pi(v)}$
taking $P$ to $P'$. The group element $g$ is
uniquely determined if and only if
$G_{\pi(v)}$ acts effectively on the set of fundamental domains at $v$.

Thus we define a \emph{path} $\gamma$ in $(\Gamma,\mathscr{G})$
to be a finite sequence 
\[ \gamma = e_1'g_1e_2g_2\dotsb g_{k-1}e_k',\]
where $e_1'e_2\dotsb e_{k-1}e_k'$ is an edge path in $\Gamma$,
denoted $|\gamma|$,
and for each $i$ satisfying $1 \le i \le k-1$
we have $g_i \in G_{\tau(e_i)} = G_{\tau(\bar e_{i+1})}$. 
We allow the terminal and initial segments $e_1'$ and $e_k'$ to be empty,
in which case they should be dropped from the notation.
To reverse the orientation of $\gamma$, replace it by
\[ \bar\gamma = \bar e_k'g_{k-1}^{-1}\dotsb g_2^{-1}\bar e_2 g_1^{-1}\bar e_1'. \]
A  path $\gamma$ is a \emph{loop} if $|\gamma|$ is a loop in $\Gamma$,
i.e.\ if its endpoints are equal. A loop is
\emph{based} at its endpoint.

\paragraph{Lifting and homotopy.}
Let $\gamma= e_1'g_1e_2g_2\dotsb g_{k-1}e_k'$
be a path in $(\Gamma,\mathscr{G})$
with initial endpoint $p$.
For any lift of the initial endpoint
$\tilde p \in \pi^{-1}(p)$ and puzzle piece $P$ containing $\tilde p$,
we can construct a path $\tilde\gamma\colon [0,1]\to T$ with
$\pi\tilde\gamma = |\gamma|$. We proceed inductively.
Let $\tilde e_1'$ be the lift of $e_1'$ contained in $P$.
Write $P_1 = P$ and $v_1 = \tau(\tilde e_1')$.
Inductively, let $P_{i+1} = g_{i}.P_i$ under the action of
$G_{\pi(v_i)}$ on the set of puzzle pieces in $T$ containing $v_i$.
Then $\tilde e_{i+1}$ is the lift of $e_{i+1}$ contained in $P_{i+1}$,
and $v_{i+1} = \tau( e_{i+1})$.  In the end we construct a path
\[ \tilde \gamma = \tilde e_1'\tilde e_2\dotsb \tilde e_{k-1}\tilde e_k'. \]
Write $P(\tilde\gamma) = P_k$. This is the \emph{ending piece} of $\gamma$,
or more properly its lift at $(\tilde p,P)$.

It is not hard to see that the following operations on $\gamma$
do not change the homotopy class of its lift to $T$:
\begin{enumerate}
	\item Inserting or removing subpaths of the form $\sigma\bar\sigma$
	\item For $h \in G_e$, $g \in G_{\tau(\bar e)}$ and 
		$g' \in G_{\tau(e)}$, replacing subpaths of the form
		$ge\iota_{e}(h)g'$ with  $g\iota_{\bar e}(h) eg'$.
\end{enumerate}
Two  paths $\gamma$  and $\eta$ in $(\Gamma,\mathscr{G})$ are
\emph{homotopic rel endpoints} 
if $\gamma$ can be transformed into $\eta$
in a finite number of moves as above\footnote{
	Following Haefliger~\cite[III.$\mathscr{G}$.3.2]{theBible},
	it might be better to think of the second operation as
	\emph{equivalence} and reserve \emph{homotopy} for the first.
	We have no obvious need for such a distinction, so will not draw one.
}.
The homotopy class of $\gamma$ rel endpoints is denoted $[\gamma]$,
and if $\gamma$ and $\eta$ are homotopic rel endpoints, we write
$\gamma\simeq \eta$.

\begin{dfn} 
	Let $(\Gamma,\mathscr{G})$ be a graph of groups, and $p$
	a point of $\Gamma$. The \emph{fundamental group of 
	the graph of groups} $\pi_1(\Gamma,\mathscr{G},p)$
	is the group of homotopy classes of loops in $(\Gamma,\mathscr{G})$
	based at $p$
	under the operation of concatenation.
\end{dfn}

It is clear, as in the case of the ordinary fundamental group,
that if $\gamma$ is a path in $(\Gamma,\mathscr{G})$ from $p$ to $q$,
then the map $\pi_1(\Gamma,\mathscr{G},p) \to \pi_1(\Gamma,\mathscr{G},q)$
defined by $[\sigma] \mapsto [\bar\gamma\sigma\gamma]$ is an isomorphism.
If $p = q$, then $\gamma$ is a loop,
and this isomorphism is an inner automorphism.

\begin{lem}\label{action}
	Given $p \in \Gamma$, a choice of lift $\tilde p\in \pi^{-1}(p)$
	and puzzle piece $P$ containing $\tilde p$
	determine an action of $\pi_1(\Gamma,\mathscr{G},p)$ on $T$
	by automorphisms of the projection $\pi\colon T \to \Gamma$.
	In particular, the action is without inversions in edges.
	Different choices of $(\tilde p,P)$ alter the action
	by an inner automorphism of $\pi_1(\Gamma,\mathscr{G},p)$.
\end{lem}

\begin{proof}
	Given $\tilde p$ and $P$ as in the statement and
	$g \in \pi_1(\Gamma,\mathscr{G},p)$, we define an automorphism
	of the projection and show that these automorphisms define
	an action of $\pi_1(\Gamma,\mathscr{G},p)$ on $T$.
	Suppose $g$ is represented by a loop $\gamma \in (\Gamma,\mathscr{G})$
	based at $p$.
	We constructed a lift $\tilde \gamma$ of $\gamma$ to $T$ 
	corresponding to the choice
	$(\tilde p,P)$. Associated to $\tilde\gamma$ was its
	ending piece $P(\tilde\gamma)$ containing $\tilde\gamma(1)$. 
	Both the terminal endpoint
	$\tilde\gamma(1)$ and ending piece 
	depend only on the homotopy class $g$ and not on 
	the choice of representative path $\gamma$.

	Path lifting respects composition in $\pi_1(\Gamma,\mathscr{G},p)$
	in the following sense. If $\gamma$ and $\eta$
	are loops representing elements of $\pi_1(\Gamma,\mathscr{G},p)$,
	then we see that the path $\tilde\gamma\tilde\eta$
	is a lift of the concatenation 
	$[\gamma\eta] \in \pi_1(\Gamma,\mathscr{G},p)$,
	where $\tilde\gamma\tilde\eta$ is
	obtained by first lifting $\gamma$ at $(\tilde p,P)$
	to $\tilde\gamma$, then lifting
	$\eta$ at $(\tilde\gamma(1),P(\tilde\gamma))$ 
	to $\tilde \eta$ and $P(\tilde\gamma\tilde\eta)$ and finally concatenating.

	This lifting allows us to construct an 
	automorphism of the projection $\pi\colon T \to \Gamma$
	corresponding to $g\in\pi_1(\Gamma,\mathscr{G},p)$.
	Namely, define $g.\tilde p$ and $g.P$ to be
	$\tilde\gamma(1)$ and $P(\tilde\gamma)$
	coming from lifting some (and hence any)
	loop $\gamma$ representing $g$ at $(\tilde p,P)$.
	
	More generally, if $x \in T$, there is a path
	$\tilde\sigma$ in $T$ from $\tilde p$ to $x$.
	There is a path $\sigma$ which lifts at $(\tilde p,P)$ to $\tilde\sigma$.
	Take $g.\tilde\sigma$ to be the lift of
	$\sigma$ at $(g.\tilde p,g.P)$, and $g.x$ to be $g.\tilde\sigma(1)$.
	It is clear that $\pi(g.x) = \pi(x)$, 
	and not hard to see that $x \mapsto g.x$
	is a graph automorphism.
	The argument in the previous paragraph shows that this defines an action of
	$\pi_1(\Gamma,\mathscr{G},p)$ on $T$.
	
	If $(\tilde p',P')$ are another choice of lift and fundamental domain,
	there is a loop $\gamma$ based at $p$ whose lift at $(\tilde p,P)$
	terminates at $(\tilde p',P')$.
	The loop $\gamma$ determines an element $g \in \pi_1(\Gamma,\mathscr{G},p)$,
	and the conjugation $h \mapsto ghg^{-1}$
	is  the required inner automorphism.
\end{proof}

Observe that the action of $\pi_1(\Gamma,\mathscr{G},p)$
on $T$ sends puzzle pieces to puzzle pieces,
and furthermore that the action of an element $g$
is determined by the data $(g.\tilde p, g.P)$.

\begin{rk}
It should be noted that homotopy for edge paths 
in $(\Gamma,\mathscr{G})$ is possibly a finer invariant
than homotopy of their lifts to $T$. 
This is the case when the action of $\pi_1(\Gamma,\mathscr{G},p)$ on $T$
is not effective. We will describe the kernel below.
\end{rk}

\begin{cor}\label{fundamentaldomain}
	The action of $\pi_1(\Gamma,\mathscr{G},p)$ on $T$
	induces a homeomorphism 
	\[\pi_1(\Gamma,\mathscr{G},p)\backslash T \cong \Gamma. \]
	The action admits a fundamental domain.
\end{cor}

\begin{proof}
	In the proof of the previous lemma,
	we saw that $\pi_1(\Gamma,\mathscr{G},p)$
	acts transitively on the set of pairs
	$(\tilde q,Q)$, where $q\in\pi^{-1}(p)$, and
	$Q$ is a puzzle piece containing $q$.
	Since the action is by automorphisms of the
	projection $\pi\colon T \to \Gamma$, this implies
	that the quotient is homeomorphic to $\Gamma$.

	Recall that the puzzle piece $P$ is the
	universal cover of $\Gamma$ as an ordinary graph.
	The action of the ordinary fundamental group of $\Gamma$
	(based at $p$) on $P$ admits a fundamental domain
	$F$. If we view $F$ as a connected subtree of $P \subset T$,
	we see that $F$ is a fundamental domain for
	the action of $\pi_1(\Gamma,\mathscr{G},p)$.
	Indeed, $F$ contains one edge for each edge of $\Gamma$,
	so the restriction of $\pi\colon T \to \Gamma$
	is an embedding on the interior of $F$.
	It follows from the above that the translates of $F$
	cover $T$.
\end{proof}

\section{A Presentation of the Fundamental Group}
From our description of homotopy of loops, we can give a presentation
for $\pi_1(\Gamma,\mathscr{G},p)$. Let $T_0\subset \Gamma$ be a
spanning tree
of $\Gamma$ containing the basepoint $p$.

\begin{thm}\label{presentation}
The group $\pi_1(\Gamma,\mathscr{G},p)$ is isomorphic to a quotient of
the free product of the groups
$G_v$ for each vertex $v \in \Gamma$ with the free group on the set
$E(\Gamma)$ of \emph{oriented} edges in $\Gamma$ 
\[ \left(*_{v\in V(\Gamma)}G_v\right)*F(E(\Gamma))/R \]
Where $R$ is the normal subgroup imposing the following relations.
\begin{enumerate}
	\item $\iota_{\bar e}(h)e = e\iota_{e}(h)$,
		where $e$ is an oriented edge of $\Gamma$,
		and $h \in G_e$.
	\item $\bar{e} = e^{-1}$, where $e$ is an oriented edge of $\Gamma$.
	\item $e = 1$ if $e$ is an oriented edge of $T_0$.
\end{enumerate}
\end{thm}

We remark that the theorem implies that each $G_v$ is naturally a subgroup
of $\pi_1(\Gamma,\mathscr{G},p)$.
It follows that different choices of spanning tree $T_0$
yield differing but isomorphic presentations.

\begin{proof}
	Write $\pi(\mathscr{G})$ for the free product in the statement,
	\[
		\pi(\mathscr{G}) = \left(*_{v\in V(\Gamma)}G_v\right)*F(E(\Gamma)).
	\]
	The group $\pi(\mathscr{G})/R'$, where $R'$ imposes the first
	two relations in the statement is called the \emph{path group}
	and fills a convenient role in other expositions of the theory.
	Write $G = \pi(\mathscr{G})/R$ for the group 
	described in the statement of the theorem.
	We will define a map $f$ sending words in $\pi(\mathscr{G})$
	to paths in $(\Gamma,\mathscr{G})$.
	Let $\gamma_v$ denote the unique tight path in $\Gamma$
	(as an ordinary graph)
	from $p$ to $v$ contained in $T_0$.
	We define
	\[ f(g) = \gamma_vg\bar\gamma_v 
		\quad\text{for }g \in G_v,\quad\text{and}\quad
	f(e) = \gamma_{\tau(\bar e)}e\gamma_{\tau(e)} 
	\quad\text{for }e \in E(\Gamma).\]
	Then define $f$ on the word $x_1\dotsb x_n$ by $f(x_1)\dotsb f(x_n)$.
	One checks that the following statements hold.
	\begin{enumerate}
		\item If $w$ and $w'$ represent the same element 
			of $\pi(\mathscr{G})$,
			then $f(w)\simeq f(w')$.
		\item $f(\iota_{\bar e}(h)e)\simeq f(e\iota_{e}(h))$,
			where $e \in E(\Gamma)$ and $h \in G_e$.
		\item $f(\bar e) = \overline{f(e)}$.
		\item $f(e)$ is null-homotopic (i.e. homotopic to the constant path
			at $p$) if $e$ is an oriented edge belonging to $T_0$.
	\end{enumerate}

	The map $f$ defines a homomorphism 
	$f_\sharp\colon G \to \pi_1(\Gamma,\mathscr{G},p)$.
	If $g \in G$ is represented by a word $w$ in $\pi(\mathscr{G})$,
	the homotopy class of $f(w)$ depends only on $g$, so
	define $f_\sharp(g) = [f(w)]$.

	We claim that $f_\sharp$ is an isomorphism.
	First observe that if $\gamma  = e_1'g_1e_2\dotsb g_{k-1}e_k'$
	is an edge path loop based at $p$, then we have that
	\[ f(g_1e_2\dotsb e_{k-1}g_{k-1}) \simeq \gamma, \]
	so $f_\sharp$ is surjective.
	On the other hand, suppose $g$ and $g'$ are elements of $G$
	with $f_\sharp(g) = f_\sharp(g')$. If we represent $g$ and $g'$
	by words $w$ and $w'$ in $\pi(\mathscr{G})$ respectively, then
	\[ f(w) \simeq f(w') \]
	by assumption. But observe that the edge path loop $f(w)$
	may be thought of as a word in $\pi(\mathscr{G})$ representing $g$,
	since each $\gamma_v$ (perhaps dropping the initial segment)
	is a sequence of elements
	of $E(\Gamma)$ whose image in $G$ is trivial.
	Performing a homotopy from $f(w)$ to $f(w')$ amounts to 
	applying a finite number of defining relations in $G$.
	Thus the images of $f(w)$ and $f(w')$ as words in $\pi(\mathscr{G})$
	are equal in $G$, implying that $g = g'$, so $f_\sharp$ is injective.
\end{proof}

We continue to write $G$ for the group 
described in the statement of \Cref{presentation}.
The injection of $G_v$ into the free product, combined with
the isomorphism in the proof exhibits monomorphisms
$f_v\colon G_v\to \pi_1(\Gamma,\mathscr{G},p)$, defined as
$f_v(g) = [\gamma_vg\bar\gamma_v]$.

\begin{cor}\label{ordinary}
	Let $\gamma$ be an edge loop 
	representing an element of $\pi_1(\Gamma,\mathscr{G},p)$.
	The forgetful map $\gamma \mapsto |\gamma|$ induces a surjection
	$|\cdot|\colon\pi_1(\Gamma,\mathscr{G},p) \to \pi_1(\Gamma,p)$
	onto the ordinary fundamental group of $\Gamma$.
	The map $|\cdot|$ admits a section.

	The kernel of the map $|\cdot|$ is the subgroup normally
	generated by $f_v(G_v)$ for each vertex $v \in \Gamma$.
\end{cor}
\begin{proof}
	Consider an element $g$ of the ordinary fundamental group
	of $\Gamma$ based at $p$. Represent $g$ by an edge path
	$\gamma = e_1'e_2\dotsb e_{k-1}e_k'$. Inserting the identity
	element of every vertex group present in $\gamma$,
	we may view $\gamma$ as a loop in $(\Gamma,\mathscr{G})$.
	The homotopy class of $\gamma$ in $\pi_1(\Gamma,\mathscr{G},p)$
	is independent of the choice of $\gamma$.
	This shows that the map $g \mapsto [\gamma]$ defines a section 
	of $|\cdot|$, which is thus surjective.

	It is clear from the proof of \Cref{presentation}
	that each $f_v(G_v)$ is contained in the kernel of $|\cdot|$.
	Conversely, consider an element
	$h\in\pi_1(\Gamma,\mathscr{G},p)$ in the
	kernel of $|\cdot|$. Let $\eta = e_1'g_1e_2\dotsb g_{k-1}e_k'$ 
	be an edge path loop representing $h$.
	We will use the map $f$ from the proof of \Cref{presentation}.
	Write
	\[	\eta' = f(g_1)f(e_2)\dotsb f(g_{k-1}). \]
	As we saw, $\eta' \simeq \eta$. For $2 \le i \le k-1$, write
	$\gamma_i = f(e_2\dotsb e_i)$, and observe that
	\[	\eta' \simeq f(g_1)\gamma_2f(g_2)\bar\gamma_2\dotsb
	\gamma_{k-2}f(g_{k-2})\bar\gamma_{k-2}\gamma_{k-1}f(g_{k-1}).  \]
	Observe that the path $\gamma_{k-1}$ satisfies $\gamma_{k-1}$ 
	satisfies $\gamma$ satisfies $\gamma$ satisfies $\gamma \simeq |\eta|$
	in $(\Gamma,\mathscr{G})$,
	and is thus null-homotopic. Removing it from the above concatenation
	via a homotopy
	exhibits $h$ as an element of the subgroup normally generated by
	the subgroups $f_v(G_v)$.
\end{proof}

In the terms of the surjection $\pi(\mathscr{G}) \to G$, 
the content of the above corollary is that the subgroup
of $G$ containing the image of the free group on $E(\Gamma)$
is isomorphic to the ordinary fundamental group $\pi_1(\Gamma,p)$,
and there is a retraction from $G$ to this subgroup sending
each $G_v$ to $1$.

\section{Proof of the Main Theorem}
We collect a few lemmas needed to complete the proof of
\Cref{BassSerreThm}. Fixing a choice $(\tilde p,P)$ in $T$
and a choice of spanning tree $T_0\subset \Gamma$ containing $p$,
using the isomorphism $G\cong \pi_1(\Gamma,\mathscr{G},p)$
from \Cref{presentation},
we will work with the action of $G$ on $T$.
By \Cref{fundamentaldomain}, there is a fundamental domain 
$F_0 \subset P\subset T$ containing $\tilde p$,
and a subtree $\tilde T_0\subset F_0$ mapped homeomorphically
to $T_0$ under the projection $\pi\colon T \to \Gamma$.
If $v$ is a vertex of $\Gamma$, write $\tilde v$ for the lift of $v$ contained
in $\tilde T_0$. If $e$ is an edge of $\Gamma$, write $\tilde e$ for the
lift of $e$ contained in $F_0$.
We will frequently confuse edges of $\Gamma$ and vertex groups $G_v$
with their images in $G$.

\paragraph{Description of the action of $G$ on $T$.}
If $g \in G_v$, recall that the image of $g$ in $G$ corresponds
to the homotopy class of $\gamma_vg\bar\gamma_v$, where
$\gamma_v$ is the unique tight path in $T_0$ (as an ordinary tree)
from $p$ to $v$. It follows from the description of the action
of $\pi_1(\Gamma,\mathscr{G},p)$ in \Cref{action} that
$G_v \le G$ stabilizes $\tilde v$ and its action on the set
of puzzle pieces containing $\tilde v$ agrees with the
action described in the construction of $T$.

If $e$ is an edge of $\Gamma$ not contained in $T_0$
orient $e$ so that $\tau(\tilde e)$ does not lie
in $\tilde T_0$. Write $v = \tau(e)$.
The action of $e$ (as an element of $G$) on $T$
preserves $P$ and sends $\tilde v$ to $\tau(\tilde e)$.

\begin{lem} The kernel of the action of $G$ on $T$
	is the largest normal subgroup contained in the intersection
	\[ 	\bigcap_{e\in E(\Gamma)} \iota_e(G_e). \]
\end{lem}

\begin{proof}
	The kernel of the action is clearly normal. An element
	$h \in G$ fixes the tree $T$ pointwise if and only if it sends
	the pair $(\tilde p,P)$ to itself. 
	If $\eta$ is a loop representing $h$,
	it follows that $\tilde\eta$, the lift of $\eta$
	at $(\tilde p,P)$ is null-homotopic,
	and thus contains a subpath of the form
	$\tilde\sigma\bar\tilde\sigma$. Such a subpath is a \emph{backtrack.}
	In $\eta$, a backtrack of $\tilde\eta$ corresponds to a subpath
	of the form $\sigma g\bar \sigma$.
	If $e$ is the last edge appearing in the edge path
	decomposition of $\sigma$, it follows that 
	$g$ belongs to the subgroup $\iota_e(G_e)$ of $G_{\tau(e)}$,
	say $g = \iota_e(g')$.
	Perform a homotopy replacing $eg\simeq  \iota_{\bar e}(g')e$
	and removing  $e\bar e$.
	By iterating the above argument, we see that we
	may shorten any backtrack contained in $\eta$
	that contains a full edge.

	Thus we have that
	\[ \eta \simeq e_1'g'\bar e_1'\text{, for some }g'\in G_{\tau(e_1)}, \]
	where $e_1'$ is a (possibly empty) segment of an edge
	containing $p$. This shows that $h$ belongs to $G_{\tau(e_1)}$.
	By assumption, $g'$ sends $P$ to itself, 
	so we have that
	\[	g' \in \bigcap_{e \in \st(\tau(e_1))} \iota_e(G_e).	\]
	Given an oriented edge $e \in \st(\tau(e_1))$, if we write $g' = \iota_e(k')$,
	we have
	\[	\eta\simeq e_1'e\iota_e(k')\bar e\bar e_1', \]
	so $h$ belongs to $G_{\tau(e)}$, and a similar argument
	to the above shows that for any vertex $v$,
	\[	h \in \bigcap_{e\in\st(v)} \iota_e(G_e), \]
	so we conclude that $h$ belongs to $\iota_e(G_e)$ for all
	$e \in E(\Gamma)$.

	Now suppose $N$ is a normal subgroup of $G$
	satisfying $N \le \iota_e(G_e)$ for all $e\in E(\Gamma)$.
	In particular, if $v$ is a vertex, we have
	\[	N \le \bigcap_{e\in\st(v)} \iota_e(G_e).	\]
	Thus if $h \in N \le G_v$, it follows from the action of
	$G_v$ on the set of puzzle pieces containing $\tilde v$
	that $h$ sends $P$ to itself. Furthermore, $N$ is contained
	within the kernel of $G \to \pi_1(\Gamma,p)$ by \Cref{ordinary},
	so it follows that $N$ also stabilizes $\tilde p$.
	Thus $N$ acts trivially on $T$.
\end{proof}

\begin{cor}
	The action of $G$ on $T$ induces the graph of groups structure
	$(\Gamma,\mathscr{G})$.
\end{cor}

\begin{proof}
	We claim that $G_v$ is the stabilizer of $\tilde v$. Indeed,
	that $G_v$ stabilizes $\tilde v$ is clear. If $h$ stabilizes
	$\tilde v$, its action on the set of puzzle pieces 
	containing $v$ agrees with some $g \in G_v$.
	Thus $g^{-1}h$ sends  $(\tilde p,P)$
	to itself, and the claim follows from the previous lemma.

	It follows that if $e$  is an edge of $\Gamma$
	with $\tau(\tilde e) \in \tilde T_0$, then $\iota_e(G_e)$
	is the stabilizer of $\tilde e$.
	If $\tau(\tilde e)$ is not contained in $\tilde T_0$,
	then $\iota_{\bar e}(G_e) = e\iota_e(G_e)e^{-1}$ is the stabilizer
	of $\tilde e$.
\end{proof}

\begin{rk}
	The tree $T$ admits an alternate description: by the orbit-stabilizer
	theorem, vertices of $T$ lifting $v$ 
	correspond to left cosets of $G_v$ in $G$.
	Edges lifting $e$ correspond to cosets of 
	$\iota_{\bar e}(G_e)$ in $G$, and the edge corresponding
	to $ge\iota_{\bar e}(G_e)$ has initial vertex 
	corresponding to the coset $gG_{\tau(\bar e)}$
	and terminal vertex corresponding to $geG_{\tau(e)}$.
	Fix an orientation for the edges of $\Gamma$ such that
	$\tau(\tilde e)$ is not in $\tilde T_0$ when
	$e$ is not in $\tilde T_0$. Combinatorially this is
	a subset $O\subset E(\Gamma)$ containing one element
	of each orbit of the involution $e \mapsto \bar e$.
	The tree $T$ is the quotient of the disjoint union
	\[	\coprod_{e\in O} G/\iota_{\bar e}(G_e)\times \{e\}\]
	by identifying vertices as above. 
	This is the construction of $T$ given in \cite{Trees}.
\end{rk}

\begin{proof}[Proof of \Cref{BassSerreThm}.]
	Much of the statement
	of the theorem follows from the previous lemmas.
	To complete the proof, suppose that we have a group
	$G'$ acting on a tree $T'$ without inversions such that
	$G'\backslash T'$ is homeomorphic to $\Gamma$, and that the
	induced graph of groups structure is isomorphic to
	$(\Gamma,\mathscr{G})$ in the following sense.

	By assumption, there is a fundamental domain 
	$F'_0\subset T'$ for the action of $G'$ on $T'$
	such that $F'_0$ and $F_0$ are homeomorphic via a homeomorphism
	such that the following diagram commutes
	\[\begin{tikzcd}
		F_0 \ar[r,"\cong"]\ar[d] & F'_0 \ar[d] \\
		\Gamma \ar[r,"\cong"] & G'\backslash T',
	\end{tikzcd}\]
	where the vertical maps are the restriction of the natural projections.
	Similarly there is a lift $\tilde T'_0$  of $T_0$.
	For each vertex $v$ and edge $e$,
	write $\tilde v'$ for the lift of $v$ in $\tilde T'_0$
	and $\tilde e'$ for the lift of $e$ in $F'_0$,
	and let $G'_{\tilde v'}$ and $G'_{\tilde e'}$ denote
	their stabilizers in $\tilde G'$, respectively.
	By assumption, there are isomorphisms $f_v\colon G_v \to G'_{\tilde v'}$.
	If $\tau(\tilde e') \in \tilde T'_0$, there is an isomorphism
	$f_e \colon G_e \to G'_{\tilde e'}$ such that the following diagram
	commutes
	\[ \begin{tikzcd}[remember picture]
		G_e \arrow[r,"f_e"]\arrow[d,"\iota_e"] & G'_{\tilde e'} \\
		G_v \arrow[r,"f_v"] & G'_{\tilde v'}.
	\end{tikzcd}
	\begin{tikzpicture}[overlay,remember picture]
	\path (\tikzcdmatrixname-1-2) to node[midway,sloped]{$\subset$}
	(\tikzcdmatrixname-2-2);
	\end{tikzpicture}
	\]

	If $\tau(\tilde e')$ is \emph{not} in $\tilde T'_0$
	and $v = \tau(e)$,
	then by assumption there exists $g_e \in G'$
	with $g_e.\tilde v' = \tau(\tilde e')$, such that
	for $h \in G_e$, we have $g_e^{-1}f_e(h)g_e = f_v(\iota_e(h))$.
	Define $g_{\bar e} = g_e^{-1}$.
	
	The maps $f_v \colon G_v \to G'_{\tilde v'} \subset G'$,
	$e \mapsto 1$ if $\tilde e' \in \tilde T'_0$ and
	$e \mapsto g_e$ otherwise define a homomorphism
	\[ \hat f\colon \left(*_{v\in V(\Gamma)}G_v\right)*F(E(\Gamma)) \to G'. \]
	One checks that the relations in the statement of \Cref{presentation}
	are satisfied in $G'$, so $\hat f$ induces a homomorphism
	$f\colon G \to G'$.

	We will show that $f$ is an isomorphism.
	Since $G'$ admits $F'_0$ as a fundamental domain,
	a classical result of Macbeath \cite{Macbeath}
	says that $G'$ is generated by the set
	\[ S = \{g\in G' \colon gF'_0 \cap F'_0 \neq \varnothing\}. \]
	If $g \in S$, then either $g$ stabilizes 
	some vertex of $\tilde T_0'$, or
	there is an edge $e$ in $\Gamma$ such that $g_e'^{-1}g$ stabilizes
	a vertex of $\tilde T_0'$. It follows
	that $f\colon G \to  G'$ is surjective.

	We define a map $f_T\colon T\to T'$. Define
	$f_T$ on $F_0$ to be the isomorphism taking $\tilde e$
	to $\tilde  e'$ for each edge $e$ of $\Gamma$.
	For a point $x \in T$, choose a lift $\tilde\pi(x)$ of $\pi(x)$
	in $F_0$ in the following way. If $x$ is not a vertex,
	there is a unique lift. If $x$ is a vertex, choose
	the lift lying in $\tilde T_0$.
	There exists $g_x\in G$ such that $g_x.\tilde\pi(x) = x$.
	Extend $f_T$ to all of $T$ by defining 
	\[ f_T(x) = f(g_x).f_T(\tilde \pi(x)). \]
	
	The map $f_T$ is well-defined. Indeed, if $h\in G$
	satisfies $h.\tilde\pi(x) = x$, then 
	$h^{-1}g_x.\tilde\pi(x) = \tilde\pi(x)$, so $f(h^{-1}g_x) \in S$,
	from which it follows that 
	\[ f(h).f_T(\tilde\pi(x)) = f(g_x).f_T(\tilde\pi(x)). \]
	Since $f$ is surjective, $f_T$ is a $G$-equivariant
	surjective morphism of graphs.

	We will show that the induced maps
	$f_{\tilde v}\colon \st(\tilde v) \to \st(\tilde v')$
	are injective, from which it follows from
	$G$-equivariance and simple connectivity of $T'$
	that $f_T$ is an isomorphism, and thus
	$f$ is also an isomorphism.
	Indeed, the map $f$ restricts to an isomorphism
	from the stabilizer of $\tilde v$ to the stabilizer
	of $\tilde v'$, and similarly for $\tilde e$ and $\tilde e'$.
	By the orbit stabilizer theorem, 
	both $\st(\tilde v)$ and $\st(\tilde v')$
	are in $G_v$-equivariant bijection as sets with the set
	\[ \coprod_{e \in \st(v)} G_v/\iota_{e}(G_e), \]
	and this bijection is respected by $f_{\tilde v}$.
	This completes the proof.
\end{proof}

\section{Examples}
	
\begin{ex} Let $A$, $B$ and $C$ be groups with monomorphisms
	$\iota_A\colon C \to A$ and $\iota_B\colon C \to B$.
	The \emph{free product of $A$ and $B$ amalgamated along $C$,}
	written $A*_C B$
	is the fundamental group of the graph of groups $(\Gamma,\mathscr{G})$,
	and where $\Gamma$ is the graph with one edge $e$ and two vertices
	$v_1$ and $v_2$, and $\mathscr{G}$ is the graph of groups structure
	$G_{v_1} = A$, $G_{v_2} = B$, $G_e = C$, and  the monomorphisms
	are $\iota_A$ and $\iota_B$. It has the following  presentation
	\[ A*_C B = \langle A,B \mid \iota_A(c) = \iota_B(c),\ c \in C\rangle. \]
\end{ex}

An example of a fundamental group of a graph of groups whose action
on the Bass--Serre tree is not effective comes from the classical isomorphism
\[
\operatorname{SL}_2(\mathbb Z)\cong
\mathbb Z/4\mathbb Z*_{\mathbb Z/2\mathbb Z} \mathbb{Z}/6\mathbb{Z}.
\]
The tree in this example is the \emph{Serre tree} dual
to the Farey ideal tessellation of the hyperbolic plane. 
It has two orbits of vertices,
with valence $2$ and $3$, respectively.
The edge group $\mathbb{Z}/2\mathbb{Z}$ corresponds to
the \emph{hyperelliptic involution} $-I$ in $\operatorname{SL}_2(\mathbb{Z})$,
which acts trivially on the Farey tessellation, and thus trivially on the tree.
Thus the action factors through the quotient map 
$\operatorname{SL}_2(\mathbb Z)\to \operatorname{PSL}_2(\mathbb Z)$,
which acts effectively.

\begin{ex} Let $A$ and $C$ be groups, with two monomorphisms
	$\iota_1\colon C \to A$, $\iota_2\colon C \to A$.
	The \emph{HNN extension of $A$ with associated subgroups
	$\iota_1(C)$ and $\iota_2(C)$} is the fundamental group
	of the graph of groups $(\Gamma,\mathscr{G})$,
	where $\Gamma$ is the graph with one edge $e$ and one vertex $v$,
	and $\mathscr{G}$ is the graph of groups structure
	$G_v = A$, $G_e = C$, and the monomorphisms are $\iota_1$ and $\iota_2$.
	The name stands for Higman, Neumann and Neumann, 
	who gave the first construction. The standard notation
	is $A*_C$.
	It has the following presentation
	\[ A*_C = \langle A,t \mid t\iota_1(c)t^{-1} = \iota_2(c),\ 
	c \in C\rangle.	\]
\end{ex}

In the special case where $C = A$, $\iota_1 = 1_A$ is the identity
and $\iota_2 \in \aut(A)$, the HNN extension is isomorphic to
the semidirect product $A\rtimes_{\iota_2}\mathbb{Z}$,
and the Bass--Serre tree is homeomorphic to $\mathbb{R}$.
$A$ acts trivially on the tree and $\mathbb{Z}$ acts by a unit
translation of the graph structure.

\section{The Category of Graphs of Groups}\label{morphisms}

Thinking of a graph of groups $(\Gamma,\mathscr{G})$ as
a graph $\Gamma$ and a functor $\mathscr{G}$, a \emph{morphism}
$f\colon(\Lambda,\mathscr{L}) \to (\Gamma,\mathscr{G})$ is
a morphism of graphs $f\colon \Lambda\to \Gamma$ and a
\emph{pseudonatural transformation}
of functors as below left. Concretely, for each edge $e$ 
of $\Lambda$ and vertex $v = \tau(e)$, there are homomorphisms
$f_e\colon \mathscr{L}_e \to \mathscr{G}_{f(e)}$ and
$f_v\colon \mathscr{L}_v \to \mathscr{G}_{f(v)}$
of edge groups and vertex groups. Additionally,
for each oriented edge $e$, there is an element $\delta_e \in G_{f(v)}$
such that the diagram
below right commutes
\[\begin{tikzcd}
	\Lambda \ar[r,bend left=65, "\mathscr{L}"{name=L}]
	\ar[r,bend right=65,swap, "\mathscr{G}f"{name=Gf}]
	\ar[from=L,to=Gf,Rightarrow,shorten <=6pt, shorten >=6pt,swap,"f\ "]
	& \mathrm{Grp^{mono}}
\end{tikzcd}
\qquad\qquad
\begin{tikzcd}[column sep = large]
	\mathscr{L}_e \ar[r,"f_e"]\ar[d,"\iota_e"]
	& \mathscr{G}_{f(e)} \ar[d,"\ad(\delta_e)\iota_{f(e)}"] \\
	\mathscr{L}_v \ar[r,"f_v"] & \mathscr{G}_{f(v)},
\end{tikzcd}\]
where $\ad(\delta_e)\colon \mathscr{G}_{f(v)}\to \mathscr{G}_{f(v)}$
is the inner automorphism $g \mapsto \delta_eg\delta_e^{-1}$.
Note that if the graph morphism $f$ collapses the edge $e$,
the monomorphism
$\iota_{f(e)}$ should be replaced by the identity of
$\mathscr{G}_{f(e)}=\mathscr{G}_{f(v)}$.
In other words, if the morphism $f$ collapses the edge $e$,
then as a functor of small categories, $f\colon \Lambda \to \Gamma$
must send the morphism $\iota_e$ to the identity of $f(\tau(e))$.

Bass \cite[2.1]{Bass} also defines a notion of a morphism of a graph of groups.
In the case where $f$ does not collapse edges, Bass's definition is equivalent 
to our definition followed by an inner automorphism of
$\pi_1(\Gamma,\mathscr{G},p)$; ours is more like his morphism $\delta\Phi$
\cite[2.9]{Bass}.
It is a straightforward exercise to check that the composition
of two morphisms is a morphism, and that the evident ``identity''
morphism really behaves as the identity.

\section{The Action on Paths}
A morphism $f\colon(\Lambda,\mathscr{L})\to(\Gamma,\mathscr{G})$
acts on edge paths
\[\gamma=e_1'g_1e_2g_2\dotsb g_{k-1}e_k'\in (\Lambda,\mathscr{L}) \]
with each $g_i \in \mathscr{L}_{v_i}$ by acting
as $f_{v_i}$ on $\mathscr{L}_{v_i}$ and sending $e_i$
to $\delta_{\bar e_i}f(e_i)\delta_{e_i}^{-1}$. Thus $f$ sends $\gamma$
to the edge path
\[f(\gamma) = f(e_1')\delta_{e_1}^{-1}f_{v_1}(g_1)\delta_{\bar e_2}f(e_2)
\delta_{e_2}^{-1}f_{v_2}(g_2)\dotsb f_{v_{k-1}}(g_{k-1})\delta_{\bar e_k}
f(e_k') \in (\Gamma,\mathscr{G}).\]	
Observe that the pseudonaturality condition implies that 
if $e$ is an edge of $\Gamma$ with $h \in G_e$, then we have that
\begin{align*}  
	f(e\iota_{e}(h)) 
&= \delta_{\bar e}f(e)\cdot\delta_e^{-1}f_{\tau(e)}(\iota_{e}(h))
= \delta_{\bar  e}f(e)\cdot\iota_{f(e)}(f_e(h))\delta_e^{-1}\\
&\simeq\delta_{\bar e}\iota_{f(\bar e)}(f_e(h))\cdot f(e)\delta_e^{-1}
= f_{\tau(\bar e)}(\iota_{\bar e}(h))\delta_{\bar e}\cdot f(e)\delta_e^{-1}
= f(\iota_{\bar e}(h)e),
\end{align*}
from which it follows that $f$ preserves homotopy classes of paths
in $(\Lambda,\mathscr{L})$. Thus if $p$ is a point in $\Lambda$,
$f$ induces a homomorphism
\[f_\sharp\colon\pi_1(\Lambda,\mathscr{L},p)\to\pi_1(\Gamma,\mathscr{G},f(p)).\]

\begin{ex}
	Let $\gamma$ be a path in $(\Gamma,\mathscr{G})$ with
	endpoints at vertices, say
	\[	\gamma = g_1e_1\dotsb e_kg_{k+1}.	\]
	The path $\gamma$ determines a morphism 
	$\gamma\colon I\to(\Gamma,\mathscr{G})$,
	where $I$ is an interval of $\mathbb{R}$.
	To wit: give $I$ a graph structure with $k+1$ vertices and $k$ edges
	$E_1,\dotsb, E_k$ and the trivial graph of groups structure.
	Orient the $E_i$ so they point towards the terminal endpoint of $I$.
	There is an obvious graph morphism $\gamma\colon I \to \Gamma$
	sending $E_i$ to $e_i$. 
	To define $\gamma\colon I \to(\Gamma,\mathscr{G})$, we
	need to choose $\delta_E\in G_{\tau(f(E))}$ for each
	oriented edge $E$ of $I$. So set $\delta_{\bar E_i} = g_i$
	for $1\le i \le k$. For
	$1\le i\le k-1$, define $\delta_{E_i} = 1$,  and define
	$\delta_{E_k} = g_{k+1}^{-1}$.
	Since each edge and vertex group of $I$ is trivial,
	the pseudonaturality condition is trivially satisfied, 
	and observe that
	\[ \gamma(E_1\dotsb E_k) = 
	\delta_{\bar E_1}e_1\delta_{E_1}^{-1}\dotsb
	\delta_{\bar E_k}e_k\delta_{E_k}^{-1}
	= g_1e_1\dotsb e_kg_{k+1} = \gamma. \]
\end{ex}

\paragraph{A normal form for paths.}
For each oriented edge $e \in E(\Gamma)$, fix $S_e$ a set of left
coset representatives for $G_{\tau(e)}/\iota_e(G_e)$.
We require that $1\in G_{\tau(e)}$ belongs to $S_e$.
A path $\gamma$ in $(\Gamma,\mathscr{G})$ with endpoints
at vertices is in \emph{normal form} if
\[ \gamma = s_1e_1s_2\dotsb s_{k-1}e_kg_k, \]
where $s_i \in S_{\bar e_i}$ and $g_k \in G_{\tau(e_k)}$. Furthermore,
for those $s_i$ with $s_i = 1$, we require that $e_{i-1}\ne \bar e_i$.
Any path is homotopic rel endpoints to a unique path in normal form, and
one may inductively put a path into normal form in the following way.
Suppose $\gamma$ is a path satisfying
\[\gamma = \eta g_{i-1}e_ig_i\dotsb e_kg_k,\]
where $\eta=s_1e_1\dotsb g_{i-2}e_{i-1}$ is in normal form. Then
$g_{i-1}= s_{i-1}\iota_{\bar e_i}(h)$ for some $s_{i-1}\in S_{e_{i-1}}$,
and we may replace $\gamma$ with
\[ \gamma' = \eta s_{i-1}e_i\iota_{e_i}(h)g_{i}\dotsb e_kg_k. \]
We have that $\gamma\simeq\gamma'$, and $\eta s_{i-1}e_i$ is in normal
form unless $s_{i-1} = 1$ and $e_{i-1} = \bar e_i$, in which case we
may remove $e_{i-1}$ and $e_i$ via a homotopy.
We have either increased the portion of $\gamma$ which is in normal
form or shortened the length of $|\gamma|$ by two, so iterating this process
will terminate in finite time.

\section{Morphisms and the Bass--Serre Tree}\label{localmap}
The most important example of a morphism between graphs
of groups is the natural projection $\pi\colon T \to (\Gamma,\mathscr{G})$.
Since $T$ is an ordinary graph, its vertex groups and edge groups
are all trivial. For  $\pi$ to define a morphism of graphs of groups,
we need only to define $\delta_{\tilde e}$ for each oriented edge 
$\tilde e$ of $T$. 

Our definition depends on a
choice of fundamental domain $F \subset T$.
Choose a basepoint $p \in \Gamma$, a lift $\tilde p \in \pi^{-1}(p)$,
and a fundamental domain $F$ containing $\tilde p$.
In this way we fix an action of $\pi_1(\Gamma,\mathscr{G},p)$
on $T$.
Observe that for every vertex $\tilde v$ of $T$, there is a unique
path $\gamma$ in $(\Gamma,\mathscr{G})$ in the normal form
\[	\gamma=e_1's_1e_2\dotsb s_{k-1}e_k1\]
which lifts to a tight path $\tilde\gamma$ from $\tilde p$ to $\tilde v$.
If $\tilde e$ is an oriented edge of $T$, let $\tilde\gamma$ be the unique
tight path in $T$ connecting $\tilde p$ to $\tau(\tilde{e})$.
The path $\tilde\gamma$ corresponds to a unique path $\gamma$
in $(\Gamma,\mathscr{G})$ of the above form. If $\pi(\tilde e)=e_k$ 
is the last edge appearing in $\gamma$, define 
$\delta_{\overline{\tilde e}}=s_{k-1}$.
Otherwise define $\delta_{\overline{\tilde e}} = 1$.
Observe that under this definition, $\pi(\tilde\gamma) = \gamma$
as a path in $(\Gamma,\mathscr{G})$.

Recall that for each vertex $v\in\Gamma$ and each lift $\tilde v\in\pi^{-1}(v)$,
there is a $G_v$-equivariant bijection
\[\st(\tilde v) \cong \coprod_{e\in\st(v)} G_v/\iota_e(G_e)\times\{e\}.\]
In fact, the morphism $\pi$ induces this bijection: if 
$\pi(\tilde e) = \delta_{\overline{\tilde e}}e\delta_{\tilde e}^{-1}$,
the above map is $\tilde e \mapsto ([\delta_{\tilde e}],e)$, where
$[\delta_e]$ denotes the left coset of $\iota_e(G_e)$ in $G_v$
represented by $\delta_e$.

\begin{rk}
	Observe that $\pi_1(\Gamma,\mathscr{G},p)$ does \emph{not} act 
	as a group of
	automorphisms of $\pi\colon T \to (\Gamma,\mathscr{G})$ as a morphism
	of graphs of groups: if $g \in \pi_1(\Gamma,\mathscr{G},p)$,
	the equality $\delta_{g.\tilde e} =  \delta_{\tilde e}$
	only holds when $\tau({\tilde e}) \ne \tilde p$.
	This rigidity of the projection will turn out to be useful.
\end{rk}

More generally let $f\colon(\Lambda,\mathscr{L})\to(\Gamma,\mathscr{G})$
be a morphism of graphs of groups. If the graph morphism $f$ does not collapse
any edges, define a map
\[
	f_{\st(v)}\colon\coprod_{e\in\st(v)}
	\mathscr{L}_v/\iota_e(\mathscr{L}_e)\times\{e\}\to
	\coprod_{e'\in\st(f(v))}\mathscr{G}_{f(v)}/\iota_{e'}(\mathscr{G}_{e'})
	\times\{e'\}\]
	\[f_{\st(v)}([g],e) = 
	([f_v(g)\delta_{e}],f(e)).\]
The pseudonaturality condition ensures that $f_{\st(v)}$ is well-defined: since
\[ f_v(g\iota_e(h))\delta_e = f_v(g)f_v(\iota_e(h))\delta_e
= f_v(g)\delta_e\iota_{f(e)}(f_e(h)), \]
we conclude different choices of representative determine 
the same coset of $\iota_{f(e)}(\mathscr{G}_{f(e)})$ under $f_{\st(v)}$.

\begin{prop}\label{diagramfill1}
	Let $f\colon(\Lambda,\mathscr{L}) \to (\Gamma,\mathscr{G})$ be a morphism,
	$p \in \Lambda$ and $q = f(p)\in \Gamma$.
	Write $\tilde\Lambda$ and $\tilde\Gamma$ for the Bass--Serre trees of
	$(\Lambda,\mathscr{L})$ and $(\Gamma,\mathscr{G})$, respectively. 
	There is a unique morphism $\tilde f\colon \tilde\Lambda \to \tilde\Gamma$
	such that the following diagram commutes
	\[\begin{tikzcd}
		(\tilde\Lambda,\tilde p) \ar[r, dashed, "\tilde f"]
		\ar[d, "\pi_\Lambda"] & (\tilde\Gamma,\tilde q) \ar[d, "\pi_\Gamma"] \\
		(\Lambda,\mathscr{L},p) \ar[r, "f"] & (\Gamma,\mathscr{G},q),
	\end{tikzcd}\]
	where $\pi_\Lambda$ and $\pi_\Gamma$ are the covering projections,
	and $\tilde p$ and $\tilde q$ are the distinguished lifts of $p$ and $q$, 
	respectively.
\end{prop}

\begin{proof}
	If $x$ is a point of $\tilde\Lambda$, there is a unique tight path
	$\gamma$ from $\tilde p$ to $x$. The path $f\pi_\Lambda(\gamma)$
	has a (unique) lift $\tilde\gamma$ to 
	$\tilde\Gamma$ beginning at $\tilde q$
	such that $\pi_\Gamma(\tilde\gamma)$
	differs from $f\pi_\Lambda(\gamma)$ only possibly
	by an element of $\mathscr{G}_{f\pi_\Lambda(x)}$ when $x$ is a vertex.
	In particular, the terminal endpoint $\tilde\gamma(1)$ is well-defined,
	and depends only on the homotopy class rel endpoints of 
	$f\pi_\Lambda(\gamma)$. 
	Define $\tilde f(x) = \tilde\gamma(1)$.
	It is easy to see that $\tilde f\colon\tilde\Lambda \to \tilde\Gamma$
	defines a morphism, and that the diagram commutes.
	For uniqueness, observe that any morphism $\tilde f'$ making the 
	above diagram commute must satisfy
	\[\pi_\Gamma\tilde f'(\tilde\eta) = f\pi_\Lambda(\tilde\eta) \]
	for any path $\tilde\eta$ in $(\tilde\Lambda)$, so specializing to paths 
	connecting $\tilde p$ to $x$, we see that $\tilde f'(x) = \tilde f(x)$ 
	for all $x \in \Lambda$.
\end{proof}

Notice that the map $\tilde f$ is $f_\sharp$-equivariant, in the sense
that if $h \in \pi_1(\Lambda,\mathscr{L},v)$, 
then for all $x \in \tilde\Lambda$ we have
\[\tilde f(h.x) = f_\sharp(h).\tilde f(x).\]
The converse is also true. 
\begin{prop}[cf.\ 4.1--4.5 of \cite{Bass}]\label{diagramfill2}
	Let $f_\sharp\colon \pi_1(\Lambda,\mathscr{L},p) 
	\to \pi_1(\Gamma,\mathscr{G},q)$ be a homomorphism,
	and let $\tilde f\colon (\tilde\Lambda,\tilde p) 
	\to (\tilde\Gamma,\tilde q)$
	be an $f_\sharp$-equivariant morphism of trees in the sense above. 
	There is a morphism 
	$f\colon (\Lambda,\mathscr{L},p) \to (\Gamma,\mathscr{G},q)$ 
	which induces $f_\sharp$ and $\tilde f$ and 
	such that the following diagram commutes
	\[\begin{tikzcd}
		(\tilde\Lambda,\tilde p) \ar[r, "\tilde f"] \ar[d, "\pi_\Lambda"] 
		& (\tilde\Gamma,\tilde q) \ar[d, "\pi_\Gamma"] \\
		(\Lambda,\mathscr{L},p) \ar[r, dashed, "f"] & (\Gamma,\mathscr{G},q).
	\end{tikzcd}\]
\end{prop}

\begin{proof}
	As a morphism of graphs $f$ is easy to describe. 
	By $f_\sharp$-equivariance, the map $\pi_\Gamma\tilde f$ yields a
	well-defined map on $\pi_1(\Lambda,\mathscr{L},p)$-orbits;
	this is the map $f\colon \Lambda \to \Gamma$ as a morphism of graphs.

	Identify the graph of groups structures $\mathscr{L}$ and $\mathscr{G}$
	with those induced by the actions of the respective fundamental groups on
	$\tilde\Lambda$ and $\tilde\Gamma$. For $\tilde\Lambda$ 
	this involves a choice of
	fundamental domain $F\subset\tilde\Lambda$ containing $\tilde p$.
	Each edge $e \in \Lambda$ has a single preimage $\tilde e \in F$,
	and as a morphism of graphs of groups, $\pi_\Lambda$ satisfies
	\[	(\pi_\Lambda)_{\st(\tau(\tilde e))}(\tilde e ) = ([1],e). \]
	If $e$ is not collapsed by $f$, 
	define $\delta_e$ for the morphism $f$ as $\delta_{\tilde f(\tilde e)}$
	for the morphism $\pi_\Gamma$. Thus, for $\tilde\gamma$ a path in $F$,
	$\pi_\Gamma\tilde f(\tilde \gamma) = f\pi_\Lambda(\tilde \gamma)$.
	If $e$ is collapsed, define $\delta_e = 1$.

	Let $v$ be a vertex of $\Lambda$ and write $w = f(v)$.
	To define $f_v\colon \mathscr{L}_v \to \mathscr{G}_{w}$,
	recall that under the identification, $\mathscr{L}_v$ 
	is the stabilizer of a vertex $\tilde v \in \pi_\Lambda^{-1}(v)\cap F$.
	Similarly, $\mathscr{G}_{f(v)}$ is identified with the stabilizer of 
	some vertex $\tilde w$ in $\tilde\Gamma$. 
	The stabilizers of $\tilde f(\tilde v)$ and $w$ are conjugate in
	$\pi_1(\Gamma,\mathscr{G},q)$ via some element $g_w$ such that
	$g_w.\tilde f(\tilde v) = w$. Furthermore recall that there is a 
	preferred translate of the fundamental domain for the action of
	$\pi_1(\Gamma,\mathscr{G},q)$ containing $\tilde f(\tilde v)$ and
	another preferred translate containing $\tilde w$.
	Namely, in the language we have been developing, the translate containing
	the edges in the preimage of
	\[	\{([1],e) : e\in\st(w)\}	\]
	under the maps $(\pi_\Gamma)_{\st( \tilde f(\tilde v) )}$ and
	$(\pi_\Gamma)_{\st(\tilde w)}$, respectively.
	We require $g_w$ to take the preferred translate for
	$\tilde f(\tilde v)$ to the preferred translate for $\tilde w$.
	The restriction of 
	\[h \mapsto g_wf_\sharp(h)g_w^{-1}\]
	to the stabilizer of $\tilde v$ defines a homomorphism 
	$f_v\colon \mathscr{L}_v\to\mathscr{G}_w$.

	Now let $e$ be an edge of $\Lambda$ and write $a = f(e)$. 
	The story is similar:
	the stabilizer of the preimage $\tilde e$ in $F$ is identified with
	$\mathscr{L}_e$, and some element $g_a\in\pi_1(\Gamma,\mathscr{G},q)$
	takes $\tilde f(\tilde e)$ to the preferred preimage $\tilde a$.
	If $a$ is a vertex, we again require $g_a$ 
	to match up preferred fundamental domains.
	The homomorphism $f_e\colon \mathscr{L}_e \to \mathscr{G}_a$ is
	\[ h \mapsto g_af_\sharp(h)g_a^{-1}.	\]

	Whether $e$ is an edge of $\Gamma$ or $\Lambda$, the monomorphisms 
	$\iota_e$ have a uniform description. If $v = \tau(e)$,
	and $\tau(\tilde e) = \tilde v$, then
	the monomorphism $\iota_e$ is the inclusion
	$\stab(\tilde e)\hookrightarrow \stab(\tilde v)$.
	If not, there is some element $t_e$ with $t_e.\tau(\tilde e) = \tilde v$
	and matching up preferred fundamental domains. Then $\iota_e$
	is 
	\[	h \mapsto t_eht_e^{-1}.	\]
	In the former case, write $t_e = 1$.
	In the case where $f$ collapses $e$ to a vertex $a$,
	let $\iota_a = 1_{\mathscr{G}_a}$ and $t_a = 1$.
	Tracing an element $h \in\mathscr{L}_e$ 
	around the pseudonaturality square, we have
	\[\begin{tikzcd}
		h \ar[r,maps to,"f_e"] \ar[d,maps to,"\iota_e"]
		& g_af_\sharp(h)g_a^{-1} \ar[d,maps to,"\ad(\delta_e)\iota_a"]  \\
		t_eht_e^{-1} \ar[r,maps to,"f_v"]
		& g_wf_\sharp(t_eht_e^{-1})g_w^{-1} 
		= \delta_et_ag_af_\sharp(h)g_a^{-1}t_a^{-1}\delta_e^{-1}.
	\end{tikzcd}\]
	Equality holds: to see this, we only need to check the case where
	$e$ is not collapsed. In this case, note that $\delta_e\in \mathscr{G}_w$
	was defined to take $(t_ag_a).\tilde f(\tilde e)$,
	whose image under $(\pi_\Gamma)_{\st(\tilde w)}$
	is $([1],e)$, to $(g_wf_\sharp(t_e)).\tilde f(\tilde e)$.
	This completes the definition of $f$ as a morphism of graphs of groups.

	With this description of $f$, it is easy to see that $f$ induces
	$f_\sharp$. Choose a spanning tree $T \subset \Lambda$.
	For $v$ a vertex of $\Lambda$, let $\gamma_v$ be a path
	in $T$ (as a graph) between $p$ and $v$. 
	A set of generators of $\pi_1(\Lambda,\mathscr{L},p)$ is given by
	\[	\{\gamma_vg\bar\gamma_v, \gamma_{\tau(\bar e)}e\bar\gamma_{\tau(e)} :
	g \in \mathscr{L}_v,\ e \text{ an edge of } \Lambda\setminus T\}.	\]
	One checks that $f(\gamma_v)f_v(g)f(\bar\gamma_v) 
	= f_\sharp(\gamma_vg\bar\gamma_v)$ by observing that both elements
	stabilize $\tilde f(\tilde v)$ and act as the same element of 
	$\mathscr{G}_{f(v)}$ on the set of fundamental domains containing 
	$\tilde f(\tilde v)$. Similarly, observe that
	$f(\gamma_{\tau(\bar e)}e\bar\gamma_{\tau(e)}) = 
	f_\sharp(\gamma_{\tau(\bar e)}e\bar\gamma_{\tau(e)})$ 
	since if $\tau(e) = v$, both elements take
	$\tilde f(\tilde v)$ to $\tilde f(\tau(\tilde e))$ and 
	match fundamental domains identically.
	From this it follows from $f_\sharp$-equivariance that $f$ also
	induces $\tilde f$.
\end{proof}

\section{Equivalence of Morphisms}\label{equivalence}
In the notation of the previous proposition,
it is not the case that $f$ is uniquely determined 
by $f_\sharp$ and $\tilde f$.
Here is a simple example. 
Let $\Gamma$ be the graph with one edge $e$ and two vertices, $v_1$ and $v_2$,
and let $\mathscr{G}$ be the graph of groups structure $\mathscr{G}$ 
with vertex groups $G_1$ and $G_2$ and trivial edge group.
Orient $e$ so that $\tau(e) = v_1$.
Let $f\colon(\Gamma,\mathscr{G}) \to (\Gamma,\mathscr{G})$ be the morphism
which is the identity on each vertex and edge group,
and set $\delta_e = z$ for some element $z$ in the center of $G_1$,
and set $\delta_{\bar e} = 1$.
At any basepoint other than $v_1$, the morphism $f$ has
$f_\sharp$ and $\tilde f$ equal to the identity,
while at $v_1$ the homomorphism 
$f_\sharp\colon\pi_1(\Gamma,\mathscr{G},v_1) \to \pi_1(\Gamma,\mathscr{G},v_1)$
is the automorphism $\ad(z)$, and $\tilde f$ is equal to the action of $z$
on the Bass--Serre tree $T$, viewed as an element of $\stab(\tilde v_1)$.

More generally, if $f\colon (\Lambda,\mathscr{L},p) \to (\Gamma,\mathscr{G}),q)$
is a morphism of graphs of groups,
the following operations do not change $f_\sharp$ or $\tilde f$.
\begin{enumerate}
	\item If $v \ne p$ is a vertex of $\Lambda$ and
		$g \in \mathscr{G}_{f(v)}$,
		replace $f_v$ with $\ad(g)\circ f_v$
		and replace $\delta_e$ with $g\delta_e$
		for each oriented edge $e \in \st(v)$.
	\item If $e$ is an edge of $\Lambda$ not containing $p$
		and $g \in \mathscr{G}_{f(e)}$,
		replace $\delta_e$ and $\delta_{\bar e}$
		with $\delta_{\bar e}\iota_{\overline{f(e)}}(g)$
		and $\delta_e\iota_{f(e)}(g)$, respectively.
\end{enumerate}
A simple calculation reveals that neither operation affects
$f_\sharp$ nor $\tilde f$.
If $f$ can be transformed into $f'$ 
by a finite sequence of the above operations,
we say $f$ and $f'$ are \emph{equivalent fixing $p$.}\footnote{
	One could argue for saying \emph{homotopy} equivalent instead.
	We see no particular reason to prefer one terminology over the other,
	save that neither $f_\sharp$ nor $\tilde f$ 
	see any change under this form
	of equivalence.
} More generally, if one ignores the stipulations regarding the basepoint,
we will say two such morphisms are \emph{equivalent.} 
Equivalence classes of morphisms will become relevant for representing
outer automorphisms of $\pi_1(\Lambda,\mathscr{L})$.

\paragraph{Twist automorphisms.}
We would like to draw attention to a particular family of automorphisms
of $(\Lambda,\mathscr{L})$.
For $v$ a vertex of $\Lambda$,
$e$ an oriented edge of $\st(v)$
and $g \in \mathscr{L}_v$,
define the \emph{twist of $e$ by $g$,}
\[t_{ge}\colon (\Lambda,\mathscr{L}) \to (\Lambda,\mathscr{L})\]
by setting $(t_{ge})_x$ equal to the identity of $\mathscr{G}_x$ 
for $x$ an edge or vertex of $\Lambda$,
setting $\delta_e = g$, 
and setting $\delta_{e'} = 1$ for each other edge of $\Lambda$.
In the Bass--Serre tree $\tilde\Lambda$,
the lift $\tilde t_{ge}$ corresponds to changing the fundamental domain $F$
for the action of $\pi_1(\Lambda,\mathscr{L},p)$.
Since $F$ is a tree, 
removing the interior of $\tilde e$ separates $F$ into two components.
Call the component containing $\tilde p$ $F_0$,
call $F_1$ the closure of $F\setminus F_0$.
The lift $\tilde t_{ge}$ acts by fixing $F_0$ and sending $F_1$ to $g.F_1$.

Cohen and Lustig \cite{CohenLustig} define a similar notion
which they call a \emph{Dehn twist automorphism.} 
Our twists are Dehn twists when $g$ lies in the image
of the center of the edge group $Z(\mathscr{L}_e)$ under $\iota_e$.
Recall that an essential simple closed curve $\gamma$ on a surface $S$
determines a splitting of $\pi_1(S)$ 
as a one-edge graph of groups $(\Gamma,\mathscr{G})$.
The action of the Dehn twist about $\gamma$ on $\pi_1(S)$
may be realized as a Dehn twist automorphism of $(\Gamma,\mathscr{G})$.
More generally, Dehn twist automorphisms on some splitting of
$\pi_1(S)$ as a graph of groups with cyclic edge groups correspond
to products of commuting Dehn twists about some multi-curve on $S$.

\section{Covering Spaces}
Let $f$ be a morphism  that does not collapse  edges.
The map $f_{\st(v)}$ captures the local behavior of $f$ at $v$.
We say that $f$ is \emph{locally injective} or an \emph{immersion} if 
\begin{enumerate*}[label=(\roman*)]
	\item each homomorphism $f_e$ and $f_v$ is a monomorphism, and 
	\item each $f_{\st(v)}$ is injective.
\end{enumerate*}
An immersion is an \emph{embedding} if
additionally $f$ is injective as a morphism of graphs.
We say an immersion 
$f$ is a \emph{covering map} if $f$ is surjective as a morphism of graphs
and each $f_{\st(v)}$ is a bijection.
The composition of two immersions is clearly an immersion,
and the composition of two covering maps is a covering map.
(Recall that we require graphs of groups to be connected.)

\begin{lem} If $f\colon(\Lambda,\mathscr{L},p)\to(\Gamma,\mathscr{G},f(p))$
	is an immersion, the induced homomorphism
	\[	f_\sharp\colon\pi_1(\Lambda,\mathscr{L},p)\to
	\pi_1(\Gamma,\mathscr{G},f(p))	\]
	is a monomorphism.
\end{lem}

\begin{proof}
	Suppose to the contrary that $g\in\pi_1(\Lambda,\mathscr{L},p)$
	is a nontrivial element in the kernel of $f_\sharp$,
	and represent $g$ by an
	immersion $\gamma\colon I \to (\Lambda,\mathscr{L})$.\footnote{
	Technically our combinatorial definition of morphism
	requires $p$ to be a vertex of $\Lambda$. 
	This can be accomplished by declaring $p$
	to be a vertex and altering $\mathscr{L}$ in the obvious way.
}
	Since $f(\gamma)$ is null-homotopic, (the image of)
	$\gamma$ has a subpath $ege'$ 
	such that $f(e) = f(\bar e')$ and 
	\[ \delta_e^{-1}f_v(g)\delta_{e'} = \iota_{f(e)}(h), \]
	where $v = \tau(e)$ and $h\in G_{f(e)}$.
	Thus $[\delta_e] = [f_v(g)\delta_{e'}]$, 
	so  we conclude that 
	\[ f_{\st(v)}([1],e) = f_{\st(v)}([g],\bar e').	\]
	But because $\gamma\colon I \to (\Lambda,\mathscr{L})$ is
	an immersion, either $e \ne \bar e'$ or $[g] \ne [1]$,
	so this contradicts the assumption that $f$ is an immersion.
\end{proof}

Put another way, the above proof shows that
an immersion $f\colon(\Lambda,\mathscr{L})\to(\Gamma,\mathscr{G})$
sends immersed paths in $(\Lambda,\mathscr{L})$ to immersed paths in 
$(\Gamma,\mathscr{G})$.

\begin{lem}[Corollary 4.6 of \cite{Bass}]
	Let $f\colon (\Lambda,\mathscr{L},p) \to (\Gamma,\mathscr{G},q)$
	be a morphism. The corresponding morphism of Bass--Serre trees
	$\tilde f\colon \tilde\Lambda \to \tilde\Gamma$ is an embedding
	if and only if $f$ is an immersion, and is an isomorphism if and only if
	$f$ is a covering map. \hfill\qedsymbol
\end{lem}

We saw in the previous section that the projection
$\pi\colon T\to (\Gamma,\mathscr{G})$ from the Bass--Serre
tree $T$ is a covering map.  
Let $H$ be a subgroup of $\pi_1(\Gamma,\mathscr{G},p)$. 
The action of $H$ on $T$ gives $H\backslash T$ 
the structure of a graph of groups,
with covering map $\rho\colon T \to H\backslash T$.

\begin{cor}
	With notation as above, the map $\rho$ fits in
	a commutative diagram of covering maps of graphs of groups
	as below
	\[\begin{tikzcd}[row sep = small,column sep = small]
		T\ar[dd,"\pi"]\ar[dr,"\rho"] & \\
		& H\backslash T\ar[dl,"r"] \\
		(\Gamma,\mathscr{G}).
	\end{tikzcd}\]
\end{cor}

\begin{proof} Apply \Cref{diagramfill2} to the following diagram
	\[\begin{tikzcd}
		T \ar[r,equals] \ar[d,"\rho"] & T \ar[d,"\pi"] \\
		H\backslash T \ar[r,dashed,"r"] & (\Gamma,\mathscr{G}).
	\end{tikzcd}\]
	
	\vspace{-2em}\qedhere
\end{proof}

\begin{prop} Suppose $f\colon (\Lambda,\mathscr{L},p) \to (\Gamma,\mathscr{G},q)$
	is a morphism of graphs of groups and that 
	$\rho\colon (\bar\Gamma,\bar{\mathscr{G}},\bar q) \to (\Gamma,\mathscr{G},q)$
	is a covering map. There exists a morphism $\bar f$ making
	the following diagram commute
	\[\begin{tikzcd}
		& (\bar{\Gamma},\bar{\mathscr{G}},\bar q) \ar[d,"\rho"] \\
		(\Lambda,\mathscr{L},p) \ar[r,"f"] \ar[ur,dashed,"\bar f"]
		& (\Gamma,\mathscr{G},q)
	\end{tikzcd}\]
	if and only if $f_\sharp(\pi_1(\Lambda,\mathscr{L},p)) \le 
	\rho_\sharp(\pi_1(\bar\Gamma,\bar{\mathscr{G}},\bar q))$.
\end{prop}

\begin{proof}
	That the condition is necessary is obvious. For sufficiency observe
	that the previous corollary implies that \Cref{diagramfill2} 
	applies to the diagram
	\[\begin{tikzcd}
		(\tilde\Lambda,\tilde p) \ar[r,"\tilde f"] \ar[d,"\pi_\Lambda"]
		& (\tilde\Gamma,\tilde q) \ar[d,"\pi_{\bar{\Gamma}}"] \\
		(\Lambda,\mathscr{L},p) \ar[r,dashed,"\bar f"] & 
		(\bar\Gamma,\bar{\mathscr{G}},q).
	\end{tikzcd}\]

	\vspace{-2em}\qedhere
\end{proof}

The proposition immediately implies the following corollary.
\begin{cor}\label{decktransformation}
	If $\rho_1\colon (\Lambda_1,\mathscr{L}_1,p_1) \to (\Gamma,\mathscr{G},q)$
	and $\rho_2\colon (\Lambda_2,\mathscr{L}_2,p_2) \to (\Gamma,\mathscr{G},q)$
	are covering maps, and 
	$(\rho_1)_\sharp(\pi_1(\Lambda_1,\mathscr{L}_1,p_1)) =
	(\rho_2)_\sharp(\pi_1(\Lambda_2,\mathscr{L}_2,p_2))$,
	then there is an isomorphism
	$f\colon (\Lambda_1,\mathscr{L}_1,p_1) \to 
	(\Lambda_2,\mathscr{L}_2,p_2)$
	such that $\rho_2f = \rho_1$.\hfill\qedsymbol
\end{cor}

In fact, this observation proves the following classification of 
covering spaces.

\begin{thm}\label{Galoiscorrespondence}
	There is a Galois correspondence between
	(connected) pointed covering maps
	$\rho\colon (\Lambda,\mathscr{L},q)\to (\Gamma,\mathscr{G},p)$
	and subgroups $H\le\pi_1(\Gamma,\mathscr{G},p)$.
	The correspondence sends a covering map $\rho$
	to the subgroup $\rho_\sharp(\pi_1(\Lambda,\mathscr{L},q))$
	and sends a subgroup $H$ to the covering map
	$r\colon (H\backslash T,q)\to(\Gamma,\mathscr{G},p)$
	constructed above, where $q$
	is the image of $\tilde p$ under the covering projection
	$T \to H\backslash T$.\hfill\qedsymbol
\end{thm}

\chapter{Folding}\label{foldingchapter}

In the seminal paper \cite{Stallings}, Stallings introduced the notion of
\emph{folding} morphisms of graphs 
and used it to give conceptually powerful new proofs
of many classical results about free groups.
He later observed \cite{StallingsGTrees} that 
the result of equivariantly folding a $G$-tree is again a $G$-tree,
so the theory extends to morphisms of $G$-trees.
Extending Dunwoody's work on accessibility, 
Bestvina and Feighn \cite{BestvinaFeighnGTrees} analyzed the effects of folding
a $G$-tree on the quotient graph of groups to bound the complexity of
small $G$-trees when $G$ is finitely presented. 
Recall that a group is \emph{small} if it does not contain a free 
subgroup of rank two. A $G$-tree is \emph{small} if each edge stabilizer is small.
Accessibility of finitely presented groups is equivalent to a finite
bound (depending  on $G$) for the complexity of
$G$-trees with \emph{finite} edge  stabilizers.
Dunwoody \cite{Dunwoody} extended folding to equivariant morphisms of trees
to show that finitely generated inaccessible groups contain infinite
torsion subgroups. Kapovich, Weidmann and 
Miasnikov \cite{KapovichWeidmannMiasnikov} used foldings of graphs of groups
to study the membership problem.

\paragraph{The idea of Stallings folding.}
Here is the setup. Suppose $G$ is the fundamental group of a graph $\Gamma$
(based at a vertex $p \in \Gamma$).
By \Cref{Galoiscorrespondence},
a subgroup $H \le G$ corresponds to a covering map
$\Gamma_H \to \Gamma$. Even if $H$ is finitely generated, if it is not of
finite index in $G$, the graph $\Gamma_H$ will be infinite.
In this case $\Gamma_H$ contains a compact \emph{core} $K_H$
where the fundamental group is concentrated, and the complement of $K_H$
in $\Gamma_H$ is a forest. The restriction of the covering map
$\Gamma_H \to \Gamma$ to $K_H$ is no longer a covering map, but
it remains an immersion, and the image in $\pi_1(\Gamma)$ is $H$.

If one knew how to describe the core $K_H$,
a whole host of useful information about $H$ would follow:
the rank of $H$,
a nice free basis for $H$, whether $H$ has finite index in $G$,
and so forth. If on the other hand one were 
just given a generating set $S$ for $H$, recovering this information 
purely algebraically is not straightforward. 

On the other hand, $S$ \emph{does}
give a representation of $H$ as a morphism of graphs with codomain $\Gamma$:
for each $s \in S$, represent $s$ by an edge path loop $\gamma_s$ in $\Gamma$ 
based at $p$. This yields a map from the circle $\gamma_s\colon S^1 \to \Gamma$.
Subdividing the circle yields a basepointed morphism of graphs.
Taking the wedge sum at the basepoint $\ast$ 
yields a graph $\Lambda_H$ and a morphism 
$f\colon (\Lambda_H,\ast) \to (\Gamma,p)$ with 
$f_\sharp(\pi_1(\Lambda_H,\ast)) = H$.
Since $H$ is finitely generated, $\Lambda_H$ is a finite graph.

The main insight of Stallings is that $f$ may be promoted to an immersion
in a finite and algorithmic fashion as we now explain.
If $f$ is not already an immersion, then it identifies a pair of edges
incident to a common vertex. We say $f$ \emph{folds} these edges,
and $f$ factors through the quotient map $\Lambda_H \to \Lambda'$
which identifies them. The quotient map is called a \emph{fold.}
If the resulting morphism $f'\colon\Lambda' \to \Gamma$ 
is not an immersion, we continue to fold.
At each step, the number of edges in the graph decreases,
so at some finite stage, the resulting morphism $\bar f\colon\bar\Lambda \to \Gamma$ 
must be an immersion---in fact, 
it is not hard to see that the core of $\bar\Lambda$ 
and $K_H$ are isomorphic!

In principle, if $H$ is not finitely generated (or more precisely, if the set $S$
is not finite), nothing untoward happens:
one just folds ``forever,'' and passing to the direct limit yields an immersion.
Although the process is no longer algorithmic in any strict sense, 
the constructive nature of the process remains valuable in applications.

Our interest in folding will be primarily to study automorphisms
of groups acting on trees.
We take up this study at the end of the chapter.

\section{Folding Bass--Serre Trees}
Let us begin making the foregoing discussion precise. Call a morphism 
$f\colon (\Lambda,\mathscr{L}) \to (\Gamma,\mathscr{G})$ of graphs of groups
\emph{Stallings} if it 
\begin{enumerate*}[label=(\roman*)]
	\item does not collapse edges and \item each group homomorphism
		$f_v\colon\mathscr{L}_v \to \mathscr{G}_v$ and
		$f_e\colon\mathscr{L}_e \to \mathscr{G}_e$
		is a monomorphism.
\end{enumerate*}
By \Cref{diagramfill1}, once we fix a basepoint, there is an
equivariant morphism $\tilde f\colon \tilde\Lambda \to \tilde\Gamma$ 
of Bass--Serre trees lifting $f$. If $f$ is not an immersion,
neither is $\tilde f$, so there is some vertex $\tilde v \in \tilde \Lambda$ 
such that $\tilde f_{\st(\tilde v)}$ is not injective, i.e.
there are distinct edges $\tilde e_1$ and $\tilde e_2$ 
with $\tau(\tilde e_1) = \tau(\tilde e_2) = \tilde v$, and 
$\tilde f(\tilde e_1) = \tilde f(\tilde e_2')$.

Define an equivalence relation on the $\tilde\Lambda$ as follows:
\begin{enumerate}
	\item $\tilde e_1 \sim \tilde e_2$
	and $\tau(\bar{\tilde e}_1) \sim \tau(\bar{\tilde e}_2)$.
\item Extend item 1. equivariantly: 
	if $x \sim y$ and $g \in \pi_1(\Lambda,\mathscr{L})$,
		then $g.x \sim g.y$.
\end{enumerate}
Let $\tilde \Lambda/_{ [\tilde e_1 = \tilde e_2] }$ be the quotient
of $\tilde\Lambda$ by this equivalence relation.
Stallings \cite[Section 3]{Stallings} notes that this is an example of 
a \emph{pushout} construction.
An observation of Chiswell~\cite[Theorem 4]{StallingsGTrees} says that
$\tilde\Lambda/_{ [\tilde e_1 = \tilde e_2] }$ 
is a tree, $\pi_1(\Lambda,\mathscr{L})$ 
acts without inversions, and the morphism $\tilde f$ factors as below
\[\begin{tikzcd}[row sep = tiny]
	\tilde\Lambda \ar[rr,"\tilde f"]\ar[rd] & & \tilde\Gamma. \\
	& \tilde\Lambda/_{ [\tilde e_1 = \tilde e_2] } \ar[ur] & 
\end{tikzcd}\]
The morphism $\tilde\Lambda \to \tilde\Lambda/_{[\tilde e_1 = \tilde e_2]}$ is
called a \emph{fold.} In fact, by \Cref{diagramfill2}, there is
also a factorization of $f$ such that the following diagram commutes
\[\begin{tikzcd}[row sep = tiny]
	\tilde\Lambda \ar[rd]\ar[rr,"\tilde f"]\ar[dd,"\pi_\Lambda"]
	& & \tilde \Gamma \ar[dd,"\pi_\Gamma"]\\
	& \tilde\Lambda/_{[\tilde e_1 = \tilde e_2]} \ar[ur] & \\
	(\Lambda,\mathscr{L}) \ar[rr,near start,"f"]\ar[rd]
	& &  (\Gamma,\mathscr{G}), \\
	& (\Lambda,\mathscr{L})/_{[e_1 = e_2]} \ar[from=uu,crossing over] 
	\ar[ur] & 
\end{tikzcd}\]
where $e_i = \pi_\Lambda(\tilde e_i)$, where we define
\[(\Lambda,\mathscr{L})/_{[e_1 = e_2]}
= \pi_1(\Lambda,\mathscr{L})\backslash\tilde\Lambda/_{[\tilde e_1=\tilde e_2]},\]
and where the morphism $\tilde\Lambda/_{[\tilde e_1 = \tilde e_2]}
\to (\Lambda,\mathscr{L})/_{[e_1=e_2]}$
is the covering projection.

Bestvina and Feighn analyze the combinatorial possibilities
for the resulting morphism 
$(\Lambda,\mathscr{L})\to (\Lambda,\mathscr{L})/_{[e_1 = e_2]}$,
which is Stallings, and which we will also call a fold. 
We discuss these possibilities below.

Ideally, one would like to say that if
$f\colon(\Lambda,\mathscr{L}) \to (\Gamma,\mathscr{G})$
is a Stallings morphism and $\Lambda$ is a finite graph,
then inductively $f$ factors as a finite sequence of folds followed by
an immersion. Unfortunately, this is just not true;
here is the problem. In the notation above,
let $\tilde v_1$ and $\tilde v_2$ denote
the initial vertices of $\tilde e_1$ and $\tilde e_2$ respectively,
and let $\tilde w$ be their common image in 
$\tilde\Lambda/_{[\tilde e_1=\tilde e_2]}$.
Note that 
\[ \langle \stab(\tilde v_1),\stab(\tilde v_2)\rangle \le \stab(\tilde w), \]
but that the left-hand subgroup is typically much \emph{larger} than
\[ \langle f_\sharp(\stab(\tilde v_1)),f_\sharp(\stab(\tilde v_2))\rangle
\le \stab(\tilde f(\tilde v_1)).\]
Thus the morphism $(\Lambda,\mathscr{L})/_{[e_1=e_2]}\to (\Gamma,\mathscr{G})$
may not be Stallings. 
On the level of group theory, observe that
$(\Lambda,\mathscr{L})$ and $(\Lambda,\mathscr{L})/_{[e_1=e_2]}$
have the same fundamental group. Thus a necessary condition
for a factorization into folds would be that
$f_\sharp\colon\pi_1(\Lambda,\mathscr{L}) \to  \pi_1(\Gamma,\mathscr{G})$
is injective. This is already too restrictive a condition 
when studying subgroups of free groups, since we begin with only a generating set
and not a free basis.

Dunwoody \cite[Section 2]{Dunwoody}
remedies this problem by introducing 
what he calls \emph{vertex morphisms---}we prefer \emph{vertex collapses---}effectively
replacing $\stab(\tilde w)$ with its image in
$\stab(\tilde f(\tilde v_1))$.
Dunwoody's vertex morphisms have to do with \emph{collapsing.}
We will pause to develop tools for collapsing graphs of groups, and then
return to folding.

\section{Collapsing Graphs of Groups}

There are two kinds of collapsing we shall be interested in, corresponding
to the two ways in which a morphism of graphs of groups can fail to be Stallings.
The former is collapsing of \emph{subgraphs.} In terms of the graph of groups
$(\Gamma,\mathscr{G})$, this operation preserves $\pi_1(\Gamma,\mathscr{G})$
while altering the topology of $\Gamma$. 
The latter is collapsing of \emph{stabilizers.}
This operation preserves $\Gamma$, but alters the graph of groups structure 
$\mathscr{G}$,
which has the effect of replacing $\pi_1(\Gamma,\mathscr{G})$ with
its image under $f_\sharp$. 
In each case the quotient morphism $f\colon(\Gamma,\mathscr{G})
\to (\Gamma',\mathscr{G}')$ induces surjective maps $\tilde f$ and $f_\sharp$.

The results of this section are summarized as follows.
\begin{prop}\label{collapsing}
	Every morphism $f\colon(\Lambda,\mathscr{L}) \to (\Gamma,\mathscr{G})$
	factors into a composition 
	of collapse moves followed by a Stallings morphism.
	The collapse moves proceed by
	collapsing each component of each preimage $f^{-1}(v)$
	as $v$ varies over the vertices of $\Gamma$,
	followed by collapsing the resulting morphism 
	to $(\Gamma,\mathscr{G})$ onto its image.  
	\hfill\qedsymbol
\end{prop}

We remark that if $\Lambda$ is not a finite graph,
then the first sequence of collapses need not be finite.
The collapses nonetheless form a directed system,
and one could define a direct limit of these collapses.
We are primarily interested in morphisms 
where $\Lambda$ is finite.

\paragraph{Collapsing subgraphs.} Let $(\Gamma,\mathscr{G})$ be a graph of groups.
A \emph{subgraph of groups} is a pair $(\Lambda,\mathscr{G}|_{\Lambda})$,
where $\Lambda$ is a (nonempty) connected subgraph of $\Gamma$ and 
$\mathscr{G}|_\Lambda$ is the restriction of the graph of groups structure
$\mathscr{G}$ to $\Lambda$. In terms of abstract nonsense, the functor
$\mathscr{G}|_\Lambda$ is the restriction of $\mathscr{G}$ to
the full subcategory of $\Gamma$ determined by $\Lambda$. 
The inclusion $(\Lambda,\mathscr{G}|_\Lambda) \to (\Gamma,\mathscr{G})$
is an embedding. For convenience choose a basepoint in $\Lambda$.
Thus $\pi_1(\Lambda,\mathscr{G}|_\Lambda)$ becomes the subgroup of 
$\pi_1(\Gamma,\mathscr{G})$ consisting of homotopy classes of loops contained in
$(\Lambda,\mathscr{G}|_\Lambda)$. The graph of groups obtained by
\emph{collapsing the subgraph $\Lambda$} is $(\Gamma/\Lambda,\mathscr{G}/\Lambda)$,
where the graph of groups structure $\mathscr{G}/\Lambda$ is obtained from
$\mathscr{G}$ in the following way. For $x$ a vertex or an edge not contained in 
$\Lambda$, set $(\mathscr{G}/\Lambda)_x = \mathscr{G}_x$.
The vertex $\Lambda$ of $\Gamma/\Lambda$ has vertex group
$(\mathscr{G}/\Lambda)_\Lambda = \pi_1(\Lambda,\mathscr{G}|_\Lambda)$.
If an oriented edge $e$ and its terminal vertex $\tau(e)$ both lie outside 
$\Lambda$, the monomorphism $\iota_e$ remains the same in $\mathscr{G}/\Lambda$.
If $\tau(e) \in \Lambda$, then $\iota_e$ becomes the composition
$\begin{tikzcd}\mathscr{G}_e \ar[r,"\iota_e"] & \mathscr{G}_{\tau(e)}
\ar[r,hook] & \pi_1(\Lambda,\mathscr{G}|_\Lambda).\end{tikzcd}$

Alternatively, this process can be seen in the Bass--Serre tree.
By \Cref{diagramfill1}, the inclusion 
$(\Lambda,\mathscr{G}|_\Lambda) \to (\Gamma,\mathscr{G})$
induces an embedding of Bass--Serre trees $\tilde\Lambda \to \tilde\Gamma$.
The image of $\tilde\Lambda$ is the component of the full preimage of 
$\Lambda$ in $\tilde\Gamma$ containing the lift of the basepoint.
Let $\tilde\Gamma/\tilde\Lambda$ be the graph obtained by collapsing
each component of the $\pi_1(\Gamma,\mathscr{G})$-orbit of $\tilde\Lambda$
to a vertex. Thus $\tilde\Gamma/\tilde\Lambda$ naturally inherits
an action of $\pi_1(\Gamma,\mathscr{G})$, and 
$\pi_1(\Lambda,\mathscr{G}|_\Lambda)$ stabilizes the image of $\tilde\Lambda$.

In fact, these two descriptions are equivalent.

\begin{prop} In the notation above, $\pi_1(\Gamma/\Lambda,\mathscr{G}/\Lambda) = 
	\pi_1(\Gamma,\mathscr{G})$, and $\tilde\Gamma/\tilde\Lambda$ is the
	Bass--Serre tree, and the following diagram of covering projections
	and quotient maps described above commutes
	\[\begin{tikzcd}
		\tilde\Gamma \ar[r]\ar[d] & \tilde\Gamma/\tilde\Lambda \ar[d] \\
		(\Gamma,\mathscr{G}) \ar[r] & (\Gamma/\Lambda,\mathscr{G}/\Lambda).
	\end{tikzcd}\]
\end{prop}

\begin{proof} Stallings's neat homological proof \cite[Theorem 4]{StallingsGTrees} 
	of Chiswell's result above applies to show that $\tilde\Gamma/\tilde\Lambda$
	is also a tree. The idea is that a tree is a graph with the homology 
	of a point, so there is a short exact sequence of reduced 
	simplicial chain groups for $\tilde\Gamma$ 
	(which are $\pi_1(\Gamma,\mathscr{G})$-modules). 
	The $\pi_1(\Gamma,\mathscr{G})$-modules generated
	by the images of $1$- and $0$-chains of $\tilde\Lambda$ in $\tilde\Gamma$
	are isomorphic and each equal to the corresponding 
	kernel of the chain map induced by the
	quotient $\tilde\Gamma\to\tilde\Gamma/\tilde\Lambda$.
	That $\tilde\Gamma/\tilde\Lambda$ is a tree follows from the nine lemma.

	By definition, an element $g \in \pi_1(\Gamma,\mathscr{G})$ 
	stabilizes the image
	of $\tilde\Lambda$ in $\tilde\Gamma/\tilde\Lambda$ if and only if 
	$g$ preserves $\tilde\Lambda$ in its action on $\tilde \Gamma$.
	But then $g$ has a representative loop entirely contained in
	$(\Lambda,\mathscr{G}|_\Lambda)$, so $g$ belongs to the image of
	$\pi_1(\Lambda,\mathscr{G}|_\Lambda)$. It follows that
	$\tilde\Gamma/\tilde\Lambda$ is the Bass--Serre tree for
	$(\Gamma/\Lambda,\mathscr{G}/\Lambda)$.
\end{proof}

More generally one can collapse a disconnected subgraph of groups
by collapsing each component.

\paragraph{Collapsing stabilizers.}
Suppose $f\colon (\Lambda,\mathscr{L}) \to (\Gamma,\mathscr{G})$ is a morphism
which does not collapse edges.
We can produce a new graph of groups structure on $\Lambda$
by \emph{collapsing $\mathscr{L}$ to its image $f(\mathscr{L})$}
in the following way.
If $x$ is a vertex or an edge of $\Lambda$,
define $f(\mathscr{L})_x = f_x(\mathscr{L}_x) \le \mathscr{G}_{f(x)}$.
Let $e$ be an edge of $\Lambda$ with $\tau(e)=v$.
We define the monomorphism 
$\iota_e\colon f(\mathscr{L})_e \to f(\mathscr{L})_v$
as 
\[\begin{tikzcd} f(\mathscr{L})_e \ar[r,hook]
	& \mathscr{G}_{f(e)} \ar[r,"\iota_{f(e)}"]
	& \mathscr{G}_{f(v)} \ar[r,"\ad(\delta_e)"]
	& \mathscr{G}_{f(v)},
\end{tikzcd}\]
where $\delta_e \in \mathscr{G}_{f(v)}$ is part of the data of $f$.
Observe that the pseudonaturality condition of \Cref{morphisms} guarantees
that the image of the above monomorphism lies in $f(\mathscr{G})_v$.

The morphism $f$ factors as follows
\[\begin{tikzcd}[row sep = tiny]
	(\Lambda,\mathscr{L}) \ar[rr,"f"]\ar[rd,swap,"f_*"]
	& & (\Gamma,\mathscr{G}), \\
	& (\Lambda,f(\mathscr{L})) \ar[ur]
\end{tikzcd}\]
where the morphism
$f_\ast\colon (\Lambda,\mathscr{L}) \to (\Lambda,f(\mathscr{L}))$
is the identity on $\Lambda$
and acts as $f$ on vertex and edge groups.
By definition, 
if $e$ is an edge with $\tau(e) = v$,
we have $\iota_e(f_\ast)_e = (f_\ast)_v\iota_e$,
so the conjugating element $\delta_e$ for $f_\ast$ is trivial.
The morphism $(\Lambda,f(\mathscr{L})) \to (\Gamma,\mathscr{G})$
is Stallings.
It acts as the natural inclusion on vertex and edge groups,
as $f$ on $\Lambda$,
and has the same conjugating elements $\delta_e$ as $f$.
The fundamental group $\pi_1(\Lambda,f(\mathscr{L}))$ 
is a quotient of $\pi_1(\Lambda,\mathscr{L})$.

The Bass--Serre tree $\tilde\Lambda/f$ for $(\Lambda,f(\mathscr{L}))$
is built from the same ``puzzle piece,'' 
in the language of \Cref{construction},
as $\tilde\Lambda$,
only the attaching differs.
Recall that the stabilizer of a vertex $\tilde v\in \tilde\Lambda$
is canonically isomorphic to $\mathscr{L}_v$ and
permutes the set of puzzle pieces containing $\tilde v$.
Every vertex has a preferred puzzle piece
corresponding to the identity element of its stabilizer.
One can give an inductive description of the morphism 
$\tilde f_\ast\colon\tilde\Lambda \to \tilde\Lambda/f$
as follows:
the preferred puzzle piece containing the lift of the basepoint
in $\tilde\Lambda$ 
is sent via an isomorphism to
the preferred puzzle piece in $\tilde\Lambda/f$
containing the lift of the basepoint.
If $\tilde f_\ast(\tilde v)$ is already defined, extend $\tilde f_\ast$
to the set of puzzle pieces containing $\tilde v$  
by sending the preferred puzzle piece $P$ for $\tilde v$
to the preferred puzzle piece for $\tilde f_\ast(\tilde v)$,
and then using the identification of the stabilizer of $\tilde v$
with $\mathscr{L}_v$, 
sending $g.P$ to $(f_\ast)_v(g).\tilde f_\ast(P)$
for $g \in \stab(\tilde v)$.

\paragraph{}
Dunwoody's vertex morphisms,
which we will call \emph{vertex collapses} to avoid  confusion
are a special case 
of collapsing a graph of groups structure onto its image 
that arises when folding.
Let $v$ be a vertex of $\Lambda$, 
and let $K_v$ be a normal subgroup of $\mathscr{L}_v$
such that for each oriented edge $e \in \st(v)$,
$\iota_e(\mathscr{L}_e) \cap K_v = 1$.
We may define a new graph of groups structure $\mathscr{L}/K_v$
on $\Lambda$ by setting  $(\mathscr{L}/K_v)_v = \mathscr{L}_v/K_v$,
and setting $(\mathscr{L}/K_v)_x = \mathscr{L}_x$ otherwise.
By assumption, the morphism 
\[\begin{tikzcd}
	\mathscr{L}_e \ar[r,"\iota_e"] 
	& \mathscr{L}_v \ar[r,two heads] 
	& \mathscr{L}_v/K_v
\end{tikzcd}\]
is a monomorphism for each oriented edge $e$ in $\st(v)$.
The evident morphism $f\colon (\Lambda,\mathscr{L}) \to (\Lambda,\mathscr{L}/K_v)$,
is equal to the result of collapsing $\mathscr{L}$ to its image under $f$.
We call $f$ a \emph{vertex collapse at $v$.}

\section{Folding Graphs of Groups}
We return to our description of folding.
Let $(\Lambda,\mathscr{L})$ be a graph of groups and let
$\tilde\Lambda/_{[\tilde e_1 = \tilde e_2]}$ 
be the result of folding a pair of edges 
$\tilde e_1$ and $\tilde e_2$ in $\tilde\Lambda$.
Let $\tilde v_i$ be the initial vertex of $\tilde e_i$,
and let $e_i$ and $v_i$
be the corresponding edge and vertex of $\Lambda$.
Let $e$ and $v$ be the images of the $e_i$ and $v_i$,
respectively, in $\Lambda/_{[e_1=e_2]}$.
Bestvina and Feighn \cite[Section 2]{BestvinaFeighnGTrees}
distinguish several types of folds
by their effect on $(\Lambda,\mathscr{L})$.
There are three types. 
Bestvina and Feighn distinguish these further into subtypes
based primarily on whether $\tau(e_1)$ is distinct from the $v_i$.
We will not draw such a distinction.

Throughout, let $f\colon (\Lambda,\mathscr{L}) \to (\Gamma,\mathscr{G})$
be a Stallings morphism that is not an immersion at $v$.
Recall that this means that the map
\[ f_{\st(v)}\colon\coprod_{e\in\st(v)} \mathscr{L}_v/\iota_e(\mathscr{L}_e)\times\{e\}
\to \coprod_{e'\in\st(f(v))} \mathscr{G}_{f(v)}/\iota_{e'}(\mathscr{G}_{e'})\times\{e'\}\]
\[([g],e) \mapsto ([f_v(g)\delta_e],f(e))\]
is not injective.

\usetikzlibrary{decorations.pathmorphing, shapes, quotes}
\tikzset{->-/.style = {
		very thick,
		decoration = {
			markings,
			mark = at position .5 with {\arrow{>}}
		},
		postaction = {decorate}
	},
	pt/.style = {
		circle,
		fill = black,
		scale = 0.5,
	},
	-<-/.style = {
		very thick,
		decoration = {
			markings,
			mark = at position .5 with {\arrow{<}}
		},
		postaction = {decorate}
	}
}
\paragraph{Type I.}
In this case, 
$e_1 \ne e_2$,
and $v_1 \ne v_2$.
In other words, 
the orbits of the $\tilde e_i$ and $\tilde v_i$
are distinct.
The stabilizer of the image of $\tilde e_1$
is the subgroup generated 
by the stabilizers of $\tilde e_1$ and $\tilde e_2$,
and similarly for the stabilizer of the image of $\tilde v_1$.
The situation in the quotient graph of groups
is described in \Cref{foldI}.
By $\langle\mathscr{L}_{v_1},\mathscr{L}_{v_2}\rangle$,
for instance, 
we mean the subgroup of $\pi_1(\Lambda,\mathscr{L})$
generated by $\mathscr{L}_{v_1}$ and $\mathscr{L}_{v_2}$,
under their identification with the stabilizers of
$\tilde v_1$ and $\tilde v_2$.
Note that we do not assume 
that $w=\tau(e_1)$ is distinct
from $v_1$ or $v_2$,
even though we have drawn it as distinct in the figure.

We briefly describe the resulting morphism
$\hat f\colon (\Lambda,\mathscr{L})/_{[e_1 = e_2]} \to (\Gamma,\mathscr{G})$.
Certainly as a morphism of graphs, it is clear how to define $\hat f$.
If $x$ is a vertex or edge of $\Lambda/_{[e_1=e_2]}$,
define $\hat f_x = f_x$.
The morphism 
$\hat f_v\colon\langle\mathscr{L}_{v_1},\mathscr{L}_{v_2}\rangle 
\to \mathscr{G}_{\hat f(v)}$
must be the restriction of $f_\sharp$ to
$\langle\mathscr{L}_{v_1},\mathscr{L}_{v_2}\rangle$.
Suppose first that as part of the data of $f$, we have
$\delta_{e_1} = \delta_{e_2}$ and $\delta_{\bar e_1} = \delta_{\bar e_2}$.
Observe that the pseudonaturality condition says that
if $h \in \mathscr{L}_{e_i}$ for $i = 1,2$, we have that
\[ \delta_{\bar e_i}^{-1}f_{v_i}\iota_{\bar e_i}(h)\delta_{\bar e_i} 
\in \iota_{f(e_i)}(\mathscr{G}_{f(e_i)}). \]
So if $\delta_{\bar e_1} = \delta_{\bar e_2}$,
we may use the above to define $\hat f_e$,
and check that the pseudonaturality condition holds with
$\delta_e = \delta_{e_1}$ and $\delta_{\bar e} = \delta_{\bar e_1}$.

If it is not the case that $\delta_{e_1} = \delta_{e_2}$
and $\delta_{\bar e_1} = \delta_{\bar e_2}$,
we alter the situation.
If $\delta_{e_1} \ne \delta_{e_2}$, 
that $f_{\st(v)}$ is not injective implies that
there exist $g \in \mathscr{L}_v$ and $h \in \mathscr{G}_{f(e)}$
such that $f_v(g)\delta_{e_2}\iota_{f(e)}(h) = \delta_{e_1}$.
We may pass from $f$ to an equivalent morphism as in \Cref{equivalence}
so that $\delta'_{e_2} = \delta_{e_2}\iota_{f(e)}(h)$.
Finally, we may twist $e_2$ by $g$ so that $\delta''_{e_2} = \delta_{e_1}$.
If now $\delta_{\bar e_1} \ne \delta_{\bar e_2}$,
observe that at least one of the $v_i$ is distinct from $w$,
so we may pass from $f$ to an equivalent morphism 
so that $\delta_{\bar e_1} = \delta_{\bar e_2}$.
All of these operations are demonstrated in the examples in the following section.

\begin{figure}[ht]
	\begin{center}
		\begin{tikzpicture}[auto]
		\node[pt,label = 180:$\mathscr{L}_w$] (w) at (0,0) {};
		\node[pt,label = {
				[xshift=5pt]:$\mathscr{L}_{v_1}$
			}] (v1) at (2,1) {};
		\node[pt,label = {
				[xshift=5pt,yshift=-2pt]below:$\mathscr{L}_{v_2}$
			}] (v2) at (2,-1) {};
		\draw[->-,swap] (v1) to node {$\mathscr{L}_{e_1}$} (w);
		\draw[->-] (v2) to node {$\mathscr{L}_{e_2}$} (w);
		\draw[->,
			decorate,
			decoration = {
				snake,
				pre length = 10mm,
				post length = 1mm
			}] (2.2,0) -- (4,0);
		\node[pt,label = 105:$\mathscr{L}_w$] (ww) at (5,0) {};
		\node[pt,label = {
				[xshift=5pt]:{$\langle\mathscr{L}_{v_1},\mathscr{L}_{v_2}\rangle$}
			}]
			(vv) at (7.5,0) {};
		\draw[->-] (vv) 
			to node {$\langle\mathscr{L}_{e_1},\mathscr{L}_{e_2}\rangle$}
			(ww);
	\end{tikzpicture}
	\end{center}
	\caption{The effect of a Type I fold.}
	\label{foldI}
\end{figure}

\paragraph{Type II.}
In this case, $e_1 = e_2$, so $v_1 = v_2$. In other words,
there is some $g \in \pi_1(\Lambda,\mathscr{L})$ with
$g.\tilde e_1 = \tilde e_2$.
Thus $g$ belongs to $\mathscr{L}_w$,
where $w = \tau(e_1) = \tau(e_2)$.
After folding,
$g$ will stabilize the image of $\tilde e_1$,
and thus also stabilize 
the image of $\tilde v_1$.
The situation in the quotient graph of groups 
is described in \Cref{foldII}.
Again, we do not assume 
that $w$ is distinct from $v_1$.

More generally, one might allow Type II folds 
to encompass the more general situation where $H$
is a subgroup of $\mathscr{L}_w$
and we fold $\tilde e_1$ with $h.\tilde e_1$ 
for all $h \in H$.
If $H$ happens to be finitely generated,
this is a finite composition of Type II folds in the stricter sense.

\begin{figure}[ht]
	\begin{center}
		\begin{tikzpicture}[auto]
		\node[pt,label = {
				[xshift=-5pt]:$\mathscr{L}_{w}$
			}] (w) at (0,0) {};
		\node[pt,label = {
				[xshift=5pt]:$\mathscr{L}_{v_1}$
			}] (v1) at (2.5,0) {};
		\draw[->-] (v1) to node {$\mathscr{L}_{e_1}$} (w);
		\draw[->,
			decorate,
			decoration = {
				snake,
				pre length = 7mm,
				post length = 1mm
			}] (3.2,0) -- (4.5,0);
		\node[pt,label = {
				[xshift=-5pt]:$\mathscr{L}_w$
			}] (ww) at (5.5,0) {};
		\node[pt,label = {
				[xshift=5pt]:{$\langle\mathscr{L}_{v_1},g\rangle$}
			}]
			(vv) at (8,0) {};
		\draw[->-] (vv) 
			to node {$\langle\mathscr{L}_{e_1},g\rangle$}
			(ww);
	\end{tikzpicture}
	\end{center}
	\caption{The effect of a Type II fold.}
	\label{foldII}
\end{figure}

\paragraph{Type III.}
In this case, $e_1 \ne e_2$, but $v_1 = v_2$.
In other words, there is some $g \in \pi_1(\Lambda,\mathscr{L})$
with $g.\tilde v_1 = \tilde v_2$, 
but $\tilde e_1$ and $\tilde e_2$  
are in distinct orbits.
After folding, $g$ will stabilize the image of $\tilde v_1$.

The factoring of the morphism $f$ is broadly the same as in a Type I fold.
However, since $v_1 = v_2$, we cannot always arrange that
$\delta_{\bar e_1} = \delta_{\bar e_2}$.
So set $\delta_{\bar e} = \delta{\bar e_1}$.
The image of $g$ in $\mathscr{G}_{f(v)}$ 
will be the difference $\delta_{\bar e_1}^{-1}\delta_{\bar e_2}$.

\begin{figure}[ht]
	\begin{center}
		\begin{tikzpicture}[auto]
		\node[pt,label = {
				[xshift=-5pt]:$\mathscr{L}_{w}$
			}] (w) at (0,0) {};
		\node[pt,label = {
				[xshift=5pt]:$\mathscr{L}_{v_1}$
			}] (v1) at (2.5,0) {};
		\draw[->-,bend left] (v1) to node {$\mathscr{L}_{e_1}$} (w);
		\draw[->-,bend right,swap] (v1) to node {$\mathscr{L}_{e_2}$} (w);
		\draw[->,
			decorate,
			decoration = {
				snake,
				pre length = 7mm,
				post length = 1mm
			}] (3.2,0) -- (4.5,0);
		\node[pt,label = {
				[xshift=-5pt]:$\mathscr{L}_w$
			}] (ww) at (5.5,0) {};
		\node[pt,label = {
				[xshift=5pt]:{$\langle\mathscr{L}_{v_1},g\rangle$}
			}]
			(vv) at (8,0) {};
		\draw[->-] (vv) 
			to node {$\langle\mathscr{L}_{e_1},\mathscr{L}_{e_2}\rangle$}
			(ww);
	\end{tikzpicture}
	\end{center}
	\caption{The effect of a Type III fold.}
	\label{foldIII}
\end{figure}

\begin{rk}
	Bestvina and Feighn define an additional type of fold, 
	Type IIIC, 
	which occurs when there exists $g \in \pi_1(\Lambda,\mathscr{L})$
	with $g.\tilde e_1 = \bar{\tilde e}_2$. 
	In this case, 
	$e_1 = e_2$  is a loop,
	and $g$ acts as an inversion of $\tilde\Lambda/_{[\tilde e_1 = \tilde e_2]}$.
	This type of fold does not occur in factorizations
	of a Stallings morphism 
	$f\colon (\Lambda,\mathscr{L}) \to (\Gamma,\mathscr{G})$.
	The reason is that in this case $f_\sharp(g)$ would act
	as an inversion on $\tilde \Gamma$,
	which we do not allow.
\end{rk}

The best unconditional result is the following proposition,
due to Stallings \cite[Theorem 5]{StallingsGTrees}, 
Bestvina--Feighn \cite[Proposition, Section 2]{BestvinaFeighnGTrees}
and Dunwoody \cite[Theorem 2.1]{Dunwoody}.

\begin{prop}\label{Stallingsfolding}
	Let $f \colon (\Lambda,\mathscr{L}) \to (\Gamma,\mathscr{G})$
	be a Stallings morphism.
	If $f$ is not an immersion,
	then it admits a fold followed by a vertex collapse.
	There is a factorization
	\[\begin{tikzcd}[row sep = tiny]
		(\Lambda,\mathscr{L}) \ar[rr,"f"]\ar[rd,"e"]
		& & (\Gamma,\mathscr{G}), \\
		& (H,\mathscr{H}) \ar[ur,"m"] & 
	\end{tikzcd}\]
	where $m$ is an immersion and each of
	$e\colon \Lambda \to H$,
	$\tilde e\colon \tilde\Lambda \to \tilde H$
	and $e_\sharp\colon \pi_1(\Lambda,\mathscr{L}) \to \pi_1(H,\mathscr{H})$
	are surjective.
\end{prop}

\begin{proof}
	Observe that if 
	$f\colon(\Lambda,\mathscr{L}) \to (\Gamma,\mathscr{G})$
	is not an immersion,
	then it factors through some fold
	$(\Lambda,\mathscr{L})/_{[e_1=e_2]}$.
	All the vertex and edge groups of $\mathscr{L}/_{[e_1=e_2]}$
	are equal to those of $\mathscr{L}$ save 
	one vertex group and one edge group.
	Call the new edge group $H_e$ and the new vertex group $H_v$.
	Observe that the subgroups that combine to form $H_e$
	belong to some vertex group $\mathscr{L}_w$.
	Dunwoody's insight is that thus the resulting morphism
	$f'\colon (\Lambda,\mathscr{L})/_{[e_1 = e_2]} \to (\Gamma,\mathscr{G})$
	is Stallings everywhere except possibly at $H_v$,
	so we may perform a vertex collapse.

	Nothing prevents us from iterating this process.
	Unfortunately, without further assumptions,
	nothing can prevent the process from continuing indefinitely.
	Stallings observes that one may continue to fold transfinitely
	and pass to direct limits when necessary. 
	At some possibly transfinite stage,
	the resulting graph of groups $(H,\mathscr{H})$ admits no further folds,
	so the map $m\colon (H,\mathscr{H}) \to (\Gamma,\mathscr{G})$ is an immersion.
	Each fold and vertex collapse induces surjective maps
	of graphs, Bass--Serre trees and fundamental groups.
	Thus the same holds true for the resulting map
	$e\colon (\Lambda,\mathscr{L}) \to (H,\mathscr{H})$.
\end{proof}

Note that if $(\Lambda,\mathscr{L})/_{[e_1=e_2]}$ is obtained from 
the finite graph of groups $(\Lambda,\mathscr{L})$ 
by a Type I or Type III fold,
then $\Lambda/_{[e_1=e_2]}$ has fewer edges than $\Lambda$.
Thus, if the process of foldings and vertex collapses 
does \emph{not} terminate, after some finite stage,
all further folds are of Type II.

Kapovich, Weidmann and Miasnikov 
give a simple example \cite[Example 5.10]{KapovichWeidmannMiasnikov}
of an HNN extension $G = F_2*_{F_2}$ corresponding to
an injective non-surjective endomorphism of a free group of rank two,
and a two-generated subgroup $H \le G$
for which the sequence of Type II folds
for the associated morphism of graphs of groups does not terminate.

Type II folds replace some edge group $\mathscr{L}_e$,
which we may think of via $f$ as a subgroup of $\mathscr{G}_{f(e)}$,
with some strictly larger subgroup.
Thus to guarantee that the process terminates after finitely many iterations,
we need to argue that each edge group 
will only be increased finitely many times.

An obvious sufficient condition is that for each edge, the homomorphism
\[f_e\colon \mathscr{L}_e \to \mathscr{G}_{f(e)}\]
is an isomorphism---in fact,
this would guarantee that only Type I and Type III folds occur.
More useful for our purposes would be the condition that $f_\sharp$ is surjective
and that the edge groups $\mathscr{G}_{f(e)}$ are finitely generated.
This is the condition that Dunwoody uses in \cite[Theorem 2.1]{Dunwoody},
modeled on the analogous condition in 
\cite[Proposition, Section 2]{BestvinaFeighnGTrees}.
Another possible sufficient condition 
would be that the edge groups of $(\Gamma,\mathscr{G})$
are \emph{Noetherian,} that is,
every ascending chain \[G_0 \le G_1 \le G_2 \le \dotsb\] of subgroups stabilizes.
For groups, the Noetherian property is clearly equivalent to requiring that
every subgroup is finitely generated.
This latter property of a group is sometimes called \emph{slender.}

The main family of examples of Noetherian groups 
are \emph{virtually polycyclic} groups.
A polycyclic group is a group $G$ with a (finite) subnormal series
\[ 1 = G_0 \triangleleft G_1 \triangleleft \dotsb \triangleleft G_k = G \]
whose factor groups $G_i/G_{i-1}$ are cyclic.
Examples of virtually polycyclic groups include finite groups,
finitely generated abelian groups,
and virtually solvable subgroups of $\gl_n(\mathbb{Z})$.
Auslander \cite{Auslander} and Swan \cite{Swan} show
that in fact any virtually polycyclic group is a
virtually solvable subgroup of $\gl_n(\mathbb{Z})$ for some $n$.
As such, virtually polycyclic groups have good computational properties.
For this reason,
Kapovich, Weidmann and Miasnikov restrict their algorithms 
to graphs of groups $(\Gamma,\mathscr{G})$ all of whose edge groups are
virtually polycyclic. 

Thus we have the following corollary of \Cref{Stallingsfolding}.

\begin{cor}\label{Stallingspracticalfolding}
	If $f\colon (\Lambda,\mathscr{L}) \to (\Gamma,\mathscr{G})$
	is a Stallings morphism
	where $\Lambda$ is a finite graph and either
	\begin{enumerate*}[label=(\roman*)]
	\item each edge group of $(\Gamma,\mathscr{G})$ is Noetherian, or
	\item each edge group of $\mathscr{G}$ is finitely generated
		and $f_\sharp\colon \pi_1(\Lambda,\mathscr{L}) \to \pi_1(\Gamma,\mathscr{G})$
		is surjective,
	\end{enumerate*}
	then $f$ factors into a finite sequence of folds and vertex collapses
	followed by an immersion.
	\hfill\qedsymbol
\end{cor}

\section{Examples}
Let $C_n$ denote the cyclic group  of order $n$.
Consider the following group
\[ G = \langle a,b,c,t \mid [a,b] = a^2 = b^2 = c^3 = 1,\ tbt^{-1} = a \rangle
= (C_2^2*_{C_2})*C_3.\]
We have $\langle a,b \rangle \cong C_2^2$,
and $\langle c\rangle \cong  C_3$.
The group $G$
is the fundamental group of a graph of groups $(\Gamma,\mathscr{G})$
as described in \Cref{exampleG}.
Here $x$ denotes the generator of $\mathscr{G}_{e_1} = C_2$.

\begin{figure}[ht]
	\begin{center}
		\begin{tikzpicture}[auto,commutative diagrams/every diagram]
			\node[pt,
				label = $v_1$]
				(v) at (0,0) {};
			\node[pt, label = {$v_2$}] (w) at (2,0) {};
			\draw[->-] (v) to 
				[loop left,in = 150, out = -150, min distance = 25mm]
				node {$e_1$} (v);
			\draw[->-] (v) to node {$e_2$} (w);
			\node (x) at (4,0) {$C_2$};
			\node (ab) at (7,0) {$\langle a,b \rangle$};
			\node (1) at (9,0) {$1$};
			\node (c) at (11,0) {$\langle c \rangle$};
			\path[commutative diagrams/.cd, every arrow, every label,
				shift left=1mm]
				(x) edge
				node {$\iota_{e_1}\colon x\mapsto a$} (ab);
			\path[commutative diagrams/.cd, every arrow, every label]
				(1) edge (ab)
				(1) edge (c);
			\path[commutative diagrams/.cd, every arrow, every label,
				shift left=-1mm]
				(x) edge
				node[swap] {$\iota_{\bar e_1}\colon x\mapsto b$} (ab);
		\end{tikzpicture}
	\caption{The graph  $\Gamma$ and graph of groups structure $\mathscr{G}$.}
	\label{exampleG}
		\end{center}
\end{figure}

Base the fundamental group at $v_1$. 
Thus $t\in G$ corresponds to the loop $e_1$.
We will demonstrate folding by showing how to represent the subgroup
\[H = \langle a, ctb, ctc^{-1} \rangle\] as an immersion.
We will define a graph $\Lambda$ 
and a morphism $f\colon \Lambda \to (\Gamma,\mathscr{G})$
so that $f_\sharp(\pi_1(\Lambda,\ast)) = H$.
Since our generating set for $H$ has three elements,
$\Lambda$ will be a subdivision of a rose with three petals.
Note that because $v_1$ has nontrivial stabilizer,
some elements of $G$ are represented by loops of length zero.
We choose loops $\gamma_1$, $\gamma_2$ and $\gamma_3$
based at $v_1$ corresponding to the displayed generating set for $H$.
\[
	\gamma_i = \begin{dcases} a & i = 1 \\
		e_2c\bar e_2e_1b & i = 2 \\
		e_2c\bar e_2e_1e_2c^{-1}\bar e_2 & i = 3.
	\end{dcases}\]
To better see what is happening,
we will first subdivide
each edge of $\Lambda$ into three edges,
and then replace the middle edges with the paths $\gamma_i$, 
see \Cref{foldingexI}.

\begin{figure}[ht]
	\begin{center}
		\includegraphics[width=\textwidth]{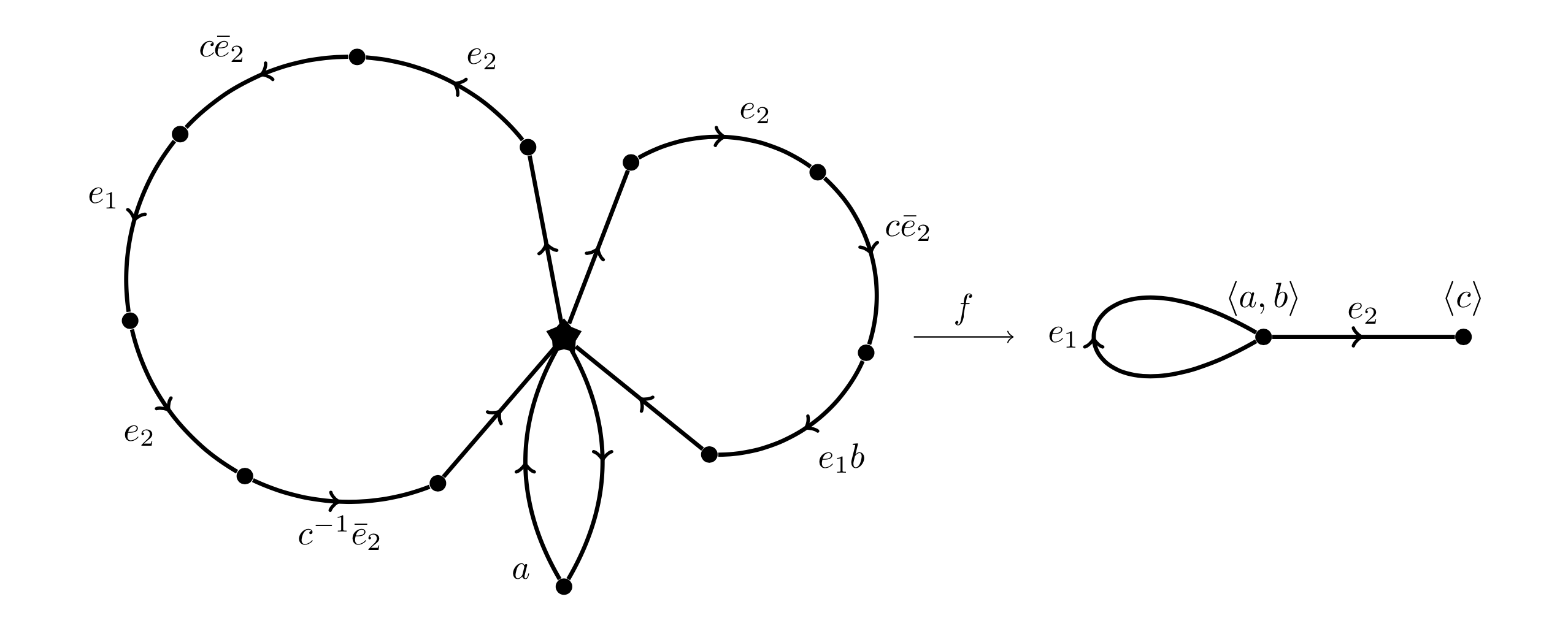}
		\caption{The morphism $f\colon \Lambda \to (\Gamma,\mathscr{G})$.}
		\label{foldingexI}
	\end{center}
\end{figure}

In \Cref{foldingexI} we have described the morphism $f$ 
by labelling $\Lambda$.
The basepoint (starred) and the unlabeled edges are collapsed to $v_1$.
The preferred orientation of each edge $e$ is indicated by the arrowhead.
The edge $e$ is labeled by its image 
$\delta_{\bar e}f(e)\delta_e^{-1}$ in the preferred orientation,
with $\delta_e = 1$ omitted. $\Lambda$ has fundamental group $F_3$.

The morphism $f$ is not Stallings,
so we collapse a subgraph,
through which $f$ factors as
$f_1 \colon (\Lambda^1,\mathscr{L}^1) \to (\Gamma,\mathscr{G})$.
In $\Lambda_1$,
the basepoint $\star$ has vertex group $\mathscr{L}_\star^1 = \mathbb{Z}$,
so $f_1$ is not Stallings.
We perform a vertex collapse at $\star$.
The resulting morphism 
$f_2\colon (\Lambda^2,\mathscr{L}^2) \to (\Gamma,\mathscr{G})$
is Stallings. See \Cref{foldingexII}.

\begin{figure}[ht]
	\begin{center}
		\includegraphics[width=\textwidth]{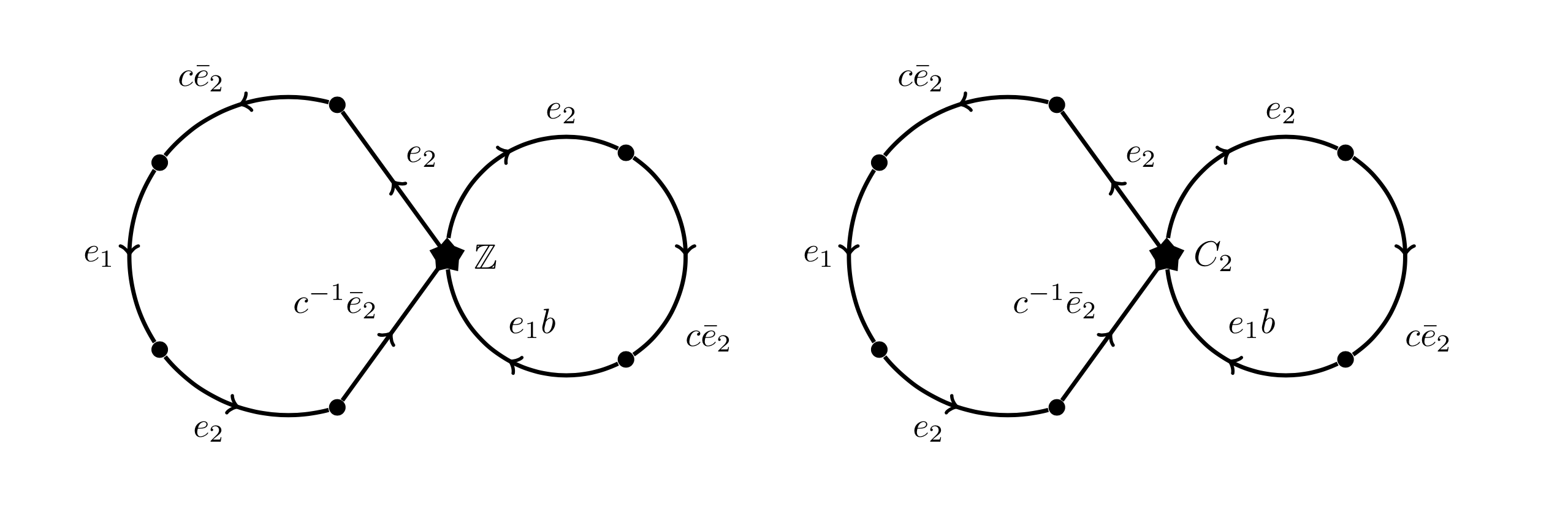}
		\caption{$(\Lambda^1,\mathscr{L}^2)$
		and $(\Lambda^2,\mathscr{L}^2)$.}
		\label{foldingexII}
	\end{center}
\end{figure}

The morphism $f_2$ is not an immersion at $\star$. 
We perform three Type I folds,
folding the left loop over the right loop.
In order to perform the third fold,
we first replace $f$ by an equivalent morphism
so that the path around the outer loop reads
\[e_2c\bar e_2e_1bb^{-1}e_2c^{-1}\bar e_2.\]
The resulting morphism $f_4$ is an immersion.
See \Cref{foldingexIII}.
We see that $H \cong F_2*C_2$.

\begin{figure}[ht]
	\begin{center}
		\includegraphics[width=\textwidth]{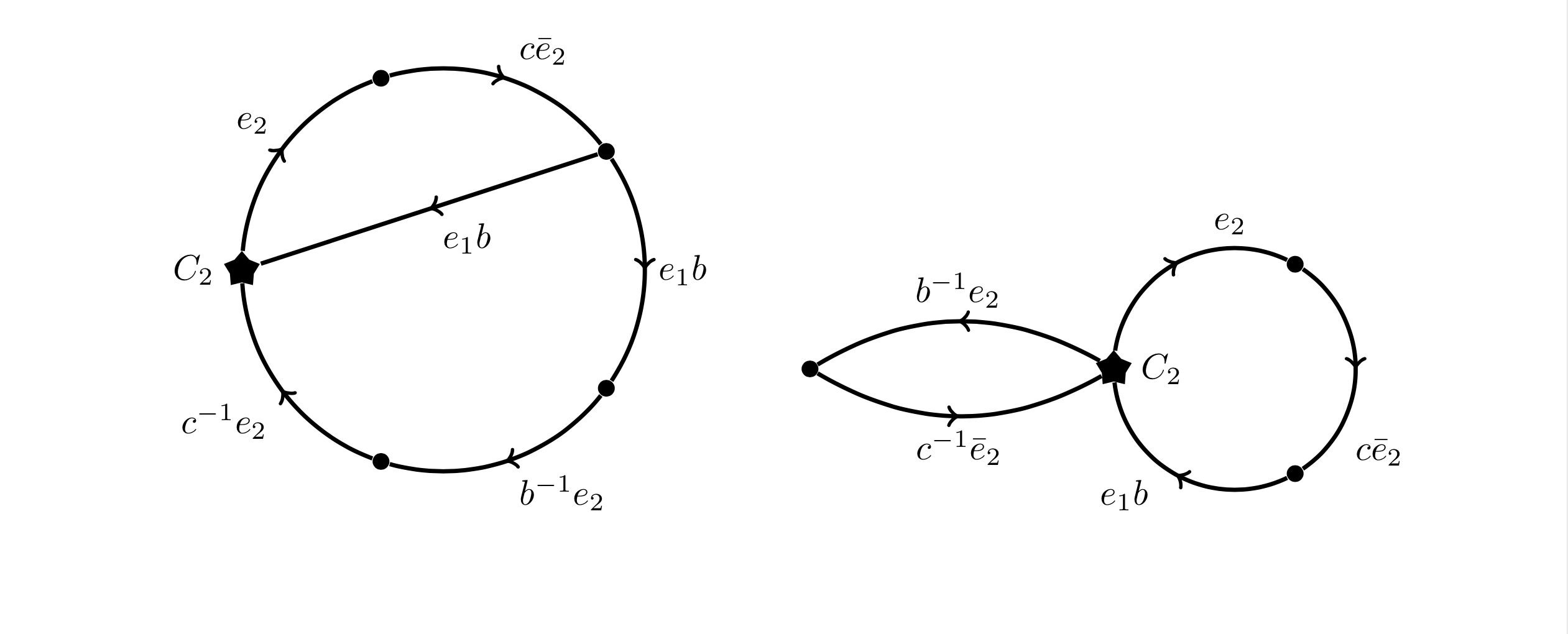}
		\caption{$(\Lambda^3,\mathscr{L}^3)$
		and $(\Lambda^4,\mathscr{L}^4)$.}
		\label{foldingexIII}
	\end{center}
\end{figure}

\paragraph{}
Suppose now that we begin with $H' = \langle a,ctc^{-1},cta \rangle$.
The beginning steps of the folding process are essentially identical,
so we omit them,
but the fourth map is no longer an immersion,
since $a^{-1}e_2$ and $e_2$ differ by an element of $\mathscr{L}'^4_\star$.
We may twist the edge labeled $a^{-1}e_2$ by $a$---or,
more properly speaking, by the element of $\mathscr{L}'^4_\star$
mapping to $a$---and 
then perform a Type III fold followed by a vertex collapse,
yielding the Stallings morphism $f'_5$.
See \Cref{foldingexIIII}.

From here there are two more Type I folds, 
and finally a Type II fold,
pulling the generator of $C_2$ across the loop labeled $e_1a$,
yielding $f'_6\colon (\Lambda'^6,\mathscr{L}'^6)$ as in
\Cref{foldingexIIII}.
Performing a final vertex collapse yields an isomorphism.

\begin{figure}[hb]
	\begin{center}
		\includegraphics[width=\textwidth]{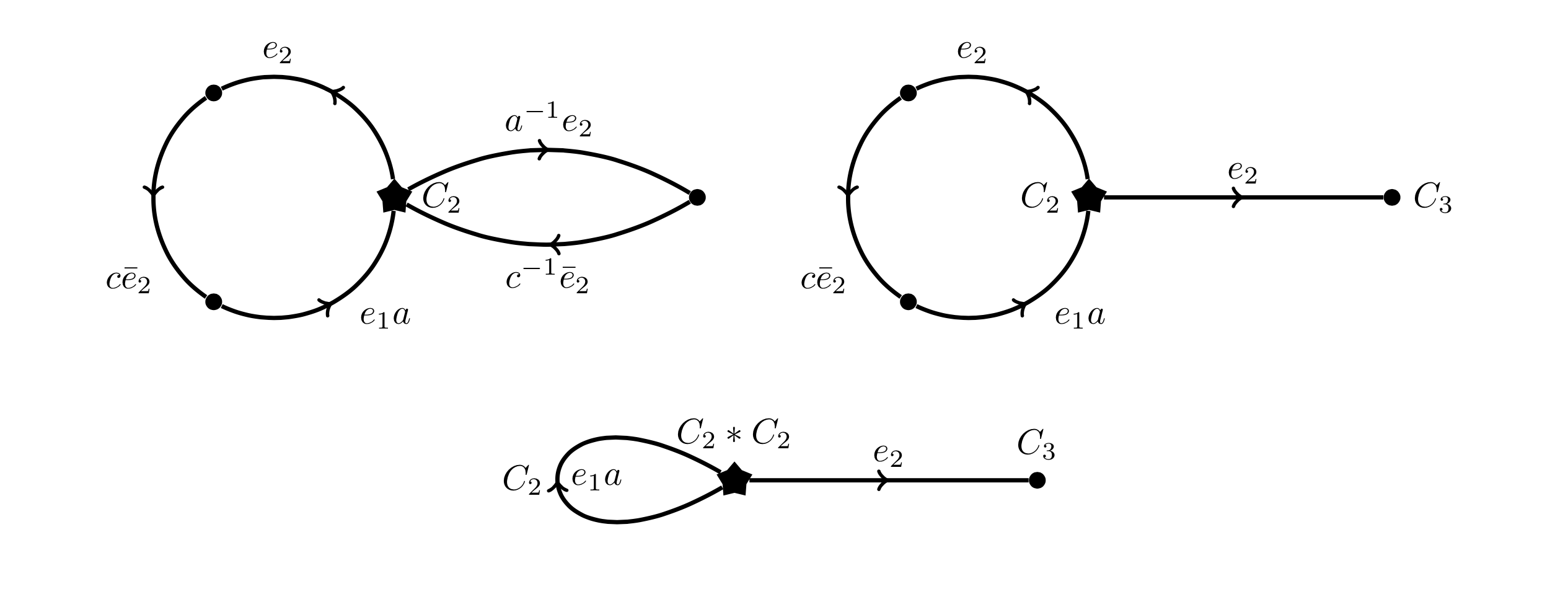}
		\caption{The graphs of groups 
			$(\Lambda'^4,\mathscr{L}'^4)$, $(\Lambda'^5,\mathscr{L}'^5)$
		and $(\Lambda'^6,\mathscr{L}'^6)$.}
		\label{foldingexIIII}
	\end{center}
\end{figure}

\section{Applications to Automorphism Groups}\label{applications}

Stallings remarks \cite{Stallings} that an automorphism of a free group can 
be thought of as a subdivision of a graph followed by a sequence of folds.
Bestvina and Handel \cite{BestvinaHandel}
developed this insight to great effect,
giving an analogue of the Nielsen--Thurston classification of
surface diffeomorphisms for homotopy equivalences of graphs.
We build on this work in the following chapter.

In this section we will use folding to give a conceptually simple proof 
of a theorem of Fouxe-Rabinovitch \cite{FouxeRabinovitch}.

\begin{thm}\label{fouxerabinovitch}
	Let $G = A_1*\dotsb*A_n*F_k$ be a finitely generated free product,
	where the $A_i$ are freely indecomposable and not infinite cyclic,
	and $F_k$ is a free group of rank $k$.
	The group $\aut(G)$ is generated by \emph{partial conjugations,}
	\emph{transvections,} 
	the $\aut(A_i)$,
	and a finite group $\Sigma_G$.
	In particular, $\aut(G)$ is finitely generated if each $\aut(A_i)$ is.
\end{thm}

Actually, Fouxe-Rabinovitch, building on earlier work of Nielsen \cite{Nielsen},
gave a full presentation of $\aut(G)$.
This is a finite presentation 
if each $A_i$ and each $\aut(A_i)$ is finitely presented.

Let $G$ be a free product as in the statement, and
fix a free basis $F_k = \langle x_1,\dotsc, x_k\rangle$.
For convenience, we write $A_{n+i} = \langle x_i \rangle$ for $1 \le i \le k$.
Let $i$ and $j$ be distinct integers with $1 \le i,j \le n+k$,
and choose $g_i \in A_i$.
The \emph{partial conjugation} of $A_j$ by $g_i$ is the automorphism
determined by its action on the free factors as
\[ \chi_{g_i,j} \begin{dcases}
	g_j \mapsto g_i^{-1}g_jg_i & g_j \in A_j \\
	g_\ell \mapsto g_\ell & g_\ell \in A_\ell,\ \ell \ne j,\ 1\le \ell \le n+k.
\end{dcases}\]

Now let $i$ and $j$ be integers with $1 \le i \le n+k$ and $1 \le j \le k$
such that $i \ne n+j$. Choose $g_i \in A_i$.
The \emph{left} and \emph{right transvections of $x_j$ by $g_i$}
are the automorphisms determined by their actions on the free factors as
\begin{gather*}
	\lambda_{g_i,j}\begin{dcases}
	x_j \mapsto g_ix_j \\
	g_\ell \mapsto g_\ell & g_\ell\in A_\ell,\ \ell \ne n+j,\ 1 \le \ell \le n+k,
\end{dcases} \\
\shortintertext{ and }
\rho_{g_i,j} \begin{dcases}
	x_j \mapsto x_jg_i \\
	g_\ell \mapsto g_\ell & g_\ell \in A_\ell,\ \ell \ne n+j,\ 1\le\ell\le n+k,
\end{dcases}
\end{gather*}
respectively. In the case of the free group $G = F_k$, 
the left and right transvections $\lambda_{x_i,j}$ and $\rho_{x_i,j}$
are called \emph{Nielsen transformations.}
These, together with a finite group $\Sigma_{F_k}$ 
that permutes and inverts the $x_i$,
are the standard generators for $\aut(F_k)$.
In general, the group $\Sigma_G$ is the product of
$\Sigma_{F_k}$ with a finite group that permutes 
those $A_i$ that are isomorphic.
Obviously this requires fixing some choice of isomorphism between them.

Our method of proof appears to be well-known 
in the case where $G = F_k$ is a free group;
we learned it from course notes of Bestvina \cite{BestvinaCourse}.
The idea is to follow up on Stallings's remark:
given an automorphism $\Phi\colon F_k \to F_k$,
construct a \emph{topological realization} 
$f\colon \Gamma \to \Gamma$ on a graph $\Gamma$ with $\pi_1(\Gamma) = F_k$
as a subdivision of a graph followed by a sequence of folds.
The result follows by carefully keeping track of the fundamental group.

In fact, we show that this method works in complete generality
so long as one can carry out the first step.
That is, if $G$ splits as $G = \pi_1(\Gamma,\mathscr{G})$,
folding can be used to analyze
any automorphism $\Phi\colon G \to G$ 
that can be topologically realized on $(\Gamma,\mathscr{G})$.

We return to the case of $G$ a free product.
In the case where each $A_i$ is finitely generated,
a choice of finite generating set yields a finite set
of partial conjugations and transvections that suffice to generate the rest.
The subgroup of $\aut(G)$ generated by partial conjugations 
and transvections is finitely presented if each $A_i$ is.
Call this subgroup $\aut^0(G)$, and the generating set described above a set of
\emph{Fouxe-Rabinovitch generators} for $\aut^0(G)$.
The algorithmic nature of folding yields the following useful corollary.

\begin{cor}\label{combing}
	There is a normal form for $\aut^0(G)$,
	and a method that takes as input an $(n+k)$-tuple
	$(g_i)$ of elements of $G$
	and determines whether $\ad(g_i^{-1})(A_i)$ together with 
	$\{g_{n+1},\dotsc,g_{n+k}\}$ generate $G$,
	and if so produces a normal form for an automorphism $\Phi\in\aut^0(G)$
	in the Fouxe-Rabinovitch generators
	taking the $A_i$ and a free basis of $F_k$
	to $\ad(g_i)(A_i)$ and $\{g_{n+1},\dotsc,g_{n+k}\}$.
\end{cor}

In principle---but not in actual fact in general---our method is algorithmic.
The problem is that arbitrary finitely-generated groups,
as is well-known, are not algorithmically well-behaved.
In many cases of interest,
for instance when the $A_i$ are small finite groups,
the method is simple enough that fairly complicated examples
can be worked easily by hand or programmed into a computer.
Of course, if the $\aut(A_i)$ are also amenable to algorithmic computation,
the corollary gives a method for putting any automorphism of $G$ 
in a normal form by first folding, then performing factor automorphisms,
and finally permuting isomorphic factors.

One particularly useful kind of combing would associate to a group element
a \emph{quasi-geodesic normal form.} 
Briefly, given a group $G$ and finite generating set $S$,
there is an associated \emph{Cayley graph,}
with vertices the elements of $G$ and an edge connecting
$h$ to $g$ when their difference $h^{-1}g$ belongs to $S$.
Any (connected) graph can be given a geodesic metric 
by declaring each (open) edge isometric to an (open) unit interval
and setting the distance between two points equal to the infimal length 
of a path connecting them.
The fine structure of a Cayley graph depends very much 
on the choice of generating set,
but the large-scale geometry 
turns out to be an invariant of the group.
A geodesic (shortest path) from $h$ to $g$
corresponds to a way of writing $h^{-1}g$ as a product
$s_1\dotsb s_\ell$ of elements of $S$ with $\ell$ as small as possible.

Let $K$ and $C$ be real numbers with $K\ge 1$ and $C > 0$,
and let $(X,d_X)$ be a geodesic metric space.
A \emph{$(K,C)$-quasi-geodesic} $\gamma\colon \mathbb{R}_+ \to X$ is a ray---that
is, a singly infinite path,
which we assume to be parametrized by arc length---such that 
for all $x,y\in \mathbb{R}_+$,
\[ \frac 1K\left|y - x\right| - C \le d_X(\gamma(x),\gamma(y))
\le K\left|y - x\right| + C. \]
In other words, the true distance in $X$ between points in the image of $\gamma$
differs from their distance along $\gamma$ by a 
bounded multiplicative and additive error.
A set of quasi-geodesic normal forms
associates to each $g \in G$ a path from $1$ to $g$ in the Cayley graph,
i.e. a finite sequence of elements of $S$ whose product is $g$,
which is $(K,C)$-quasi-geodesic in the true length of $g$.
Such a set of normal forms
would remain quasi-geodesic in any generating set,
and could prove useful for studying the large-scale geometry of $\aut(G)$.

Unfortunately, our normal forms do not enjoy this property.

\begin{prop}\label{notquasigeodesic}
	As long as $n+k \ge 4$ or $n = 0$ and $k = 3$,
	the combing constructed in \Cref{combing} is not quasi-geodesic.
\end{prop}

In the case of $G$ a free group,
the above proposition is due to Qing and Rafi \cite[Theorem A]{QingRafi}.
We prove the proposition 
by showing that Qing and Rafi's example occurs in free products more generally.
We thank Kasra Rafi for pointing out their example to us 
and suggesting the proposition.
The proof of the proposition amounts to working an enlightening example,
so we give the proof of it first, and then the theorem.

\begin{proof}[Proof of \Cref{notquasigeodesic}.]
	The case $n = 0$ and $k = 3$ is Theorem A of \cite{QingRafi}.
	It suffices to prove the cases $n+k = 4$ where $k \in \{0,1,2\}$.
	Let $G = A * B * C * D$, where $A,B,C,D$ are nontrivial
	finitely generated groups, the last $k$ of which are infinite cyclic.
	Choose nontrivial elements $a \in A$, $b \in B$ and $c \in C$.
	For $m$ and $n$ positive integers, consider the following automorphisms
	\begin{align*}
		&\Phi_{0,n,m} 
		\begin{dcases}
			x \mapsto x & x \in A*B*C \\
			d \mapsto ((ab)^{-m}c(ab)^{m}a)^{-n} d((ab)^{-m}c(ab)^{m}a)^{n} & d \in D,
		\end{dcases} \\
		&\Phi_{1,n,m}
		\begin{dcases}
			x \mapsto x & x \in A*B*C \\
			d \mapsto d((ab)^{-m}c(ab)^{m}a)^n & D = \langle d \rangle
		\end{dcases} \\
		\shortintertext{and}
		&\Phi_{2,n,m}
		\begin{dcases}
			x \mapsto x & x \in A*B*C \\
			d \mapsto d(c(ab)^m)^n & D = \langle d \rangle
		\end{dcases}
	\end{align*}
		
	\paragraph{Topological realization.}
	We will focus on $\Phi_{1,n,m}$; the other cases are entirely analogous.
	Consider the graph of groups $\mathcal{T}_{3,1}(G)$ described in
	\Cref{thistle31G} below.
	Identify $G = \pi_1(\mathcal{T}_{3,1}(G),\star)$ by choosing
	a spanning tree and orientations as in \Cref{presentation} so that
	$d$, the generator of $D$, is represented by the loop around $e_4$ in
	the positive direction.
	A similar graph of groups $\mathcal{T}_{n,k}(G)$ exists for
	$G$ an arbitrary free product. We call it the \emph{thistle
	with $n$ prickles and $k$ petals.} It interpolates between the
	thistle of \cite{TrainTracksOrbigraphs} (no petals),
	and the \emph{rose} (no prickles) commonly used to study the free group.

	\begin{figure}[ht]
		\begin{center}
			\includegraphics[width=\textwidth]{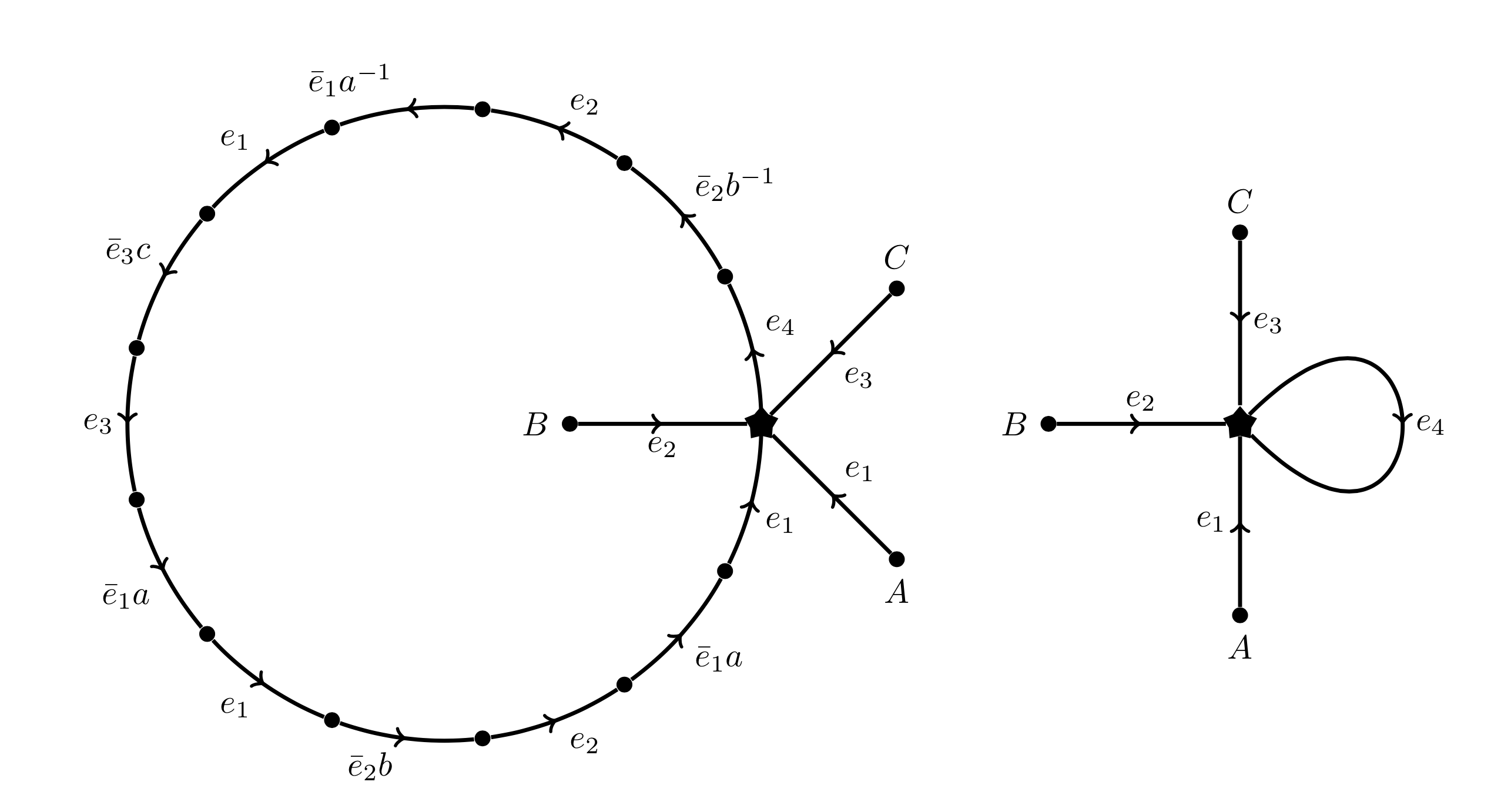}
			\caption{The thistle $\mathcal{T}_{3,1}(G)$ and the subdivision
			associated to $\Phi_{1,1,1}$.}
			\label{thistle31G}
		\end{center}
	\end{figure}

	Let $\sigma$ be the loop $\bar e_1ae_1\bar e_2be_2$ based at $\star$,
	and let $\gamma_m$ be the loop $\bar\sigma^m\bar e_3ce_3\sigma^m\bar e_1ae_1$,
	where $\sigma^m$ is the $m$-fold concatenation of $\sigma$ with itself.
	Note that $\Phi_{1,m,n}(d) \in G$ is represented by the loop
	$e_4\gamma_m^n$.
	A topological realization of $\Phi$ on $\mathcal{T}_{3,1}(G)$
	is given by the map $f\colon \mathcal{T}_{3,1}(G) \to \mathcal{T}_{3,1}(G)$ 
	which is the identity on the complement of $e_4$ and sends 
	$e_4$ to the path $e_4\gamma_m^n$.
	If we declare each point $e_4$ which is mapped to a vertex
	to be a vertex (with trivial vertex group),
	the map $f$ becomes a morphism, thus it has a factorization into folds,
	all of which must be of Type I.
	Because $e_4\gamma_m^n$ is a tight path, 
	$f$ is an immersion away from the basepoint $\star$.
	
	\paragraph{Folding Automorphisms.}
	As in \Cref{presentation}, we will choose a spanning tree
	to identify the fundamental group of the subdivided thistle with $G$.
	The obvious choice is the 
	spanning tree containing all the edges of the subdivided thistle,
	save for the one mapping to $e_4$.
	Thus $d$ is represented by the loop 
	that runs once around the subdivision of $e_4$.

	The last two edges of $\gamma_m^n$ read $\bar e_1ae_1$.
	Thus we may fold the two edges labeled $e_1$ together.
	Before folding again,
	we need to twist the edge labeled $\bar e_1a$ by $a^{-1}$.
	This sends the loop representing $d$ in the previous identification
	to the loop now representing $da$.
	Thus the twisting performed the right transvection of $d$ by $a$,
	which we will write as $\rho_{a,d}$.
	We record this as the first automorphism in our normal form.

	Iterating this process, we see that folding produces the normal forms
	\begin{align*}
		\Phi_{0,n,m} &= ((\chi_{a,d}\chi_{b,d})^{-m}\chi_{c,d}
		(\chi_{a,d}\chi_{b,d})^m\chi_{a,d})^n, \\
		\Phi_{1,n,m} &= ((\rho_{a,d}\rho_{b,d})^{-m}\rho_{c,d}
		(\rho_{a,d}\rho_{b,d})^m\rho_{a,d})^n, \\
		\shortintertext{and}
		\Phi_{2,n,m} &= (\rho_{c,d}(\rho_{a,d}\rho_{b,d})^m)^n,
	\end{align*}
	whose lengths in the Fouxe-Rabinovitch generators are
	$n(4m+2)$, $n(4m+2)$ and $n(2m+1)$, respectively.
	(Automorphisms are composed as functions.)

	On the other hand, one checks that
	\begin{align*}
		\Phi_{0,n,m} &= (\chi_{a,c}\chi_{b,c})^m(\chi_{c,d}\chi_{a,d})^n
		(\chi_{a,c}\chi_{b,c})^{-m}, \\
		\Phi_{1,n,m} &= (\chi_{a,c}\chi_{b,c})^m(\rho_{c,d}\rho_{a,d})^n
		(\chi_{a,c}\chi_{b,c})^{-m}, \\
		\shortintertext{and}
		\Phi_{2,n,m} &=
		(\rho_{a,c}\rho_{b,c})^m\rho_{c,d}^n(\rho_{a,c}\rho_{b,c})^{-m},
	\end{align*}
	whose lengths in the Fouxe-Rabinovitch generators are
	$4m + 2n$, $4m + 2n$ and $4m + n$, respectively.
	As $n$, and $m$ go to infinity, this shows that our normal forms cannot be
	$(K,C)$-quasi-geodesic for any choice of $(K,C)$.
\end{proof}

\begin{proof}[Proof of \Cref{fouxerabinovitch}.]
	All the ideas are already present in the proof above.
	A famous theorem of Grushko says that if
	$G = A_1*\dotsb*A_n*F_k$ and $G = B_1*\dotsb*B_m*F_\ell$,
	where the $A_i$ are freely indecomposable and not infinite cyclic,
	then $m = n$, $k = \ell$ and after reordering, $B_i$ is conjugate to $A_i$,
	and no element of $F_\ell$ is conjugate into any $A_i$.
	Thus if $\Phi\colon G \to G$ is an automorphism,
	it must permute the conjugacy classes of the $A_i$.

	We conclude that $\Phi$ can be topologically realized 
	by $f\colon\mathcal{T}_{n,k}(G) \to \mathcal{T}_{n,k}(G)$
	using exactly the same process as above.
	Since $\Phi$ composed with the surjection of $G$ onto $F_k$ is a surjection,
	we conclude that $f$ is a homotopy equivalence of the underlying graph.
	Thus after subdividing we may factor $f$ into a product of folds,
	followed by an isomorphism.
	The possible isomorphisms of $\mathcal{T}_{n,k}(G)$
	include elements of $\aut(A_i)$ and the finite group $\Sigma_G$.

	At each step, whenever possible, we fold edges belonging to the maximal tree,
	possibly passing to equivalent automorphisms,
	performing twist automorphisms or changing the maximal tree
	in order to keep folding.
	Both of the latter operations result in automorphisms of $G$,
	but the former does not.

	In the previous proposition, we analyzed the effect of twisting
	an edge $e$ by $g^{-1}$.
	It may happen that a single twist induces a
	commuting product of transvections and
	partial conjugations by $g$ or $g^{-1}$,
	based on whether the loop corresponding to a generator traverses $e$
	once or twice, and in which direction.

	Consider changing the maximal tree $T$ by adding an edge $e$
	that is not a loop and removing $e'$.
	The edges of the loop $\gamma_{\tau(\bar e)}e\bar\gamma_{\tau(e)}$ determine
	a cycle in the homology of the graph, we may choose any oriented edge 
	appearing with positive exponent in this cycle as our $e'$.
	If the loop $\gamma_{\tau(\bar e)}e\bar\gamma_{\tau(e)}$
	represented the generator $x_i$,
	changing the maximal tree corresponds to a commuting product of transvections
	and partial conjugations of generators not equal to $x_k$
	by $x_k$ and $x_k^{-1}$, according to whether the loop 
	corresponding to a generator traverses $e'$ once or twice, 
	and in which direction.
	A similar statement holds for folding an edge $e'$ 
	belonging to the maximal tree across a loop edge $e$.
	Writing down the automorphisms in the order they occur,
	we can write any element of $\aut(G)$ in the Fouxe-Rabinovitch generators.
\end{proof}

\begin{proof}[Proof of \Cref{combing}.]
	To make the foregoing discussion algorithmic, we need only stipulate
	\emph{how} we fold, and place the generators that occur in a commuting product
	in some order. For this step we assume that there is a normal form for
	elements of each $A_i$ and an algorithm to put elements of $A_i$ in normal form.
	For the former question,
	we will perform all folds that do not result in automorphisms first.
	Then, we perform the automorphism where the \emph{acting letter,}
	the $i$ in $\chi_{g_i,j}$ or $\rho_{g_i,j}$, is smallest.
	When we change the maximal tree, choose the edge $e'$ in the cycle 
	whose initial vertex is closest to the terminal vertex of $e$.
	Order the commuting product so that the smallest $j$ appears first.
	
	To determine whether an $(n+k)$-tuple $(g_i)$ of elements of $G$ corresponds,
	as in the statement, to a generating set,
	choose tight paths $\gamma_i$ in $\mathcal{T}_{n,k}(G)$ representing them,
	and form the map $f\colon \mathcal{T}_{n,k}(G) \to \mathcal{T}_{n,k}(G)$
	with $f(e_i) = \gamma_i$.
	Factor $f$ into a product of folds. If the last morphism
	is not an isomorphism, the tuple does not determine a generating set.
	Otherwise we have a normal form for it as above.
\end{proof}

\begin{rk}
	As we remarked earlier,
	nothing about our method of proof is special to $G$ a free product,
	except perhaps the guarantee that every automorphism $\Phi\colon G \to G$
	may be topologically realized on $\mathcal{T}_{n,k}(G)$,
	and our good understanding of automorphisms of $\mathcal{T}_{n,k}(G)$.
\end{rk}

\chapter{Train Track Maps}\label{traintrackchapter}

As we have already seen,
representing an automorphism of $G$
as a map on a graph of groups with fundamental group $G$
is a fruitful way to analyze its structure.
Ideally one wants such a representative to reveal
aspects of the dynamics of the action 
of the automorphism on the fundamental group,
similar to the Jordan normal form of a linear transformation.
In a seminal paper, Bestvina and Handel \cite{BestvinaHandel}
showed that every outer automorphism $\varphi \in \out(F_n)$
has such a representative, called a \emph{relative train track map.}

The inspiration for their construction was
the \emph{Nielsen--Thurston} classification of surface diffeomorphisms.
Let $S$ be an orientable surface with negative Euler characteristic.
The \emph{mapping class group} of $S$,
$\Mod(S)$, also sometimes called the \emph{Teichm\"uller modular group} of $S$,
is the group of homotopy classes of 
orientation-preserving diffeomorphisms of $S$.
Thurston \cite{Thurston} introduced \emph{pseudo-Anosov} diffeomorphisms
$f\colon S \to S$ and showed that for each element $\varphi \in \Mod(S)$, either
\begin{enumerate}
	\item $\varphi$ has finite order,
	\item $\varphi$ is \emph{reducible,} meaning it preserves some
		finite set of homotopy classes of essential, simple closed curves
		on $S$, or
	\item $\varphi$ is represented by a pseudo-Anosov diffeomorphism.
\end{enumerate}
In the reducible case, one may cut $S$ along the preserved collection of curves
and continue classifying.
Thurston proves that pseudo-Anosov representatives are ``efficient''
representatives: they are unique within their
homotopy class and minimize a quantity called \emph{topological entropy.}

A \emph{train track map} (see below for a definition)
is analogous to a pseudo-Anosov diffeomorphism
in that it is an efficient representative for an element of $\out(F_n)$.
Although train track maps are typically not unique,
it is still true that they minimize a form of topological entropy.

Just as not all elements of $\Mod(S)$ 
are represented by pseudo-Anosov diffeomorphisms,
not all elements of $\out(F_n)$ are represented by train track maps.
A \emph{relative train track map} parallels the reducible case above:
rather than cutting a surface, however,
the metaphor is collapsing a subgraph.
The purpose of this chapter is to extend Bestvina and Handel's construction
to automorphisms representable on a graph of groups.

As an application, we answer affirmatively a question of Paulin \cite{Paulin},
who asked whether, for a word-hyperbolic group $G$,
there exists a trichotomy for outer automorphisms $\varphi \in \out(G)$.
Paulin's trichotomy is a common generalization of the situation
for $G$ a free group and 
for $G$ the fundamental group 
of a closed surface with negative Euler characteristic.
See \Cref{paulinquestion} below for details.

\section{Train Track Maps on Graphs}
Let $\Gamma$ be a graph. A homotopy equivalence $f\colon\Gamma \to \Gamma$
is a \emph{train track map} if for all $k \ge 0$,
the restriction of $f^k$ to each edge is an immersion.
The definition is due to Thurston.
Now, in general $f$ itself cannot be an immersion:
since $f_\sharp \colon \pi_1(\Gamma) \to \pi_1(\Gamma)$
is an isomorphism,
if $f$ were an immersion,
the results of the previous chapter would imply that 
$f$ is an automorphism of $\Gamma$,
and would therefore have finite order.

\paragraph{An application of Thurston.}
Let us illustrate one useful consequence of having a train track map
for an outer automorphism $\varphi \in \out(F_n)$.
The following is due to Thurston \cite[Remark 1.8]{BestvinaHandel}.
We may associate a nonnegative matrix $M$ with integer entries
to any self-map of a finite graph
sending edges to edge paths,
once we fix an ordering of the edges.
The $ij$th entry of the matrix $M$ records the
number of times the image of the $j$th edge crosses the $i$th edge
in either direction.
If $f\colon \Gamma \to \Gamma$ is a train track map,
then the number of edges crossed by $f^k(e_j)$
is the sum of the entries in the $j$th column of $M^k$.

Let us make the additional assumption that the transition matrix $M$
is \emph{irreducible} (we give the definition below).
Every irreducible nonnegative integer matrix $M$
has a \emph{Perron--Frobenius eigenvalue} $\lambda \ge 1$,
and a corresponding eigenvector $\vec v$ with positive entries \cite{Seneta}.
If we metrize $\Gamma$ by assigning each edge $e_j$ a length of $\vec v_j$,
then $f$ expands each edge by a factor of $\lambda$.

Suppose our train track map $f\colon \Gamma \to \Gamma$ 
represents $\varphi \in \out(F_n)$.
Although outer automorphisms do not act on elements of $F_n$,
they do act on their conjugacy classes.
If $c$ is such a conjugacy class,
let $|c|$ denote the minimum word length of a representative
in some fixed finite generating set of $F_n$.
The \emph{exponential growth rate} of $c$ under $\varphi$ is
\[ \operatorname{EGR}(\varphi,c) = 
\limsup_{n\to\infty}\frac{\log |\varphi^n(c)|}{n}. \]
A quick calculation reveals that distorting the word length
on $F_n$ by a quasi-isometry does not alter $\operatorname{EGR}(\varphi,c)$.
In particular, the length of an immersed loop representing $\varphi(c)$
in the metric above will do,
and we conclude that $\operatorname{EGR}(\varphi,c) \le \log\lambda$,
with equality when $f$ restricted to the 
immersed loop determined by $c$ is an immersion.
In fact, it turns out that for any conjugacy class $c$,
$\operatorname{EGR}(\varphi,c)$ is either $\log\lambda$ or $0$,
with the latter occurring only when $\varphi$ acts periodically on $c$.
The key tool in completing the proof is Thurston's
Bounded Cancellation Lemma \cite{Cooper}.
We will complete the argument below (see \Cref{boundedcancellation})
once we are equipped with more vocabulary.

\paragraph{}
Now suppose that $f\colon \Gamma \to \Gamma$ takes edges to edge paths,
but that $f$ is not a train track map.
In this case, the Perron--Frobenius eigenvalue of the associated matrix $M$ 
\emph{overestimates} the exponential growth rate.
If one could show that in this situation $f$ may be altered 
so that the eigenvalue $\lambda$ decreases,
it follows that $f$ is a train track map when $\lambda$ reaches a minimum.
Furthermore, one could hope to find a train track map algorithmically
by applying certain moves to decrease the eigenvalue,
and arguing that a minimum must be reached after iterating this process
finitely many times.
In fact, this is exactly what Bestvina and Handel do.

\section{Train Track Maps on Graphs of Groups}
Notice that if $f\colon (\Gamma,\mathscr{G}) \to (\Gamma,\mathscr{G})$
is a \emph{topological realization} of an outer automorphism 
of $\pi_1(\Gamma,\mathscr{G})$,
it makes sense to ask whether $f$ is a train track map.
As foreshadowed in the previous chapter,
it turns out that the whole theory goes through 
so long as one can carry out the first step.
The first main result of this chapter is the following theorem.

\begin{thm}\label{traintrackthm}
	Let $G$ be a  finitely generated group.
	If $\varphi \in \out(G)$ can be realized on some 
	nontrivial splitting of $G$ with finitely generated edge groups,
	then there exists a nontrivial splitting $G = \pi_1(\Gamma,\mathscr{G})$
	with finitely generated edge groups
	and an irreducible train track map
	$f\colon (\Gamma,\mathscr{G}) \to (\Gamma,\mathscr{G})$
	realizing $\varphi$.
\end{thm}

\begin{rk}
	If in the statement above,
	$\varphi$ was originally realized on the splitting $(\Lambda,\mathscr{L})$,
	then, modulo appropriate subdivisions,
	there is a morphism $(\Lambda,\mathscr{L}) \to (\Gamma,\mathscr{G})$
	inducing an isomorphism of fundamental groups,
	but there may not be any such morphism in the other direction.
	In the above situation we say that 
	$(\Lambda,\mathscr{L}) \succeq (\Gamma,\mathscr{G})$,
	but we may not have $(\Lambda,\mathscr{L}) \asymp (\Gamma,\mathscr{G})$.
	In terms of the $G$-action on the Bass--Serre tree,
	every vertex stabilizer of $\tilde\Lambda$ 
	stabilizes some vertex of $\tilde\Gamma$,
	but some elements may stabilize vertices of $\tilde\Gamma$
	while acting freely on $\tilde\Lambda$.
\end{rk}

In Bestvina and Handel's original formulation,
not every $\varphi \in \out(F_n)$ is represented by a train track map.
The obstruction to promoting $f\colon \Gamma \to \Gamma$
to an irreducible train track map is finding a noncontractible
$f$-invariant subgraph.
In this case, Bestvina and Handel
find a stratification of $\Gamma$ into subgraphs
and a \emph{relative train track map,}
which preserves the stratification,
and resembles a train track map on each stratum.
We will give a more precise definition below.

In our proof of \Cref{traintrackthm},
$f$-invariant subgraphs are collapsed,
thus we end up with an irreducible train track map.
If $(\Lambda,\mathscr{L})$ is an $f$-invariant subgraph of groups,
it makes sense to restrict $f$ to $(\Lambda,\mathscr{L})$,
and we can apply \Cref{traintrackthm} again.
If we began with a finite graph,
the process terminates in finitely many steps,
and we end up with a stratified collection of train track maps,
where train track maps appearing lower in the stratification
happen within vertex groups of some higher-stratum train track map.

Thus one way to prove the existence of relative train track maps
from the existence of train track maps
is to carefully ``blow up'' the vertex groups
of the higher strata to yield a realization on a stratified graph of groups.
Another way is to adapt the original Bestvina--Handel argument.
This is the second main result of the chapter.

\begin{thm}\label{reltraintrackthm}
	Let $G$ be a finitely generated group.
	If $\varphi\in\out(G)$ is realized by 
	$f\colon(\Gamma,\mathscr{G}) \to (\Gamma,\mathscr{G})$,
	for some splitting $G = \pi_1(\Gamma,\mathscr{G})$
	with finitely generated vertex and edge groups,
	then there is a relative train track map
	$f'\colon (\Gamma',\mathscr{G}') \to (\Gamma',\mathscr{G}')$
	realizing $\varphi$ such that 
	$(\Gamma,\mathscr{G}) \asymp (\Gamma',\mathscr{G}')$.
\end{thm}

A satisfying consequence of the thorough understanding
of morphisms of graphs of groups we developed in previous chapters
is that adapting the original proof poses relatively few new difficulties.

\begin{dfn}
	The collection of $\varphi\in\out(G)$ realizable on a given
	splitting $G = \pi_1(\Gamma,\mathscr{G})$ form a subgroup,
	which we will call the \emph{modular group of $(\Gamma,\mathscr{G})$,}
	$\Mod(\Gamma,\mathscr{G})$.
\end{dfn}

Because every homotopy equivalence of a surface $S$ that restricts to
a homeomorphism on the boundary of $S$ is homotopic to a diffeomorphism,
$\Mod(S)$ may be defined 
as the group of homotopy classes of orientation-preserving 
homotopy equivalences $f\colon S  \to S$.
This parsimony of definitions is the source of our term.

There are other candidates for a name.
Sykiotis \cite{Sykiotis} calls such outer automorphisms \emph{symmetric.}
In some situations, $\Mod(\Gamma,\mathscr{G})$ 
coincides with the subgroup of $\out(G)$ 
preserving the \emph{deformation space} of $G$-trees $\mathfrak{D}$
containing $\tilde\Gamma$,
which Guirardel and Levitt refer to as $\out_\mathfrak{D}(G)$
\cite{GuirardelLevittDeformationSpaces}.
Symmetric automorphisms of free groups already refer to a particular
$\Mod(\Gamma,\mathscr{G})$ on the one hand,
and in certain cases, e.g. the Baumslag--Solitar group $\operatorname{BS}(2,4)$,
the modular group of a graph of groups
is much smaller than $\out_\mathfrak{D}(G)$.

If $G$ is a free product $A_1*\dotsb*A_n*F_k$,
where the $A_i$ are freely indecomposable and not infinite cyclic,
the argument in the previous chapter shows that 
$\out(G) = \Mod(\mathcal{T}_{n,k}(G))$.
The existence of relative train track maps for free products
has been proven several times,
first by Collins and Turner \cite{CollinsTurner},
and later by Sykiotis \cite{Sykiotis} (who allows certain finite edge groups),
Francaviglia and Martino \cite{FrancavigliaMartino},
and most recently by the present author \cite{TrainTracksOrbigraphs}.
This chapter generalizes the early sections of \cite{TrainTracksOrbigraphs}
and comprises parts of a future revision of that paper.

Recently Feighn and Handel \cite{FeighnHandel},
building on earlier work of Bestvina, Feighn and Handel \cite{BestvinaFeighnHandel},
developed an improvement of relative train track maps
they call \emph{completely split relative train track maps,}
CTs for short.
We plan to show how to construct CTs on graphs of groups,
but leave that to the future.

\paragraph{Hierarchical structure and a question.}
Before we proceed, let us note an interesting connection
we are presently unable to tease out in full.
Let $S$ be an orientable surface.
The \emph{curve complex} of $S$, 
denoted $\mathscr{C}(S)$,
is a simplicial complex
with vertices isotopy classes of essential, simple closed curves on $S$,
and with an edge between two vertices when the associated curves
may be realized disjointly.
The complex is \emph{flag;} a higher dimensional simplex is present
if and only if its $1$-skeleton is.\footnote{
Warren Dicks: ``every non-simplex contains a non-edge.''}
$\Mod(S)$ acts on $\mathscr{C}(S)$,
and a celebrated theorem of Masur and Minsky \cite{MasurMinsky}
says that $\mathscr{C}(S)$ is \emph{Gromov-hyperbolic},
that is, negatively curved in a large-scale sense.

A mapping class $\varphi \in \Mod(S)$ is pseudo-Anosov
if and only if it has no periodic orbit in $\mathscr{C}(S)$,
and reducible if and only if $\varphi$ fixes a simplex $\sigma$.
In the latter case, $\varphi$ acts on the subcomplex 
whose vertices are all adjacent to each vertex of $\sigma$.
This turns out to be the curve complex of the surface obtained by
cutting along the curves making up the vertices of $\sigma$.
Thus one may understand elements of $\Mod(S)$
via their action on a hierarchy of hyperbolic spaces.
This situation has been codified by Behrstock, Hagen and Sisto
\cite{BehrstockHagenSisto} into the study of
\emph{hierarchically hyperbolic spaces.}

One interesting aspect of \Cref{traintrackthm}
is how it parallels the hierarchical situation in the mapping class group.
Indeed, in the case where $G$ is a free product,
there are a number of Gromov-hyperbolic complexes
that $\out(G)$ acts on that could play the role of curve complexes.

However, there must be some subtlety:
in \cite{BehrstockHagenSistoII}, Behrstock, Hagen and Sisto
adapt an argument of Bowditch to show that
hierarchically hyperbolic groups---very roughly,
those groups with a nice action on a hierarchically hyperbolic space---satisfy
a \emph{quadratic isoperimetric inequality.}
On the other hand, Handel and Mosher \cite{HandelMosher}
and Bridson and Vogtmann \cite{BridsonVogtmann}
showed that for $n\ge 3$,
the optimal isoperimetric 
inequality---also called the \emph{Dehn function}---for 
$\out(F_n)$ is exponential.

Notably, this argument does not show \emph{where}
the hierarchically hyperbolic machinery fails.
It would be interesting to know for which splittings $(\Gamma,\mathscr{G})$,
if any, the hierarchical structure gestured at here
shows that $\Mod(\Gamma,\mathscr{G})$ is a hierarchically hyperbolic group.
Such a question likely amounts to 
asking for the Dehn function of $\Mod(\Gamma,\mathscr{G})$,
which remains unknown even for $G$ a free product that is not a free group.

\section{Topology of Finite Graphs of Groups}

In the preceding chapters,
we have been carefully describing graphs of groups
and morphisms between them in an essentially algebraic or combinatorial way.
Indeed, in principle everything that follows could be expressed this way.
We find it clarifying to instead adopt the usual language of topology.

Let $(\Lambda,\mathscr{L})$ and $(\Gamma,\mathscr{G})$ be graphs of groups
with fundamental group isomorphic to $G$.
Let $f\colon \tilde\Lambda \to \tilde\Gamma$
be a $G$-equivariant continuous map.
Each edge in the fundamental domain of $\tilde\Lambda$ is sent to
a path (a continuous map of the interval) in $\tilde\Gamma$,
which is homotopic rel endpoints to a piecewise-linear embedding 
with respect to some metric or to a constant map---a \emph{tight edge path,} 
in other words.
Additionally, each $G$-equivariant map is equivariantly homotopic 
to a map that sends vertices to vertices.
In sum, every $G$-equivariant map $f\colon \tilde\Lambda \to \tilde\Gamma$
is equivariantly homotopic to a map 
$f'\colon \tilde\Lambda \to \tilde\Gamma$
which satisfies the following conditions.
\begin{enumerate}
	\item After subdividing each edge of $\tilde\Lambda$ into finitely many edges,
		$f' \colon \tilde\Lambda \to \tilde\Gamma$ becomes a morphism,
	\item $f'$ sends vertices to vertices, and
	\item $f'$ sends edges to tight edge paths.
\end{enumerate}

In fact, the same is true of the induced map 
$(\Lambda,\mathscr{L}) \to (\Gamma,\mathscr{G})$.
We specialize this discussion to self-maps.
\begin{dfn}
	Let $\varphi\in \out(G)$ be an outer automorphism,
	and $(\Gamma,\mathscr{G})$ a graph of groups with
	$\pi_1(\Gamma,\mathscr{G},\star) = G$.
	A map $f\colon (\Gamma,\mathscr{G}) \to (\Gamma,\mathscr{G})$
	is a \emph{topological realization} of $\varphi$
	if it satisfies the three conditions above, does not collapse edges,
	and the induced outer action of $f$ on $(\Gamma,\mathscr{G},\star)$
	is $\varphi$.
	(Note that we do not require $f$ to fix the basepoint.)
\end{dfn}

Thus up to homotopy, every $G$-equivariant map of $G$-trees
may be represented on the quotient graph of groups.
There is however an additional point of subtlety:
the translation between the two depends on a fixed choice of fundamental domain
in each $G$-tree.
In the quotient graph of groups, this is the same data
as fixing an identification of the fundamental group with $G$.

\section{Markings}

In what follows, we will begin with an element 
$\varphi \in \Mod(\Gamma,\mathscr{G})$\footnote{
We are tempted to call it a mapping class.}
and a topological realization $f\colon(\Gamma,\mathscr{G}) \to (\Gamma,\mathscr{G})$.
To promote $f$ to a train track map,
we will perform a series of operations
which change $(\Gamma,\mathscr{G})$.
For $f$ to remain a topological realization of $\varphi$,
we need a way of keeping track of the fundamental group.  

To do this, fix once and for all a reference graph of groups $\mathbb{G}$,
a vertex $\star \in \mathbb{G}$,
and an identification $\pi_1(\mathbb{G},\star)$ with $G$.
For example, if $G$ is a free product,
we may take $\mathbb{G}$ to be the thistle $\mathcal{T}_{n,k}(G)$, 
and $\star$ to be the vertex of $\mathcal{T}_{n,k}(G)$.

\begin{dfn}
	A \emph{marked} graph of groups is a graph of groups $(\Gamma,\mathscr{G})$
	together with a \emph{homotopy equivalence}
	$ \tau\colon \mathbb{G} \to (\Gamma,\mathscr{G}). $
	That is, $\tau$ is a (continuous) map, 
	and there exists a (not unique) map 
	$\sigma\colon (\Gamma,\mathscr{G}) \to \mathbb{G}$
	such that $\tau\sigma$ and $\sigma\tau$ are each homotopic to the identity.
Thus $(\Gamma,\mathscr{G}) \asymp \mathbb{G}$.
\end{dfn}

To prove \Cref{traintrackthm},
it may happen that we need to pass to a graph of groups
$(\Gamma',\mathscr{G}')$ such that $\mathbb{G} \succneqq (\Gamma',\mathscr{G}')$.
The induced map $\tau'\colon \mathbb{G} \to (\Gamma',\mathscr{G}')$
will fail to admit a homotopy inverse,
essentially because it collapses certain edges.
Nonetheless, by the results of the previous chapter,
we may collapse subgraphs to yield $\mathbb{G}'$
and a homotopy equivalence $\tau''\colon \mathbb{G}' \to (\Gamma',\mathscr{G}')$,
which will play the role of a marking.
We will  typically leave  this implicit in what follows.

\section{Irreducibility}

Fix $\varphi\in\Mod(\mathbb{G})$, and a topological realization
$f\colon(\Gamma,\mathscr{G})\to(\Gamma,\mathscr{G})$ of $\varphi$.
Label the edges of $\Gamma$ as $e_1,\dotsc,e_m$.

\begin{dfn}
	The \emph{transition matrix} $M$ of $f$
	is an $m\times m$ matrix with $ij$th entry
	$m_{ij}$ equal to the number of times the edge path $f(e_j)$
	contains $e_i$ in either orientation.
	A matrix with nonnegative integer entries is \emph{irreducible}
	if for each pair $i$ and $j$, there exists some positive integer $k$
	such that the $ij$th entry of $M^k$ is positive.
\end{dfn}

\begin{rk}
	Here is a simple criterion for checking whether such a matrix $M$
	is irreducible.
	Create a \emph{directed} graph $\Gamma_M$ 
	with vertices $v_1,\dotsc v_m$ in one-to-one correspondence 
	with the columns of $M$.
	Draw $m_{ij}$ directed edges from $v_j$ to $v_i$.
	$M$ is irreducible if and only if $\Gamma_M$ is \emph{strongly connected.}
	This means that for each vertex $v \in \Gamma_M$,
	there exists directed paths connecting $v$ to every vertex of $\Gamma_M$
	(including $v$ itself).
\end{rk}

\begin{ex}\label{traintrackex}
	Consider the free product
	\[ G = C_2*C_2*C_2*C_2 = 
	\langle  a,b,c,d \mid a^2 = b^2 = c^2 = d^2 =1 \rangle\]
	and an automorphism $\Phi\colon G \to G$
	defined by its action on generators as
	\[ \Phi\begin{dcases}
		a \mapsto b \\
		b \mapsto c \\
		c \mapsto d \\
		d \mapsto cbdadbc.
	\end{dcases}\]
	(Notice that, e.g. $c^{-1} = c$.)
	Let $\mathcal{T}_4  = \mathcal{T}_{4,0}(G)$.
	A topological realization $f\colon \mathcal{T}_{4} \to \mathcal{T}_{4}$
	of $\Phi$ is described in \Cref{traintrackfig1}.
	The transition matrix of $f$ is
	\[ M = \begin{pmatrix}
		0 & 0 & 0 & 1 \\
		1 & 0 & 0 & 2 \\
		0 & 1 & 0 & 2 \\
		0 & 0 & 1 & 2
	\end{pmatrix}.\]
	One checks that $M$ is irreducible.
	\begin{figure}[ht]
		\begin{center}
			\[
				\begin{tikzpicture}[auto, node distance=4cm, baseline]
				\node[pt, star, scale = 1.5] (v) at (0,0) {};
				\node[pt, "$\langle a \rangle$" left] (a) [above left of=v] {}
					edge[->-, "$e_1$"] (v);
				\node[pt, "$\langle b \rangle$" right] (b) [above right of=v] {}
					edge[->-, "$e_2$"] (v);
				\node[pt,  "$\langle c \rangle$" right] (c) [below right of=v] {}
					edge[->-, "$e_3$"] (v);
				\node[pt, "$\langle d \rangle$" left] (d) [below left of=v] {}
					edge[->-, "$e_4$"] (v);
			\end{tikzpicture}
			\qquad
			f\begin{dcases}
				e_1 \mapsto e_2 \\
				e_2 \mapsto e_3 \\
				e_3 \mapsto e_4 \\
			e_4 \mapsto e_1\bar e_4de_4\bar e_2be_2\bar e_3ce_3
			\end{dcases}
		\]
		\end{center}
		\caption{The topological realization 
		$f\colon \mathcal{T}_4 \to \mathcal{T}_4$.}
		\label{traintrackfig1}
	\end{figure}
\end{ex}

The following theorem is the main tool for the proof of \Cref{traintrackthm}.

\begin{thm}[Perron--Frobenius \cite{Seneta}]
	Let $M$ be an irreducible, nonnegative integral matrix.
	Then there is a unique positive eigenvector $\vec v$
	with norm $1$
	whose associated eigenvalue $\lambda$ satisfies $\lambda \ge 1$.
	If $\lambda = 1$, then $M$ is a transitive permutation matrix.
	If $\vec w$ is a positive vector and $\mu > 0$ satisfies
	$(M\vec w)_i \le \mu\vec w_i$ for all $i$ and $(M\vec w)_j < \mu\vec w_j$
	for some $j$, then $\lambda < \mu$.
\end{thm}

If the transition matrix of a topological realization
$f\colon (\Gamma,\mathscr{G}) \to (\Gamma,\mathscr{G})$ 
is irreducible,
we say $f$ is \emph{irreducible}, and
we call the eigenvalue $\lambda$ in the statement above
the \emph{Perron--Frobenius eigenvalue} of $f$.

For the topological realization 
$f\colon \mathcal{T}_4 \to \mathcal{T}_4$ 
in \Cref{traintrackex},
the Perron--Frobnius eigenvalue $\lambda$
is the largest root of the polynomial $x^4 - 2x^3 - 2x^2 - 2x - 1$;
we have $\lambda \approx 2.948$.

\begin{rk}
	Note that the theorem implies that if $M$ and $M'$
	are irreducible matrices with $m_{ij} \le m'_{ij}$
	for all pairs $i$ and $j$ and with strict inequality in at least one pair,
	then the associated Perron--Frobenius eigenvalues $\lambda$ and $\lambda'$
	satisfy $\lambda < \lambda'$.
\end{rk}

\paragraph{Collapsing revisited.}
If $f\colon (\Gamma,\mathscr{G}) \to (\Gamma,\mathscr{G})$
is \emph{not} irreducible, 
then the graph associated to the transition matrix for $f$
is not strongly connected.
This means there is some edge of $\Gamma$ whose forward orbit under $f$
is contained within some 
(possibly disconnected) proper $f$-invariant subgraph of $\Gamma$.

One possibility is that $\tilde\Gamma$ is not \emph{minimal,}
i.e. $\Gamma$ contains a valence-one vertex $v$
with incident edge $e$ such that $\iota_e\colon \mathscr{G}_e \to \mathscr{G}_v$
is an isomorphism.
We will show that such valence-one vertices may be safely removed.

Here is the other possibility.
Suppose $\Gamma_0$ is a proper $f$-invariant subgraph of $\Gamma$
which is maximal with respect to inclusion (among proper subgraphs, naturally).
Let $(\Gamma_1,\mathscr{G}_1)$ be the graph of groups obtained by
collapsing each component of $\Gamma_0$ 
and let $p\colon (\Gamma,\mathscr{G}) \to (\Gamma_1,\mathscr{G}_1)$
be the quotient map.
By \Cref{collapsing}, $f$ yields a map
$f_1\colon (\Gamma_1,\mathscr{G}_1) \to (\Gamma_1,\mathscr{G}_1)$
such that the following diagram commutes
\[\begin{tikzcd}
	(\Gamma,\mathscr{G}) \ar[r,"f"] \ar[d,"p"] & (\Gamma,\mathscr{G}) \ar[d,"p"] \\
	(\Gamma_1,\mathscr{G}_1) \ar[r,"f_1"] & (\Gamma_1,\mathscr{G}_1).
\end{tikzcd}\]
If $\tau\colon \mathbb{G} \to (\Gamma,\mathscr{G})$ is the original marking,
then $p\tau$ is the new marking.
Because $f$ sends edges to tight paths,
for each edge $e$ that is not collapsed,
any maximal segment of $f(e)$ contained in $\Gamma_0$
is not null-homotopic rel endpoints.
It follows that $f_1$ sends edges to tight paths.
Since $\Gamma_0$ was assumed to be maximal, 
$f_1$ does not collapse edges and is thus a topological realization.
Since $\Gamma_0$ was a maximal invariant subgraph,
if $\tilde\Gamma_1$ is minimal, then 
$\Gamma_1$ has no proper $f_1$-invariant subgraph,
so $f_1$ is irreducible.
The transition matrix for $f_1$ is obtained from that for $f$
by removing the rows and columns corresponding to the edges of $\Gamma_0$.

We clearly have $(\Gamma,\mathscr{G}) \succeq (\Gamma_1,\mathscr{G}_1)$.
In order to have $(\Gamma,\mathscr{G}) \asymp (\Gamma_1\mathscr{G}_1)$,
each $g \in G$ must fix a point of $\tilde\Gamma_1$ if and only if
$g$ fixes a point of $\tilde\Gamma$.
This happens if and only
for each component $(\Lambda,\mathscr{L})$ of $\Gamma_0$,
each element of
$\pi_1(\Lambda,\mathscr{L})$ fixes a point of $\tilde\Lambda$.
It follows by finite generation of $G$ that 
there is a point of $\tilde\Lambda$ fixed by all of $\pi_1(\Lambda,\mathscr{L})$;
see for instance \cite{Tits}.
$\tilde\Lambda$ is equivariantly contractible to this fixed point.
Thus $\Lambda$ must be a tree,
and some vertex $v \in \Lambda$ satisfies
$\mathscr{L}_v \cong \pi_1(\Lambda,\mathscr{L})$.

In this case we say $(\Lambda,\mathscr{L})$
is a \emph{contractible tree} in $(\Gamma,\mathscr{G})$.
It is \emph{nontrivial} if it contains an edge.
(This definition is standard for train tracks,
but nonstandard for graph theory.)
If each component of a subgraph $\Gamma_0 \subseteq \Gamma$ 
determines a contractible tree in $(\Gamma,\mathscr{G})$,
we say $\Gamma_0$ is a \emph{contractible forest}.
It is \emph{nontrivial} if some component is nontrivial.
A (necessarily contractible) forest is \emph{pretrivial}
for a homotopy equivalence $f\colon(\Gamma,\mathscr{G}) \to (\Gamma,\mathscr{G})$
if each edge contained in the forest is eventually mapped to a point
by some iterate of $f$.
If $f\colon (\Gamma,\mathscr{G}) \to (\Gamma,\mathscr{G})$
only fails to be a topological realization because it collapses edges,
collapsing a maximal pretrivial forest yields a topological realization.

In what follows, whenever we collapse invariant subgraphs,
we will always be working with a marked graph of groups $(\Gamma,\mathscr{G})$
with $\tilde\Gamma$ minimal.
This avoids the silly case where collapsing a maximal invariant subgraph
causes $G$ to act with global fixed point on the resulting Bass--Serre tree.

\begin{dfn}
	We say that $\varphi\in\Mod(\mathbb{G})$ is \emph{irreducible}
	if for every topological realization 
	$f\colon (\Gamma,\mathscr{G}) \to (\Gamma,\mathscr{G})$
	with $\mathbb{G} \asymp  (\Gamma,\mathscr{G})$
	such that $\tilde\Gamma$ is minimal and
	$\Gamma$ contains no nontrivial 
	$f$-invariant contractible forests,
	we have that $f$ is irreducible.
\end{dfn}

In general, if $\varphi$ belongs to $\Mod(\mathbb{G})$,
where the underlying graph of $\mathbb{G}$ is finite,
then $\varphi$ is irreducible in some $\Mod(\mathbb{G}')$,
where $\mathbb{G} \succeq \mathbb{G}'$---if
only because every topological realization
on a graph of groups with one edge is homotopic
to an automorphism.

\section{Operations on Topological Realizations}

Suppose a topological realization 
$f\colon (\Gamma,\mathscr{G}) \to (\Gamma,\mathscr{G})$
is not a train track map.
Thus there is some edge $e$ of $\Gamma$ and $k > 1$
such that $f^k|_e$ is not an immersion.
By the results of the previous chapter,
$f^k|_e$ factors through a fold.
After performing that fold,
we may be able to tighten $f$,
thus decreasing the associated Perron--Frobenius eigenvalue.

\begin{dfn}
	A \emph{turn} based at a vertex $v$ in a marked graph of groups
	$(\Gamma,\mathscr{G})$ is an unordered pair 
	of oriented edges $\{\tilde e,\tilde e'\}$
	in $\st(\tilde v)$, where $\tilde v$ is a lift of $v$ to $\tilde\Gamma$.
	Equivalently, a turn is a pair of elements $\{([g],e),([g'],e')\}$ of the set
	\[ \coprod_{e\in\st(v)}\mathscr{G}_v/\iota_e(\mathscr{G}_e) \times \{e\}. \]
	We will use both definitions interchangeably.
	A turn is \emph{nondegenerate} if $\tilde e \ne \tilde e'$,
	otherwise it is \emph{degenerate.}
\end{dfn}

\begin{rk}
	Our definition of $\st(v)$ 
	as those oriented edges with \emph{terminal} vertex $v$
	means that our orientation convention for turns
	is the opposite of the $\out(F_n)$ literature.
\end{rk}

A topological realization $f\colon (\Gamma,\mathscr{G}) \to (\Gamma,\mathscr{G})$
yields a self-map $Df$ of the set of turns in $(\Gamma,\mathscr{G})$.
If $\hat f$ is the morphism determined by subdividing $f$, we have
\[ Df\{([g],e),([g'],e')\} = \{\hat f_{\st(v)}([g],e),\hat f_{\st(v)}([g'],e')\}, \]
where $\hat f_{\st(v)}$ is the map defined in \Cref{localmap}.
If $f(eg) = g_{i_1}e_{i_1}\dotsb e_{i_k}g_{i_{k+1}}$,
we have $\hat f_{\st(v)}([g],e) = ([g_{i_{k+1}}],e_{i_k})$.

In \Cref{traintrackex}, $\star$ is mapped to itself by $f$,
and the restriction of $Df$ to $\star$ is determined by the dynamical system 
$e_1 \mapsto e_2 \mapsto e_3 \leftrightarrow e_4$.

\begin{dfn}
	A turn is \emph{illegal} 
	if its image under some iterate of $Df$ is degenerate.
	It is \emph{legal} otherwise.
\end{dfn}

Observe that $\mathscr{G}_v$ acts on the set of turns based at $v$
by left multiplication on the cosets,
and $Df$ respects this action.
Thus legality and illegality is a property of a turn's orbit under $\mathscr{G}_v$.

In \Cref{traintrackex}, a turn $\{ e_i, e_j\}$ based at $\star$
is illegal if $i \equiv j \mod 2$, and is legal otherwise.

Consider the edge path
\[ \gamma = g_1e_1\dotsb e_kg_{k+1}. \]
We say $\gamma$ \emph{crosses} or \emph{takes} the turns
$\{([1], e_i),([g_{i+1}],\bar e_{i+1})\}$ for $1 \le i \le k$.
A path $\gamma$ is \emph{legal} if it takes only legal turns,
and is illegal otherwise.
Thus a topological realization 
$f\colon (\Gamma,\mathscr{G}) \to (\Gamma,\mathscr{G})$
is a \emph{train track map} if and only if $f$ sends each edge
of $\Gamma$ to a legal path.
In \Cref{traintrackex}, $f$ is not a train track map
because the $f$-image of $e_4$ takes the illegal turn $\{\bar e_4,\bar e_2\}$.

\begin{rk}\label{boundedcancellation}
	We can now finish the argument of Thurston.
	Suppose $f\colon \Gamma \to \Gamma$ is an irreducible train track map
	representing $\varphi\in\out(F_n)$
	with Perron--Frobenius eigenvalue $\lambda > 1$,
	and that $c$ is a conjugacy class in $F_n$.
	We may represent $c$ by an immersion $\sigma\colon S^1 \to \Gamma$.
	As we argued above, there exists a metric on $\Gamma$
	with respect to which $f$ 
	expands lengths of legal paths by a factor of $\lambda$.
	For $\gamma\colon S^1 \to \Gamma$,
	let $|\gamma|$ denote the length of the unique immersed
	circle homotopic to $\gamma$ in this metric.

	Since $f$ sends legal paths to legal paths,
	for all $k \ge 0$,
	$f^k(\sigma)$ has a finite number of number of illegal turns 
	bounded independently of $k$.
	At each illegal turn, 
	there may be nontrivial cancellation,
	so the length of an immersed circle homotopic to $f^k(\sigma)$
	is \emph{a priori} not greater than $\lambda^k(\sigma)$.
	On the other hand, the \emph{Bounded Cancellation Lemma}
	\cite{Cooper} implies that the length of cancellation 
	at any illegal turn is bounded independently of $k$ and $\sigma$.
	Thus we have
	\[ \lambda^k|\sigma| - K \le |f^k\sigma| \le \lambda^k|\sigma|, \]
	for some constant $K$.
	It follows that either $|f^k(\sigma)|$ is bounded independent of $k$,
	or $\operatorname{EGR}(\varphi,c) = \log\lambda$.
\end{rk}

\paragraph{\Cref{traintrackex} continued.}
Let us fold $f\colon \mathcal{T}_4 \to \mathcal{T}_4$ 
at the illegal turn $\{ e_4, e_2\}$.
To do this, subdivide $e_4$ at the preimage of the vertex
with vertex group $\langle c\rangle$
into the edge path $e_4'e_4''$,
and fold $e_4''$ with $e_2$.
The action of the resulting map 
$f'\colon (\Gamma_1,\mathscr{G}_1) \to (\Gamma_1,\mathscr{G}_1)$ 
on edges is obtained from $f$ by replacing instances of $e_4$ with $e_4'e_2$.
Thus \[f'(e'_4) = e_1\bar e_2\bar e_4'de_4'e_2\bar e_2b e_2\bar e_3c.\]
We may tighten $f'$ by a homotopy with support on $e_4'$ 
to remove $e_2\bar e_2$, yielding
an irreducible topological realization
$f_1\colon (\Gamma_1,\mathscr{G}_1) \to (\Gamma_1,\mathscr{G}_1)$.
See  \Cref{traintrackfig2}.

\begin{figure}[ht]
	\centering
	\[\begin{tikzpicture}[auto, node distance =4cm, baseline]
			\node[pt, star, scale = 1.5] (v) at (0,0) {};
			\node[pt, "$\langle a \rangle$" left] (a) [above of=v] {}
				edge[->-, "$e_1$"'] (v);
			\node[pt, "$\langle b \rangle$" above] (b) [right of=v] {}
				edge[->-, "$e_2$"] (v);
			\node[pt,  "$\langle c \rangle$" left] (c) [below of=v] {}
				edge[->-, "$e_3$"] (v);
			\node[pt, "$\langle d \rangle$" above] (d) [right of=b] {}
				edge[->-, "$e'_4$"] (b);
		\end{tikzpicture}\qquad
		f_1\begin{dcases}
			e_1 \mapsto e_2 \\
			e_2 \mapsto e_3 \\
			e_3 \mapsto e_4'e_2 \\
			e_4' \mapsto e_1\bar e_2\bar e'_4de'_4be_2\bar e_3c
		\end{dcases}
	\]
	\caption{The topological realization 
	$f_1\colon (\Gamma_1,\mathscr{G}_1) \to (\Gamma_1,\mathscr{G}_1)$.}
	\label{traintrackfig2}
\end{figure}

The Perron--Frobenius eigenvalue $\lambda_1$ for 
$f_1\colon (\Gamma_1,\mathscr{G}_1) \to (\Gamma_1,\mathscr{G}_1)$
is the largest root of the polynomial $x^4 - 2x^3 - 2x^2 +x - 1$
and satisfies $\lambda_1 \approx 2.663$;
thus $\lambda_1 <  \lambda$.

However, $f_1$ is still not a train track map:
$Df_1$ sends the turn $\{(1,e_4),(b,\bar e_2)\}$ crossed by $f_1(e'_4)$
to $\{(c,\bar e_3),(c,\bar e_3)\}$, which is thus illegal.
As explained in the previous chapter, in order to fold again, 
we need to twist the edge $e'_4$ by $b^{-1} = b$.
This changes the marking by replacing $\langle d\rangle$
with $\langle bdb\rangle$.
This also replaces $f_1(e_3)$ with $e'_4be_2$
and replaces $f_1(e'_4)$ with $e_1\bar e_2b\bar e'_4de'_4e_2\bar e_3c$.
Then we fold $e'_4$ and $\bar e_2$,
producing a new marked graph $(\Gamma_2,\mathscr{G}_2)$
which is abstractly isomorphic to $\mathcal{T}_4$, 
but with a different marking.
The action of the resulting map 
$f''\colon (\Gamma_2,\mathscr{G}_2) \to (\Gamma_2,\mathscr{G}_2)$
on edges is obtained by replacing instances of $e'_4$ with $e''_4\bar e_2$.
Thus we have
\[ f''(e''_4) = e_1\bar e_2be_2\bar e''_4bdbe''_4\bar e_2e_2, \]
and we may tighten to produce
an irreducible topological realization 
$f_2\colon (\Gamma_2,\mathscr{G}_2) \to (\Gamma_2,\mathscr{G}_2)$.
See \Cref{traintrackfig3}.

\begin{figure}[ht]
	\centering
	\[\begin{tikzpicture}[auto, node distance = 4cm, baseline]
		\node[pt, star, scale = 1.5] (v) at (0,0) {};
		\node[pt, "$\langle a \rangle$" left] (a) [above left of=v] {}
			edge[->-, "$e_1$"] (v);
		\node[pt, "$\langle b \rangle$" right] (b) [above right of=v] {}
			edge[->-, "$e_2$"] (v);
		\node[pt,  "$\langle c \rangle$" left] (c) [below left of=v] {}
			edge[->-, "$e_3$"] (v);
		\node[pt, "$\langle bdb \rangle$" right] (d) [below right of=v] {}
			edge[->-, "$e''_4$"] (v);
		\end{tikzpicture}
		\qquad
		f_2\begin{dcases}
			e_1 \mapsto e_2 \\
			e_2 \mapsto e_3 \\
			e_3 \mapsto e''_4\bar e_2be_2 \\
			e_4 \mapsto e_1\bar e_2be_2\bar e''_4bdbe''_4
	\end{dcases}\]
	\caption{The topological realization 
	$f_2\colon (\Gamma_2,\mathscr{G}_2) \to (\Gamma_2,\mathscr{G}_2)$.}
	\label{traintrackfig3}
\end{figure}

The Perron--Frobenius eigenvalue $\lambda_2$
is the largest root of $x^4 - 2x^3 - 2x^2 + 2x - 1$
and satisfies $\lambda_2 \approx 2.539$;
thus $\lambda_2 < \lambda_1$.
The restriction of $Df_2$ to turns incident to $\star$
is determined by the dynamical system $e_1 \mapsto e_2 \leftrightarrow e_3$,
$e_4 \mapsto e_4$.
The only nondegenerate illegal turn in $(\Gamma_2,\mathscr{G}_2)$
is $\{e_1,e_3\}$, which is not crossed by the $f_2$-image of any edge,
so $f_2\colon (\Gamma_2,\mathscr{G}_2) \to (\Gamma_2,\mathscr{G}_2)$ 
is a train track map.

The name ``train track map''
asks the reader to imagine drawing the edges incident to $\star$
in such a way that the illegal turn $e_1$ and $e_3$ make a very sharp corner,
while each other pair of eges makes a much looser one.
A legal path, then, is one that does not make any sharp turns---it
is the kind of path a train could make as it moves along its tracks.
In a setting with nontrivial vertex stabilizers,
one should perhaps imagine this picture in the Bass--Serre tree
or in an orbifold sense.

\paragraph{}
The broad-strokes outline of the proof of \Cref{traintrackthm} 
is much the same as in the previous example. 
By folding at illegal turns,
we often produce nontrivial tightening,
which decreases the Perron--Frobenius eigenvalue $\lambda$.
By controlling the presence of valence-one and valence-two vertices,
we may argue that the transition matrix lies in a finite set of matrices,
thus the Perron--Frobenius eigenvalue $\lambda$ 
may only be decreased finitely many times.
In the remainder of the section, 
we make this precise by recalling Bestvina and Handel's original analysis.
The proofs are identical to the original, so we omit them.

\paragraph{Subdivision.}
Given a topological realization,
$f\colon (\Gamma,\mathscr{G}) \to (\Gamma,\mathscr{G})$,
if $p$ is a point in the interior of an edge $e$
such that $f(p)$ is a vertex,
we may give $(\Gamma,\mathscr{G})$ a new graph of groups structure 
by declaring $p$ to be a vertex, with vertex group equal to $\mathscr{G}_e$.

\begin{lem}[Lemma 1.10 of \cite{BestvinaHandel}]
	If $f\colon (\Gamma,\mathscr{G}) \to (\Gamma,\mathscr{G})$
	is a topological realization of $\varphi\in\Mod(\mathbb{G})$,
	and $f_1\colon (\Gamma_1,\mathscr{G}_1) \to (\Gamma_1,\mathscr{G}_1)$
	is obtained by subdivision,
	then $f_1$ is a topological realization of $\varphi\in\Mod(\mathbb{G})$.
	If $f$ is irreducible, then $f_1$ is too,
	and the associated Perron--Frobenius eigenvalues are equal.
	\hfill\qedsymbol
\end{lem}

\paragraph{Valence-One Homotopy.}
We distinguish two kinds of valence-one vertex of $(\Gamma,\mathscr{G})$.
A valence-one vertex with incident edge $e$ is \emph{inessential} 
if the monomorphism $\iota_e\colon \mathscr{G}_e \to \mathscr{G}_v$
is an isomorphism.
A valence-one vertex is \emph{essential} if it is not inessential.
The Bass--Serre tree is minimal if and only if $(\Gamma,\mathscr{G})$
contains no inessential valence-one vertices.

If $v$ is an inessential valence-one vertex with incident edge $e$,
then $e$ is a contractible subtree. 
Let $(\Gamma_1,\mathscr{G}_1)$ denote the subgraph of groups
determined by $\Gamma_1 \setminus \{e,v\}$,
and let $\pi\colon (\Gamma,\mathscr{G}) \to (\Gamma_1,\mathscr{G}_1)$
be the collapsing map.
Let $f_1\colon(\Gamma_1,\mathscr{G}_1) \to (\Gamma_1,\mathscr{G}_1)$
be the topological realization obtained from $\pi f|_{(\Gamma_1,\mathscr{G}_1)}$
by tightening and collapsing a maximal pretrivial forest.
We say $f_1\colon (\Gamma_1,\mathscr{G}_1) \to (\Gamma_1,\mathscr{G}_1)$
is obtained from $f\colon (\Gamma,\mathscr{G}) \to (\Gamma,\mathscr{G})$
by a \emph{valence-one homotopy}.

\begin{lem}[Lemma 1.11 of \cite{BestvinaHandel}]
	If $f\colon (\Gamma,\mathscr{G}) \to (\Gamma,\mathscr{G})$
	is an irreducible topological realization
	with Perron--Frobenius eigenvalue $\lambda$
	and $f_1\colon (\Gamma_1,\mathscr{G}_1) \to (\Gamma_1,\mathscr{G}_1)$
	is obtained from $f\colon(\Gamma,\mathscr{G}) \to (\Gamma,\mathscr{G})$
	by performing valence-one homotopies 
	on all inessential valence-one vertices of $(\Gamma,\mathscr{G})$
	followed by the collapse of a maximal invariant subgraph,
	then $f_1\colon (\Gamma_1,\mathscr{G}_1) \to (\Gamma_1,\mathscr{G}_1)$
	is irreducible, and the associated Perron--Frobenius eigenvalue
	$\lambda_1$ satisfies $\lambda_1 < \lambda$.
	\hfill\qedsymbol
\end{lem}

\paragraph{Valence-Two Homotopy.}
We likewise distinguish two kinds of 
valence-two vertex of $(\Gamma,\mathscr{G})$.
A valence-two vertex $v$ with incident edges $e_i$ and $e_j$
is \emph{inessential} if at least one of the monomorphisms
$\iota_{e_i}\colon \mathscr{G}_{e_i} \to \mathscr{G}_v$ and
$\iota_{e_j}\colon \mathscr{G}_{e_j} \to \mathscr{G}_v$
is an isomorphism,
say $\iota_{e_j}\colon \mathscr{G}_{e_j} \to  \mathscr{G}_v$.
Again, it is \emph{essential} if it is not inessential.
(For convenience, if $(\Gamma,\mathscr{G})$ has one vertex,
that vertex is essential regardless of its valence.)
Let $\pi$ be the map that collapses $e_j$ to a point
and expands $e_i$ over $e_j$.
Define a map $f'\colon (\Gamma,\mathscr{G}) \to (\Gamma,\mathscr{G})$
by tightening $\pi f$.
Observe that no vertex of $(\Gamma,\mathscr{G})$ is mapped to $v$.
Thus we may define a new graph of groups structure 
$(\Gamma',\mathscr{G}')$
by removing $v$ from the set of vertices.
Thus the edge path $e_i\bar e_j$ is now an edge,
which we will call $e_i$.
Let $f''\colon (\Gamma',\mathscr{G}') \to (\Gamma,\mathscr{G}')$
be the map obtained by
tightening $f''(e_i) = f'(e_i\bar e_j)$.
Finally, let $f_1\colon(\Gamma_1,\mathscr{G}_1) \to (\Gamma_1,\mathscr{G}_1)$
be the topological realization obtained by collapsing
a maximal pretrivial forest.
We say that $f_1\colon(\Gamma_1,\mathscr{G}_1)\to(\Gamma_1,\mathscr{G}_1)$
is obtained by a \emph{valence-two homotopy of $v$ across $e_j$.}

\begin{lem}[Lemma 1.12 of \cite{BestvinaHandel}]
	Let $f\colon(\Gamma,\mathscr{G}) \to (\Gamma,\mathscr{G})$
	be an irreducible topological realization,
	and suppose $(\Gamma,\mathscr{G})$
	has no inessential valence-one vertices.
	Suppose $f_2\colon(\Gamma_2,\mathscr{G}_2) \to (\Gamma_2,\mathscr{G}_2)$
	is the irreducible topological realization
	obtained by performing a valence-two homotopy of
	$v$ across $e_j$
	followed by the collapse of a maximal invariant subgraph.
	Let $M$ be the transition matrix of $f$
	and choose a positive eigenvector $\vec w$ with
	$M\vec w = \lambda \vec w$.
	If $\vec w_i \le \vec w_j$,
	then $\lambda_2 \le \lambda$;
	if $\vec w_i < \vec w_j$,
	then $\lambda_2 < \lambda_1$.
	\hfill \qedsymbol
\end{lem}

\begin{rk}\label{valencetwodanger}
	The statement of the lemma hides a problem:
	if we cannot freely choose which edge incident to an inessential
	valence-two vertex to collapse via a valence-two homotopy,
	we may be forced to \emph{increase} $\lambda$.
	Fortunately, such valence-two vertices are not produced
	in the proof of \Cref{traintrackthm}.
\end{rk}

\paragraph{Folding.}
We have already seen folding in the previous chapter.
We may fold edges $e$ and $e'$ who share a common terminal vertex
and whose images under
$f\colon(\Gamma,\mathscr{G}) \to (\Gamma,\mathscr{G})$
are equal.
More generally,
we may subdivide $e$ and $e'$ and fold their maximal
common terminal segment.

\begin{lem}[Lemma 1.15 of \cite{BestvinaHandel}]
	Suppose $f\colon  (\Gamma,\mathscr{G}) \to (\Gamma,\mathscr{G})$
	is an irreducible topological realization
	and that $f_1\colon (\Gamma_1,\mathscr{G}_1) \to (\Gamma_1,\mathscr{G}_1)$
	is obtained by folding a pair of edges.
	If $f_1$ is a topological realization, then it is irreducible,
	and the associated Perron--Frobenius eigenvalues
	satisfy $\lambda_1 = \lambda$.
	Otherwise, 
	let $f_2\colon (\Gamma_2,\mathscr{G}_2) \to (\Gamma_2,\mathscr{G}_2)$
	be the irreducible topological representative obtained
	by tightening, collapsing a maximal pretrivial forest,
	and collapsing a maximal invariant subgraph.
	Then the associated Perron--Frobenius eigenvalues satisfy
	$\lambda_2 < \lambda$.
	\hfill \qedsymbol
\end{lem}

\section{Proof of the Main Theorem}
We begin with $\varphi \in \Mod(\mathbb{G})$,
and a topological realization 
$f\colon (\Gamma,\mathscr{G}) \to (\Gamma,\mathscr{G})$.
By performing preliminary valence-one and valence-two homotopies,
we may assume that $(\Gamma,\mathscr{G})$ has no inessential
valence-one and valence-two vertices.
Indeed, we may also assume that $\mathbb{G}$ does not have
inessential valence-one and valence-two vertices.
Likewise, by collapsing a maximal invariant subgraph,
we may assume that $f\colon(\Gamma,\mathscr{G}) \to (\Gamma,\mathscr{G})$
is irreducible.

\begin{lem}\label{complexity}
	Given a graph of groups $(\Gamma,\mathscr{G})$,
	let $\eta(\Gamma,\mathscr{G})$ be the number of vertices $v$ 
	of $\Gamma$ such that for each oriented edge $e \in \st(v)$,
	the monomorphism
	$\iota_e\colon \mathscr{G}_e \to \mathscr{G}_v$ is not surjective.
	For the moment we call these vertices \emph{marked.}
	Let $\beta(\Gamma)$ be the first Betti number of $\Gamma$.
	If $\mathbb{G} \succeq (\Gamma,\mathbb{G})$ 
	and $(\Gamma,\mathscr{G})$ has no inessential
	valence-one or valence-two vertices,
	then $\Gamma$ has at most
	$2\eta(\mathbb{G}) + 3\beta(\mathbb{G}) - 3$ edges.
\end{lem}

\begin{rk}
	The statement above has a corner case that is worth mentioning.
	This occurs when $\beta(\mathbb{G}) = 1$ and $\eta(\mathbb{G}) = 0$.
	In this case, $\mathbb{G}$ is homotopy equivalent to a
	graph of groups $(\Gamma,\mathscr{G})$ with one vertex and one edge $e$
	where at least one of the monomorphisms $\iota_e$ and $\iota_{\bar e}$
	is surjective.
	In this case every topological realization
	$f\colon (\Gamma,\mathscr{G}) \to (\Gamma,\mathscr{G})$
	is homotopic to an automorphism of $(\Gamma,\mathscr{G})$,
	so \Cref{traintrackthm} holds
	after passing to this graph with one vertex by repeatedly
	removing inessential valence-one and valence-two vertices
	via homotopies.
\end{rk}

\begin{proof}
	If $(\Gamma,\mathscr{G}) \asymp \mathbb{G}$,
	a homotopy equivalence between them sends 
	marked vertices of $\mathbb{G}$ to marked vertices of
	$(\Gamma,\mathscr{G})$ and vice versa,
	so $\eta(\Gamma,\mathscr{G}) = \eta(\mathbb{G})$.
	Only marked vertices of $\Gamma$ may have valence less than three.
	Similarly $\beta(\Gamma) = \beta(\mathbb{G})$.
	Choose a cyclic ordering of the marked vertices of $\Gamma$ 
	and form a new graph $\Gamma'$ by attaching an edge from
	each marked vertex of $\Gamma$ to its successor in the cyclic order.
	We have thus added $\eta(\mathbb{G})$ edges to $\Gamma$.
	The graph $\Gamma'$ has no valence-one nor valence-two vertices
	and first Betti number $\eta(\mathbb{G}) + \beta(\mathbb{G})$.
	The lemma follows in the case $(\Gamma,\mathscr{G}) \asymp \mathbb{G}$
	by a simple Euler characteristic argument.

	If instead $\mathbb{G} \succneqq (\Gamma,\mathscr{G})$,
	the map $\sigma\colon \mathbb{G} \to (\Gamma,\mathscr{G})$
	identifying $G$ with $\pi_1(\Gamma,\mathscr{G})$
	sends marked vertices of $\mathbb{G}$ to marked vertices
	of $(\Gamma,\mathscr{G})$. 
	If a vertex of $(\Gamma,\mathscr{G})$ is marked
	but not in the image of any marked vertex of $\mathbb{G}$,
	then it is obtained by collapsing some subgraph of $\mathbb{G}$
	with nontrivial first Betti number.
	In other words, we have
	\[ \eta(\Gamma,\mathscr{G}) + \beta(\Gamma) 
		\le \eta(\mathbb{G}) + \beta(\mathbb{G}) \quad\text{and}\quad
	\beta(\Gamma) \le \beta(\mathbb{G}), \]
	and at least one inequality is strict.
	The lemma follows by the argument above.
\end{proof}

We now turn the proof of the main theorem.
The argument is essentially due to Bestvina and Handel
\cite[Theorem 1.7]{BestvinaHandel}.

\begin{proof}[Proof of \Cref{traintrackthm}]
	We keep the notation as above.
	Suppose $\lambda = 1$.
	Then $f$ permutes the edges of $\Gamma$
	and is thus an automorphism of $(\Gamma,\mathscr{G})$.
	In particular, $f$ is a train track map.

	So assume $\lambda >  1$.
	We will show that if our irreducible topological realization
	$f\colon (\Gamma,\mathscr{G}) \to (\Gamma,\mathscr{G})$
	is not a train track map, 
	then there is an irreducible topological realization
	$f_1\colon (\Gamma_1,\mathscr{G}_1) \to (\Gamma_1,\mathscr{G}_1)$
	without inessential valence-one or valence-two vertices
	such that the associated Perron--Frobenius eigenvalues satisfy
	$\lambda_1 < \lambda$.

	The previous lemma shows the size of the transition matrix
	of $f_1\colon (\Gamma_1,\mathscr{G}_1) \to (\Gamma_1,\mathscr{G}_1)$
	is uniformly bounded depending only on $\mathbb{G}$.
	Furthermore, if $M$ is an irreducible matrix,
	its Perron--Frobenius eigenvalue $\lambda$ is bounded below 
	by the minimum sum of the entries of a row of $M$.  
	To see this, let $\vec w$ be the associated positive eigenvector.
	If $\vec w_j$ is the smallest entry of $\vec w$,
	$\lambda \vec w_j = (M \vec w)_j$ is greater than
	$w_j$ times the sum of the entries of the $j$th row of $M$.
	
	Thus if we iterate, reducing the Perron--Frobenius eigenvalue,
	there are only finitely many irreducible transition matrices that can occur,
	so at some finite stage the Perron--Frobenius eigenvalue will reach a minimum.
	At this point, we must have a train track map.

	To complete the proof, we turn to the question of decreasing $\lambda$.
	Suppose $f\colon (\Gamma,\mathscr{G}) \to (\Gamma,\mathscr{G})$
	is not a train track map.
	Then there exists a point $p$ in the interior of an edge
	such that $f(p)$ is a vertex,
	and $f^k$ is not locally injective at $p$ for some $k > 1$.
	We assume that topological realizations act linearly on edges
	with respect to some metric on $\Gamma$.
	Since $\lambda > 1$,
	this means the set of points of $\Gamma$ 
	eventually mapped to a vertex is dense.
	Thus we can choose a neighborhood $U$ of $p$
	so small that it satisfies the following conditions.
	\begin{enumerate}
		\item The boundary $\partial U$ is a two point set $\{ s, t \}$,
			where $f^\ell(s)$ and $f^\ell(t)$
			are vertices for some $\ell \ge 1$.
		\item $f^i|_U$ is injective for $1 \le i \le k-1$.
		\item $f^k$ is two-to-one on $U \setminus \{p\}$,
			and $f^k(U)$ is contained within a single edge.
		\item $p \notin f^i(U)$, for $1 \le i \le k$.
	\end{enumerate}

	First we subdivide at $p$. Then we subdivide
	at $f^i(s)$ and $f^i(t)$ for $1 \le i \le \ell -1$
	(in reverse order so that subdivision is allowed).
	The vertex $p$ has valence two;
	denote the incident edges by $e$ and $e'$.
	Observe that $f^{k-1}(e)$ and $f^{k-1}(e')$ are each single
	edges that are identified by $f$.
	Thus we may fold.
	The resulting map
	$f'\colon (\Gamma',\mathscr{G}') \to (\Gamma',\mathscr{G}')$
	may be a topological realization,
	in which case $\lambda' = \lambda$.
	In this case, $f'^{k-2}(e)$ and $f'^{k-2}(e')$
	are single edges that are identified by $f'$.
	In the contrary case,
	nontrivial tightening occurs.
	After collapsing a maximal pretrivial forest
	and a maximal invariant subgraph,
	the resulting irreducible topological realization
	$f''\colon (\Gamma'',\mathscr{G}'') \to (\Gamma'',\mathscr{G}'')$
	satisfies $\lambda'' < \lambda$.
	
	Repeating this dichotomy $k$ times if necessary,
	we have either decreased $\lambda$,
	or we have folded $e$ and $e'$ 
	so that $p$ is now an inessential vertex of valence one.

	We remove inessential valence-one and valence-two vertices
	by the appropriate homotopies.
	Since valence-one homotopy always decreases
	the Perron--Frobenius eigenvalue,
	the resulting irreducible topological realization
	$f_1\colon (\Gamma_1,\mathscr{G}_1) \to (\Gamma_1,\mathscr{G}_1)$
	satisfies $\lambda_1 < \lambda$.
\end{proof}

\begin{rk}
	As in the original, the proof of \Cref{traintrackthm}
	provides in outline an algorithm that takes as input
	a topological realization of an automorphism
	and returns a train track map.
\end{rk}

\section{On a Question of Paulin}\label{paulinquestion}

In this section, we establish the following alternative
for outer automorphisms of word hyperbolic groups.
It answers affirmatively two questions of Paulin
\cite[page 333]{Paulin} \cite[page 150]{Paulin}.

A group is \emph{word hyperbolic} if its word metric 
for some (and hence any) finite generating set
satisfies a certain large-scale negative curvature condition
called \emph{Gromov-hyperbolicity} or $\delta$\emph{-hyperbolicity}.
Free groups, free products of finite groups, and 
fundamental groups of surfaces with negative Euler characteristic
are all examples of word hyperbolic groups.

\begin{thm}\label{trichotomy}
	Let $G$ be a finitely-generated word hyperbolic group.
	Each $\Phi \in \aut(G)$ satisfies one of the following
	conditions.
	\begin{enumerate}
		\item Periodic: the outer class of $\Phi$
			has finite order in $\out(G)$.
		\item Reducible: there exists a splitting
			$G = \pi_1(\Gamma,\mathscr{G})$
			with virtually cyclic edge groups
			and an automorphism $\Phi'$ in the outer class of
			$\Phi$ such that $\Phi$ is an automorphism
			of $(\Gamma,\mathscr{G})$.
		\item Irreducible: There exists a \emph{small} 
			isometric action of $G$
			on a \emph{real tree} $T$
			and a \emph{homothety} $H \colon T \to T$
			with \emph{stretch factor}
			an \emph{algebraic integer} $\lambda > 1$
			such that $H(g.x) = \Phi'(g).H(x)$
			for all $x \in T$ and $g \in G$.
			The homothety $H$ has a fixed point.
	\end{enumerate}
\end{thm}

Clearly the third bullet point needs some explanation.
A \emph{real tree} (or an \emph{$\mathbb{R}$-tree}) $T$ is a metric space where
every pair of distinct points is connected by a unique arc,
by which we mean an embedding of an interval of $\mathbb{R}$.
Furthermore, this arc is a \emph{geodesic,}
i.e. its length realizes the distance between the pair of points.
A group is \emph{small} if it does not contain a free group of rank $2$.
An action of a group $G$ on a real tree $T$ is \emph{small}
if each arc stabilizer is small.
A \emph{homothety} of a metric space $(X,d)$
is a map $h\colon X \to X$
together with a \emph{stretch factor} $\lambda > 0$
such that 
\[ d(h(x),h(y)) = \lambda d(x,y)\qquad \text{for all } x, y\in X.\]
A real number is an \emph{algebraic integer} if it is the root of
a monic polynomial with integer coefficients.
Eigenvalues of integral matrices are algebraic integers.

For the reader familiar with Thurston's theory,
the trichotomy in the statement really is a true generalization
of Thurston's dynamical trichotomy, as we now briefly explain.

A simple closed curve $\gamma$ on a surface $S$ determines
a splitting of $\pi_1(S)$ as an amalgamated free product or HNN extension
of free groups with cyclic edge group generated by 
the loop determined by $\gamma$ in $\pi_1(S)$
according to whether $\gamma$ separates or does not separate $S$, respectively.
If a mapping class $\varphi\in\Mod(S)$
fixes a simple closed curve up to isotopy,
it has a representative $\Phi'\colon S \to S$ that fixes the curve
(and the basepoint).
Thus $\Phi'$ acts as an automorphism of the corresponding splitting.

Associated to a pseudo-Anosov diffeomorphism
$\Phi\colon S \to S$
there is an \emph{attracting geodesic lamination.}
The precise definition is perhaps not so important,
but one may think of points of the lamination as 
accumulation points of geodesic representatives
for the curves $\Phi^k(c)$ for $k \ge 0$
in some fixed hyperbolic metric on $S$,
where $c$ is an arbitrary essential simple closed curve.
\emph{Dual} to the lift of the attracting lamination 
to the universal cover of $S$ is an $\emph{R}$-tree $T$
equipped with a small action of $\pi_1(S)$.
Since the pseudo-Anosov diffeomorphism $\Phi$ 
preserves the leaves of the lamination,
it acts on the dual tree.
The stretch factor $\lambda$ is the stretch factor of $\Phi$,
hence is an algebraic integer satisfying $\lambda > 1$.

The proof of \Cref{trichotomy} follows rather quickly
from \Cref{traintrackthm}, building on deep theorems
of Dunwoody, Bowditch, Bestvina–Feighn and Paulin.
The method of passing from a (relative) train track
to an invariant $\mathbb{R}$-tree appears to have
been discovered simultaneously by many people.
We sketch the construction given by Gaboriau--Jaeger--Levitt--Lustig
\cite{GaboriauJaegerLevittLustig}.

Paulin \cite{Paulin2} has a partial result in the vein of
\Cref{trichotomy}. Namely, he shows the existence
of the invariant tree $T$ and homothety $H$.
Our proof shows that the stretch factor of $H$
is always an algebraic integer,
and that when $H$ has stretch factor $1$,
$\Phi$ is reducible.

\begin{proof}
	Fix a finitely-generated word hyperbolic group $G$.
	Bestvina--Feighn \cite{BestvinaFeighnRTrees} show that
	$\out(G)$ is finite when $G$ does not admit a splitting
	with virtually cyclic edge groups.
	In this case, each automorphism $\Phi\colon G \to G$
	is periodic.

	So assume $G$ admits a splitting with virtually cyclic edge groups.
	Finitely-generated word hyperbolic groups are finitely presented,
	so $G$ is \emph{accessible} 
	in the sense of Dunwoody \cite{DunwoodyAccessible}.
	Thus $G$ admits a maximal splitting as a finite graph of groups
	with finite edge groups 
	(and without inessential valence-one or valence-two vertices).
	In this case each vertex group is either finite or a \emph{one-ended}
	word hyperbolic group.
	Bowditch \cite{Bowditch} constructs a canonical splitting
	called the \emph{JSJ splitting} for one-ended word hyperbolic groups
	that do not act geometrically on the hyperbolic plane.
	It is a finite graph of groups with virtually cyclic edge groups.
	We may ``blow up'' Dunwoody's accessible splitting
	by grafting in Bowditch's JSJ splitting for every vertex
	with one-ended vertex group 
	that does not act geometrically on the hyperbolic  plane.
	As explained above, in the case where $G$ \emph{does} 
	act geometrically on the hyperbolic plane,
	the trichotomy is due to Thurston; thus we may assume
	that $G$ does not act geometrically on the hyperbolic plane,
	and we may leave such vertices alone in the splitting.
	Call the resulting graph of groups $\mathbb{G}$.

	Since the conjugacy classes of vertex groups of $\mathbb{G}$ 
	are determined by $G$, $\Mod(\mathbb{G}) = \out(G)$.
	By \Cref{traintrackthm}, if $\varphi \in \out(G)$
	is the outer class of $\Phi$,
	there exists an irreducible train track map
	$f\colon (\Gamma,\mathscr{G})  \to (\Gamma,\mathscr{G})$
	realizing $\varphi$
	with $\mathbb{G} \succeq (\Gamma,\mathscr{G})$,
	which thus has virtually cyclic edge groups.
	If the associated Perron--Frobenius eigenvalue $\lambda$
	is equal to $1$, then $f$ is an automorphism of $(\Gamma,\mathscr{G})$.
	To choose $\Phi'$, one need only choose a basepoint $\star$
	in $(\Gamma,\mathscr{G})$ 
	and a path in the marking spanning tree of $\Gamma$ 
	from $\star$ to $f(\star)$.

	If $\lambda > 1$, we use the construction of
	a real tree $T$ associated to a train track map
	from \cite{GaboriauJaegerLevittLustig}.
	In particular we refer the reader to them for the choice of
	$\Phi'$ with fixed  point in $T$.
	The idea is the following.
	Assign each edge of $\Gamma$ a length equal to
	the corresponding entry of the positive eigenvector $\vec w$ with norm $1$.
	Assume $f$ linearly expands edges by a factor of $\lambda$.
	Fix some lift $\tilde f$ to the Bass--Serre tree $\tilde\Gamma$.
	As above, this involves some choices.
	We may equip $\tilde \Gamma$ with the lifted metric.
	We have $\tilde f(g. x) = \Phi'(g). f(x)$,
	and $d(\tilde f(x),\tilde f(y)) \le \lambda d(x,y)$
	for all $g \in  G$ and $x,y \in \tilde\Gamma$.
	Define a new distance function on $T$ as
	the limit of the non-increasing sequence
	\[ d_\infty(x,y) = \lim_{k\to\infty}\frac{d(\tilde f^k(x),\tilde f^k(y))}
		{\lambda^k}.  \]
	One checks that the $G$ action preserves $d_\infty$,
	and that $d_\infty(\tilde f(x),\tilde f(y)) = \lambda d_\infty(x,y)$.
	The distance $d_\infty$ fails to be a metric on $\tilde\Gamma$ 
	only because there are a priori distinct points $x$ and $y$
	with $d_\infty(x,y) = 0$.
	The real tree $T$ is obtained from $(\tilde\Gamma,d_\infty)$
	by identifying those points $x$, $y$ of $\tilde\Gamma$
	with $d_\infty(x,y) = 0$.
\end{proof}

\section{Relative and Partial Train Track Maps}

In the previous sections, we constructed
irreducible train track maps from topological realizations
by using the machinery of graphs of groups to collapse 
any invariant subgraphs that occurred.
While this gives a conceptually simple result,
a train track map for every element of $\Mod(\mathbb{G})$,
it hides much of the dynamical complexity possible.

A finer tool is a \emph{relative} train track map.
Here is the idea.
A \emph{filtration} of a graph $\Gamma$
is a nested sequence of subgraphs
\[ \varnothing = \Gamma_0 \subset \Gamma_1 \subset \dotsb
\subset \Gamma_m = \Gamma. \]
The subgraphs need not be connected.
A relative train track map is a topological realization
respecting the filtration that looks like,
in a sense which we make precise below,
a train track map on each $\Gamma_k$
once we collapse $\Gamma_{k-1}$.

Let us make the forgoing more precise.
The $k$th \emph{stratum} of a filtration 
$\varnothing \subset \Gamma_0 \subset \Gamma_1 \subset
\dotsb \subset \Gamma_m = \Gamma$ is the subgraph
\[ H_k = \overline{\Gamma_k \setminus \Gamma_{k-1}}. \]
Suppose a topological realization 
$f\colon (\Gamma,\mathscr{G}) \to (\Gamma,\mathscr{G})$
\emph{preserves} the filtration,
in the sense that $f(\Gamma_i) \subset \Gamma_i$.
Then the $k$th stratum has a \emph{transition matrix}
$M_k$, which has rows and columns for the edges of $H_k$,
and $ij$th entry equal to, as usual,
the number of times the $f$-image of the $j$th
edge crosses the $i$th edge in either direction.
The $k$th stratum is \emph{irreducible} if $M_k$ is
an irreducible matrix.
We say a filtration is \emph{maximal}
if each stratum is either irreducible
or has transition matrix identically equal to zero.
The latter are \emph{zero strata}.
Every filtration can be refined to a maximal filtration.
We always assume that $f$ preserves a given filtration,
and that it is maximal.

Each irreducible stratum has a Perron--Frobenius eigenvalue $\lambda_k$.
We say the $k$th stratum is \emph{exponentially-growing} if $\lambda_k > 1$.
If $\lambda_k = 1$, then $H_k$ is \emph{non-exponentially-growing}.
A stratum $H_k$ \emph{contains} the turn $\{([g],e),([g'],e')\}$
if both $e$ and $e'$ belong to $H_k$.

\begin{dfn}
	A topological realization 
	$f\colon (\Gamma,\mathscr{G}) \to (\Gamma,\mathscr{G})$
	is a \emph{relative train track map}
	if there is a maximal filtration
	$\varnothing =  \Gamma_0 \subset \Gamma_1 \subset \dotsb
	\subset \Gamma_m = \Gamma$
	preserved by $f$
	such that for each exponentially-growing stratum $H_k$,
	the map $f$ satisfies the following three properties.
	\begin{enumerate}
		\item[(RTT-i)] The map $Df$ maps
			the set of turns in $H_k$ into itself.
		\item[(RTT-ii)] If $\gamma \subset \Gamma_{k-1}$
			is a nontrivial path with endpoints in
			$H_k \cap \Gamma_{k-1}$,
			then $f(\gamma)$
			is a homotopically nontrivial path with endpoints in
			$H_k \cap \Gamma_{k-1}$.
		\item[(RTT-iii)] For each legal path $\gamma$
			in $H_k$, $f(\gamma)$ is a path that does
			not contain any illegal turns in $H_k$.
	\end{enumerate}
\end{dfn}

\begin{prop}\label{ttfromrtt}
	Suppose $f\colon(\Gamma,\mathscr{G}) \to (\Gamma,\mathscr{G})$
	is a relative train track map with irreducible stratum $H_k$.
	Let $(\Gamma'_k,\mathscr{G}'_k)$ denote the marked
	graph of groups obtained from $\Gamma_k$
	by collapsing each component of $\Gamma_{k-1}$.
	The restriction of $f$ to $\Gamma_k$ factors through the collapse,
	and the resulting map
	$f_k\colon (\Gamma'_k,\mathscr{G}'_k) \to (\Gamma'_k,\mathscr{G}'_k)$,
	is an irreducible train track map.
\end{prop}

\begin{proof}
	If $H_k$ is non-exponentially-growing,
	then $f_k$ is an automorphism of $(\Gamma'_k,\mathscr{G}'_k)$.
	Therefore suppose $H_k$ is exponentially-growing.

	Let $e$ be an edge of $H_k$.
	By (RTT-i) and (RTT-iii), we have
	\[ f(e) = \alpha_1\gamma_1\dotsb \gamma_\ell\alpha_\ell, \]
	where each $\alpha_i$ is a legal path in $H_k$,
	and each $\gamma_i$ is contained in $\Gamma_{k-1}$.
	After collapsing, each $\gamma_i$
	determines a vertex group element $g_i$.
	Thus
	\[ f_k(e) = \alpha_1g_1\dotsb g_k\alpha_k,  \]
	which is a legal path except possibly at the turns
	where $\alpha_i$ and $\alpha_{i+1}$ meet.
	Suppose $T = \{([g],e),([g'],e')\}$ is such a turn.
	Because $f$ sends edges to immersed paths,
	the only case where $T$ may be illegal
	occurs if $e$ and $e'$ do not form a turn in $\Gamma_k$
	because there we have $\tau(e) \ne \tau(e')$,
	but we have $f^i(\tau(e)) = f^i(\tau(e'))$
	for some $i \ge 1$.

	Write $v = f^{i-1}(\tau(e))$ and $w = f^{i-1}(\tau(e'))$.
	In this case,
	by considering a homotopy inverse for $f$,
	we see that there is some path $\sigma$
	connecting $v$ to $w$ such that $f(\sigma)$ is homotopically trivial.
	(cf. Remark 2.8 of \cite{FeighnHandel}.)
	This contradicts (RTT-ii),
	so we conclude that $f_k(e)$ is a legal path,
	from which the statement follows.
\end{proof}

With the tools developed so far, the existence of relative train track maps
as proven by Bestvina and Handel in \cite[Section 5]{BestvinaHandel}
may be adapted straightforwardly, using the complexity bound in \Cref{complexity}
in place of Bestvina--Handel's $3n-3$.
We leave the adaptation to the reader, and record the following  result.

\begin{thm}\label{rttexistence}
	Every $\varphi\in\Mod(\mathbb{G})$ may be realized by 
	a relative train track map 
	$f\colon (\Gamma,\mathscr{G}) \to (\Gamma,\mathscr{G})$
	where $(\Gamma,\mathscr{G}) \asymp \mathbb{G}$.
	\hfill\qedsymbol
\end{thm}

\begin{rk}
	As shown in \cite{FeighnHandelAlg}, the construction of relative train track
	maps may be made algorithmic---in principle, although not in actual fact
	if our vertex and edge groups are allowed to be arbitrary.
\end{rk}

Another way to prove that
every element $\varphi\in \Mod(\mathbb{G})$ may be realized
by a relative train track map
$f\colon (\Gamma,\mathscr{G}) \to (\Gamma,\mathscr{G})$
with $(\Gamma,\mathscr{G}) \asymp \mathbb{G}$
would be to prove a kind of converse to \Cref{ttfromrtt}.
That is, if one could ``blow up'' subgraphs that were
collapsed by carefully stitching in the associated train track maps,
one could assemble a topological realization from a hierarchy
of train track maps.
In fact, what one constructs from this process is not quite yet
a relative train track map but a \emph{partial train track map}
in the sense of \cite{GaboriauJaegerLevittLustig}.
Bestvina and Handel describe two operations called
\emph{core subdivision} and \emph{collapsing inessential connecting paths}
in \cite[Lemmas 5.13 and 5.14]{BestvinaHandel},
which may be adapted to allow one to modify our partial train track maps 
into relative train track maps.
The details will appear elsewhere.
For the remainder of this section we content ourselves to describing
the ``blowing up'' in the special case where all edge groups are trivial.

We give a slightly different, \emph{stronger} definition than Gaboriau, Jaeger,
Levitt and Lustig. 
Their definition is tuned to the construction of an $\mathbb{R}$-tree,
and is thus only concerned with the top stratum.

\begin{dfn}
	A topological realization $f\colon(\Gamma,\mathscr{G}) \to (\Gamma,\mathscr{G})$
	is a \emph{partial train track map}
	with respect to an $f$-invariant
	maximal filtration
	$\varnothing \subset \Gamma_0 \subset \dotsb \subset \Gamma_m = \Gamma$
	if it satisfies the following conditions:
	\begin{enumerate}
		\item $(\Gamma,\mathscr{G})$ has no 
			inessential valence-one or valence-two vertices.
		\item Each stratum $H_k$ is irreducible.
		\item If $p$ and $q$ are points connected by a path $\sigma$
			contained in the interior
			of an edge $e$ in $H_k$ such that for some $\ell > 1$,
			$f^\ell(p) = f^\ell(q)$ and $f^\ell(\sigma)$ is
			null-homotopic,
			then $f^\ell(\sigma) \subset \Gamma_{k-1}$.
	\end{enumerate}
\end{dfn}

\begin{thm}\label{rttfromtt}
	Let $G$ be a free product.
	Every outer automorphism $\varphi \in \out(G)$
	may be topologically realized by a partial train track map
	$f \colon (\Gamma,\mathscr{G}) \to (\Gamma,\mathscr{G})$,
	where the graph of groups $(\Gamma,\mathscr{G})$ has trivial edge groups.
\end{thm}

\begin{rk}
	Of course, if $\varphi \in \Mod(\mathbb{G})$, 
	where $\mathbb{G}$ is any splitting of $G$
	with trivial edge groups,
	the proof shows that $(\Gamma,\mathscr{G}) \asymp \mathbb{G}$.
\end{rk}

\begin{proof}
	We begin with a topological realization 
	of $\varphi$ on a graph of groups $\mathbb{G}$ with trivial edge groups.
	(One exists, by the results of the previous chapter.)
	We use the method in the proof of \Cref{traintrackthm}
	to construct a hierarchy of train track maps in the following way.
	Suppose during the process of constructing the initial train track map,
	we encounter an invariant subgraph $(\Lambda,\mathscr{L})$ 
	that is not a contractible forest.
	Since $(\Lambda,\mathscr{L})$ is $f$-invariant,
	it makes sense to restrict $f$ to $(\Lambda,\mathscr{L})$.
	In fact, we shall restrict to those components
	of $(\Lambda,\mathscr{L})$ which are not contractible trees.
	Abusing notation, we will call the union of these components
	$(\Lambda,\mathscr{L})$.
	Note that in the strictest sense, 
	$(\Lambda,\mathscr{L})$ is not a graph of groups,
	and $f$ is not a topological realization,
	simply because \emph{a priori} $\Lambda$ may not be connected.
	Nonetheless, the proof of \Cref{traintrackthm} 
	does not rely on the connectivity of $\Lambda$,
	so we may apply it to $(\Lambda,\mathscr{L})$.

	In \Cref{complexity}, we gave a bound on the number of edges
	contained in a marked graph of groups
	without inessential valence-one or valence-two vertices.
	For a Grushko decomposition of $G$ as
	\[ G = A_1 * \dotsb * A_n * F_k, \]
	where the $A_i$ are freely indecomposable and not infinite cyclic
	and $F_k$ is free of rank $k$,
	if $(\Gamma,\mathscr{G})$ is a graph of groups with trivial edge groups 
	and fundamental group $G$
	this bound is $2n + 3k - 3$,
	discounting the exceptional cases $(n,k) = (1,0)$ or $(0,1)$,
	where the bounds are respectively $0$ and $1$.
	This bound strictly decreases both when collapsing an
	invariant subgraph that is not a contractible forest,
	and when passing to such an invariant subgraph.
	Thus the process of passing to invariant subgraphs
	terminates after finitely many iterations.
	As a result, we have a nested sequence of train track maps
	\[ f_1\colon (\Lambda_1,\mathscr{L}_1) \to (\Lambda_1,\mathscr{L}_1),
	\dotsc, f_m\colon (\Lambda_m,\mathscr{L}_m) \to (\Lambda_m,\mathscr{L}_m), \]
	where for each $i$ there is some $j$
	satisfying $i < j$ such that 
	each component of $\Lambda_i$ corresponds to a vertex of
	$\Lambda_j$ with vertex group equal to
	the fundamental group of that component.

	We will inductively construct a new sequence of graphs of groups
	and topological realizations of $\varphi$
	\[ f'_m\colon(\Lambda'_m,\mathscr{L}'_m) \to (\Lambda'_m,\mathscr{L}'_m),
	\dotsc,f'_1\colon(\Lambda'_1,\mathscr{L}'_1) \to (\Lambda'_1,\mathscr{L}'_1)\]
	beginning with $f'_m$ and proceeding downwards.
	The marked graph of groups $(\Lambda'_1,\mathscr{L}'_1)$
	will satisfy $(\Lambda'_1,\mathscr{L}'_1) \asymp \mathbb{G}$.
	Moreover, each component of $(\Lambda_i,\mathscr{L}_i)$
	will correspond to a vertex of $(\Lambda'_{i+1},\mathscr{L}'_{i+1})$
	as above.
	Begin by setting $(\Lambda'_m,\mathscr{L}'_m)  = (\Lambda_m,\mathscr{L}_m)$.

	Assume that each 
	$f'_i\colon(\Lambda'_i,\mathscr{L}'_i) \to (\Lambda'_i,\mathscr{L}_i)$
	has been constructed for $k \le i \le m$.
	The marking on each component $(C,\mathscr{C})$ 
	of $(\Lambda_{k-1},\mathscr{L}_{k-1})$
	includes a choice of basepoint $p_C$
	(which we may assume to be a vertex),
	and a choice of spanning tree.
	Let $v_C$ be the vertex of $\Lambda_{k}$ corresponding to $(C,\mathscr{C})$.
	The graph of groups $(\Lambda'_{k-1},\mathscr{L}'_{k-1})$
	is obtained from the disjoint union of $(\Lambda_{k-1},\mathscr{L}_{k-1})$
	and $(\Lambda'_k,\mathscr{L}'_k)$
	by removing each $v_C$ and reattaching the incident edges to $p_C$.

	The foregoing is the step we are currently unable to take 
	while allowing nontrivial edge groups.
	If all edge groups are finitely generated,
	then for each oriented edge $e \in \st(v_C)$,
	we have that $\iota_e((\mathscr{L}'_k)_e)$ stabilizes
	some vertex of the Bass--Serre tree $\tilde\Lambda_{k-1}$.
	However, if the edge groups differ, it appears that it may happen that
	the marking on $(\Lambda_{k-1},\mathscr{L}_{k-1})$
	does not allow us to choose a fundamental domain in $\tilde\Lambda_{k-1}$
	containing each of these vertices.
	
	We will define $f'_{k-1}$ from $f'_k$ and $f_{k-1}$.
	To avoid notational complexity,
	we will describe the process assuming that
	$\Lambda_{k-1}$ is connected with basepoint $p$
	and that $(\Lambda_{k-1},\Lambda_{k-1})$ corresponds to the vertex
	$v$ in $\Lambda'_k$.
	The general case is entirely analogous.

	We define 
	$f'_{k-1}\colon (\Lambda'_k,\mathscr{L}'_k) \to (\Lambda'_k,\mathscr{L}'_k)$
	in the following way.
	On the subgraph $(\Lambda_{k-1},\mathscr{L}_{k-1})$,
	set $f'_{k-1} = f_{k-1}$.
	If $e$ is an edge of $\Lambda'_k$, its image $f'_k(e)$ is of the form
	\[ f'_k(e) = \sigma_1g_1\dotsb \sigma_{\ell-1} g_{\ell -1}\sigma_\ell, \]
	where each $g_i$ belongs to 
	$(\mathscr{L}'_k)_v = \pi_1(\Lambda_{k-1},\mathscr{L}_{k-1},p)$
	and where each $\sigma_i$ is a tight path in $(\Lambda'_k,\mathscr{L}'_k)$
	whose interior does not meet $v$.
	Assume firstly that $\sigma_1$ and $\sigma_\ell$ are
	nontrivial paths, i.e. that $f'_k(e)$ neither begins nor ends at $v$.
	In this case define
	\[ f'_{k-1}(e) = \sigma_1\gamma_1\dotsb 
	\sigma_{\ell-1}\gamma_{\ell-1}\sigma_{\ell-1}. \]

	To complete the definition, we have to address a technicality.
	We have $(\mathscr{L}'_k)_v = \pi_1(\Lambda_{k-1},\mathscr{L}_{k-1},p)$,
	but the vertex group map 
	$(f'_k)_v\colon (\mathscr{L}'_k)_v \to (\mathscr{L}'_k)_v$
	is an automorphism, \emph{not} an outer automorphism.
	Meanwhile, the basepoint $p$ 
	need not be fixed by $f_{k-1}$, 
	so the train track map 
	$f_{k-1}\colon (\Lambda_{k-1},\mathscr{L}_{k-1}) 
	\to (\Lambda_{k-1},\mathscr{L}_{k-1})$
	need not yet induce a well-defined automorphism of 
	$\pi_1(\Lambda_{k-1},\mathscr{L}_{k-1},p)$,
	much less $(f'_k)_v$.
	To remedy this, choose some tight path $\eta_p$
	from $p$ to $f(p)$ such that the map 
	\[ \gamma \mapsto \eta_pf_{k-1}(\gamma)\bar\eta_p \]
	of loops in $(\Lambda_{k-1},\mathscr{L}_{k-1})$ based at $p$ 
	induces the automorphism $(f'_k)_v$.
	In the notation of the previous paragraph,
	if $f'_k(e)$ begins or ends at $v$,
	the corresponding path $\sigma_1$ or $\sigma_{\ell}$ is trivial,
	and it should be replaced in the definition of $f'_{k-1}(e)$
	by $\bar\eta_p$ or $\eta_p$ respectively and then tightening.
	That is, if $f'_k(e)$ both begins and ends at $v$, we define
	$f'_{k-1}(e)$ by tightening the path
	\[ \bar\eta_p\gamma_1\sigma_2\dotsb\sigma_{\ell-1}\gamma_{\ell-1}\eta_p. \]
	In the general case where $\Lambda_{k-1}$ is not connected,
	the foregoing discussion applies \emph{mutatis mutandis.}
	For instance,
	the path $\eta_C$ for the component $C$ should be
	defined as the appropriate tight path from
	$p_{C'}$ to $f_{k-1}(p_C)$,
	where $C'$ is the component of $\Lambda_{k-1}$ containing $f(p_C)$.
	The rest of the construction follows by induction.

	Note that each 
	$f'_{k}\colon(\Lambda'_k,\mathscr{L}'_k) \to (\Lambda'_k,\mathscr{L}'_k)$
	is a topological realization,
	and $(\Lambda'_1,\mathscr{L}'_1)$ satisfies 
	$(\Lambda'_1,\mathscr{L}'_1) \asymp \mathbb{G}$. 
	Each edge of $\Lambda_i$ determines an edge of $\Lambda'_1$.
	Set $\Gamma_i$ to be the union of the edges of $\Lambda'_1$
	coming from each $\Lambda_j$ for $j \le i$.
	This defines a maximal filtration
	\[ \varnothing = \Gamma_0 \subset \Gamma_1 
	\subset \dotsb \subset \Gamma_m = \Lambda'_1 \]
	preserved by $f$.

	The desired partial train track map is
	$f'_1\colon (\Lambda'_1,\mathscr{L}'_1) \to (\Lambda'_1,\mathscr{L}'_1)$.
	It is clear from the construction that $(\Lambda'_1,\mathscr{L}'_1)$
	has no inessential valence-one nor valence-two vertices,
	simply because each $(\Lambda_i,\mathscr{L}_i)$ satisfy this property.
	Likewise each stratum is irreducible.
	Finally the train track property for each
	$f_i\colon (\Lambda_i,\mathscr{L}_i) \to (\Lambda_i,\mathscr{L}_i)$
	implies that if $p$ and $q$ are points of $H_k$
	connected by a path $\sigma$ that lies in the interior of a single edge $e$,
	then if $f_1'^\ell(p) = f_1'^\ell(q)$ and
	$f_1'^\ell(\sigma)$ is null-homotopic for some $\ell > 1$,
	then $f_1'^\ell(\sigma) \subset \Gamma_{k-1}$.
\end{proof}

\chapter*{Questions and Problems}
\label{openproblems}
\addcontentsline{toc}{chapter}{\nameref{openproblems}}

We collect here fifteen questions we find interesting
for which the techniques in this thesis might prove useful.
No attempt has been made at being comprehensive.
We often default to asking about $\out(W_n)$,
where $W_n$ is the free product of $n$ copies 
of the cyclic group of order $2$.
Many of these questions could just as well be asked
about the outer automorphism group of a free product 
or virtually free group,
or even $\Mod(\mathbb{G})$ as defined in the previous chapter.

\paragraph{Recent Work.}
Automorphism and outer automorphism groups of free products
have attracted a lot of interest recently.
Let us mention a few results.
Das \cite{Das} showed that when $G$ is a free product of $n$ finite
groups satisfying $n \ge 4$,
$\out(G)$ is \emph{thick} in the sense of \cite{BehrstockDrutuMosher},
and thus not relatively hyperbolic.

On the other hand, Healy \cite{Healy} showed that $\out(W_n)$
is \emph{acylindrically hyperbolic,}
a property also enjoyed by $\out(F_n)$ and the mapping class group
of a closed surface.
Genevois and Horbez \cite{GenevoisHorbez} showed that $\aut(G)$
is acylindrically hyperbolic when $G$ is finitely generated and 
has infinitely many ends. 
They remark that in general the following question is open.
\setcounter{question}{0}
\begin{question}
	If $G = A_1 * \dotsb * A_n$ is finitely generated and $n$
	satisfies $n \ge 3$, is  $\out(G)$ acylindrically hyperbolic?
\end{question}

Varghese \cite{Varghese} showed that $\aut(W_n)$ does not have
\emph{Kazhdan's property (T)} as soon as it is infinite.
Her argument exhibits a kind of ``strand-forgetting'' homomorphism
similar to that defined on the pure braid group.
Thus her argument implicitly shows that $\aut(G)$ and $\out(G)$
do not have Kazhdan's property (T) when $G$ is a free product
of $n$ finite groups.

Finally, in a striking parallel with the case of the 
outer automorphism group of the braid group,
Guerch \cite{Guerch} showed that $\out(\out(W_n))$ is trivial
for $n \ge 5$ and a cyclic group of order $2$ when $n = 4$.

\section*{Connections to $\cato$ Groups}

Let us recall the following question from the \nameref{introduction}.
Write $W_n$ for the free product of $n$ copies of the cyclic
group of order $2$.
We say a group $G$ is $\cato$ when there exists a
geometric action of $G$ on a metric space $X$ satisfying
Gromov's non-positive curvature condition $\cato$.

\begin{question}[Ruane]\label{kimquestion}
	Is $\out(W_n)$ a $\cato$ group when $n\ge 4$?
\end{question}

The thrust of this question asks whether $\out(W_n)$
is more similar to the outer automorphism group of a free group,
which is not a $\cato$ group,
or more similar to the quotient of the braid group by its center,
which is a $\cato$ group when the number of strands is at most $6$.
There are a number of ways to address this question,
all of which are interesting in their own right.

Cunningham in his thesis \cite{Cunningham}
shows that \emph{McCullough--Miller space} \cite{McCulloughMiller},
a contractible simplicial complex of dimension $n-2$
on which $\out(W_n)$ acts simplicially
with finite stabilizers and finite  quotient,
cannot support an $\out(W_n)$-equivariant $\cato$ metric.
We would like to highlight two directions this result may be pushed further.
Firstly, McCullough--Miller space is an equivariant
deformation retract of the spine of Guirardel--Levitt's
\emph{Outer Space} for $W_n$, which is also 
a contractible simplicial complex of dimension $n-2$
on which $\out(W_n)$ acts with finite stabilizers and finite quotient.

\begin{question}
	Does Cunningham's result extend to the spine of Guirardel--Levitt
	Outer Space?
\end{question}

Both McCullough--Miller space and the spine of 
Guirardel--Levitt Outer Space 
are defined for free products more generally,
and the action of $\out(G)$ on each space has finite stabilizers
when $G$ is a free product of $n$ finite groups.

\begin{question}
	Does Cunningham's result extend to
	McCullough--Miller space for $\out(G)$,
	where $G$ is a free product of $n$ finite groups and $n\ge 4$?
\end{question}

A weaker question than \Cref{kimquestion} asks
whether $\out(W_n)$ (or $\out(G)$) acts geometrically
on a $\cato$ \emph{cube complex.}
A result of Huang, Jankiewicz and Przytycki \cite{HuangJankiewiczPrzytycki}
in the case $n = 4$ and extended to all $n \ge 4$
by Haettel \cite{Haettel} says
that the $n$-strand braid group cannot act geometrically
on a $\cato$ cube complex.
Thus one expects a negative answer to the following question.

\begin{question}
	For $n \ge 4$, if $G$ is the free product of $n$ finite
	(or infinite cyclic) groups, does $\out(G)$
	act geometrically on a $\cato$ cube complex?
\end{question}

As discussed in the introduction, the present author
shows in \cite{SomeNewCAT0FreeByCyclicGroups}
that a natural class of free-by-cyclic groups
that might provide a negative answer to \Cref{kimquestion}
are in fact (virtually) $\cato$ groups.

We say a free-by-cyclic group $F_n\rtimes\mathbb{Z}$ 
is of \emph{Gersten type} if it contains elements $g$ and $h$
such that
\begin{enumerate}
	\item the subgroup $\langle g, h \rangle$ is
		free  abelian of rank two and
	\item the elements $h$, $gh$ and $g^2h$ are
		all conjugate in $F_n\rtimes\mathbb{Z}$.
\end{enumerate}

A Gersten-type free-by-cyclic group cannot be 
a subgroup of a $\cato$ group.
The previously mentioned result of the author suggests
the following question
\begin{question}
	If a free-by-cyclic group $G$ is not of Gersten type,
	is $G$ (virtually) a $\cato$ group?
	In particular, are all $W_n\rtimes\mathbb{Z}$ groups
	virtually $\cato$?
\end{question}

Ignat Soroko has informed the author that he
and Brady ask whether a free-by-cyclic group $G$ is $\cato$
if and only if it is \emph{virtually special,}
that is, if $G$ acts geometrically on a $\cato$ cube complex
$X$ whose geometry is particularly well-behaved.
(A famous result of Agol and Wise says that
fundamental groups of hyperbolic $3$-manifolds are virtually special.)
A weaker, strictly group-theoretic form of this question appears 
as Question 1 in \cite{BradySoroko}.

The $\cato$ $2$-complexes constructed by the present author
do not admit a cubical structure in general.
Since every $W_n\rtimes\mathbb{Z}$ group is virtually
free-by-cyclic, one avenue for pursuing Brady--Soroko's
question would be to answer the following.
\begin{question}
	Are polynomially-growing $W_n\rtimes\mathbb{Z}$ groups
	virtually special?
\end{question}
	
Finally, Rodenhausen and Wade \cite{RodenhausenWade}
compute a presentation for the centralizers in
$\aut(F_n)$ of \emph{Dehn-twist automorphisms} 
of the free group $F_n$, a class originally defined by
Cohen and Lustig \cite{CohenLustig}.
Using this computation for a generator of $\aut(F_n)$
called a \emph{Nielsen transformation},
they prove very strong restrictions for actions of $\aut(F_n)$
on complete $\cato$ metric spaces.

\begin{question}
	Do centralizers of Dehn-twist automorphisms of $W_n$
	exhibit the same behavior as those of $\aut(F_n)$?
\end{question}

Similar behavior is exhibited by Dehn twists in the mapping class group
of a closed surface, but is absent in the braid groups.

\section*{Topology and Geometry of Outer Spaces}

By analogy with the \emph{Teichm\"uller space}
of a surface, Culler and Vogtmann define in \cite{CullerVogtmann}
a contractible space now called \emph{Outer Space,}
$CV_n$ on which $\out(F_n)$ acts properly discontinuously.
Guirardel--Levitt generalize this construction to
the outer automorphism group of a free product in \cite{GuirardelLevitt},
and Krsti\'c--Vogtmann study fixed-point sets of finite group actions
on $CV_n$ in \cite{KrsticVogtmann}.

The topology of $CV_n$ is fairly well-understood,
but its geometry remains mysterious.
In particular, the most natural choice of metric on Outer Space,
called the \emph{Lipschitz metric} is not symmetric!
Nevertheless, one can realize an irreducible train track map 
with Perron--Frobenius eigenvalue $\lambda > 1$ as defining
an \emph{axis} in the Lipschitz metric 
for the action of the associated outer automorphism
$\varphi \in \out(F_n)$  on $CV_n$.
Francaviglia and Martino prove the existence of relative train track maps
for outer automorphisms of free products by finding such axes
in  \cite{FrancavigliaMartino}.
The results of this thesis imply, for instance,
that if $\varphi  \in \out(F_n)$ 
has an irreducible train track representative
with Perron--Frobenius eigenvalue $\lambda >1$
and $\varphi$ centralizes a finite subgroup $H \le \out(F_n)$,
then $\varphi$ has an axis entirely contained within
the fixed-point set of $H$.
It would be interesting to know if the following holds.

\begin{question}
	Are fixed-point sets of finite subgroups $H \le \out(F_n)$
	convex in the Lipschitz metric?
\end{question}

The \emph{spine} of Culler--Vogtmann and Guirardel--Levitt Outer Space
is a simplicial complex to which the corresponding Outer Space
equivariantly deformation retracts.
We may give the spine a metric by setting each edge to have unit length.
Let $G$ be a free product of finitely generated groups.
Given $\varphi \in \out(G)$ and an automorphism 
$\Phi\colon G \to G$ representing it,
the normal form given in \Cref{combing} defines a preferred
path in the $1$-skeleton of the spine of Guirardel--Levitt
Outer Space from $\varphi.v$ to $v$, for a fixed base vertex $v$.
Although we saw in \Cref{notquasigeodesic} that these paths
are not quasi-geodesic in any choice of 
$\out(G)$-equivariant metric on the spine,
a good understanding of homotopies between such paths
combined with a counting argument could provide upper bounds on
the \emph{Dehn function} of $\out(G)$,
which is a kind of optimal isoperimetric inequality for $\out(G)$
and any space on which it acts geometrically.

\begin{question}
	Do the paths in \Cref{combing} satisfy
	an exponential isoperimetric inequality?
	What is the Dehn function of $\out(G)$?
\end{question}

For $\out(F_n)$, the optimal isoperimetric inequality is exponential,
but both upper and lower bounds remain unknown for $\out(G)$.
A positive answer to \Cref{kimquestion} would yield
a quadratic Dehn function for $\out(W_n)$.
Let us also mention in connection with this question
the question mentioned in \Cref{traintrackchapter}.

\begin{question}
	For which graphs of groups $\mathbb{G}$ is
	the group $\Mod(\mathbb{G})$ a \emph{hierarchically hyperbolic}
	group?
	In particular, is $\out(W_n)$ a hierarchically hyperbolic group?
\end{question}

In \cite{Vogtmann}, Vogtmann studies paths in the $1$-skeleton
of the spine of $CV_n$. By showing how to push loops and homotopies
between loops outside of compact sets, she shows that the spine of $CV_n$
is \emph{simply-connected at infinity} when $n \ge 5$.

\begin{question}
	For $n \ge 6$, is $\out(W_n)$ simply-connected at infinity?
\end{question}

Connectivity at infinity of $\out(W_n)$ would allow one to prove
that $\out(W_n)$ is a virtual duality group in the sense of
Bieri--Eckmann. For this in the case of $\out(W_n)$,
one needs $(n-4)$-connectivity.

\begin{question}
	Is $\out(W_n)$ $(n-4)$-connected at infinity?
\end{question}

\section*{Dynamics of Train Track Maps}
In \cite{ThurstonEntropy}, Thurston shows that every algebraic integer
$\lambda \ge 1$ occurs as the Perron--Frobenius eigenvalue of
an irreducible train track map representing an element $\varphi\in\out(F_n)$
as $n$ is allowed to vary. By comparison,
Perron--Frobenius eigenvalues attached to pseudo-Anosov
diffeomorphisms of surfaces must satisfy
the stronger condition of being \emph{bi-Perron numbers.}
It would be interesting to know the answer to the following question.

\begin{question}
	Does every algebraic integer $\lambda \ge  1$ occur
	as the Perron--Frobenius eigenvalue of a train track map
	representing an element $\varphi \in \out(W_n)$ as $n$ varies?
\end{question}

For each fixed $n$, by contrast, there is a smallest $\lambda > 1$
that can be a Perron--Frobenius eigenvalue for a train track map
representing an element $\varphi \in \out(W_n)$.
This $\lambda$---or rather its logarithm---is the
\emph{least dilatation} of $\out(W_n)$, $L(W_n)$.
Since $W_n$ contains the quotient of the $n$-strand braid group
modulo its center, and every train track map for $\varphi\in\out(W_n)$
determines a train track map for $\hat\varphi\in\out(F_{n-1})$
with identical Perron--Frobenius eigenvalue,
we have 
\[ L(F_{n-1}) \le L(W_n) < L(B_n/Z), \]
where $B_n/Z$ denotes the $n$-strand braid group modulo its center.
The author has examples showing that the right hand inequality is in
fact a strict inequality.

\begin{question}
	Does $L(F_{n-1})$ satisfy $L(F_{n-1}) < L(W_n)$?
	Describe the asymptotics of
	\[ \frac{L(F_{n-1})}{L(W_n)}\quad\text{and}\quad
		\frac{L(W_n)}{L(B_n/Z)}
	\]
	as $n$ tends to infinity.
\end{question}

\bibliography{bib.bib}

\begin{thebibliography}{{Gue}20}

\bibitem[Aus67]{Auslander}
Louis Auslander.
\newblock On a problem of {P}hilip {H}all.
\newblock {\em Ann. of Math. (2)}, 86:112--116, 1967.

\bibitem[Bas93]{Bass}
Hyman Bass.
\newblock Covering theory for graphs of groups.
\newblock {\em J. Pure Appl. Algebra}, 89(1-2):3--47, 1993.

\bibitem[BDM09]{BehrstockDrutuMosher}
Jason Behrstock, Cornelia Dru\c{t}u, and Lee Mosher.
\newblock Thick metric spaces, relative hyperbolicity, and quasi-isometric
  rigidity.
\newblock {\em Math. Ann.}, 344(3):543--595, 2009.

\bibitem[Bes01]{BestvinaCourse}
Mladen Bestvina.
\newblock Folding graphs and applications, d’apr\`es stallings, 2001.

\bibitem[BF91]{BestvinaFeighnGTrees}
Mladen Bestvina and Mark Feighn.
\newblock Bounding the complexity of simplicial group actions on trees.
\newblock {\em Invent. Math.}, 103(3):449--469, 1991.

\bibitem[BF95]{BestvinaFeighnRTrees}
Mladen Bestvina and Mark Feighn.
\newblock Stable actions of groups on real trees.
\newblock {\em Invent. Math.}, 121(2):287--321, 1995.

\bibitem[BFH00]{BestvinaFeighnHandel}
Mladen Bestvina, Mark Feighn, and Michael Handel.
\newblock The {T}its alternative for {${\rm Out}(F_n)$}. {I}. {D}ynamics of
  exponentially-growing automorphisms.
\newblock {\em Ann. of Math. (2)}, 151(2):517--623, 2000.

\bibitem[BH92]{BestvinaHandel}
Mladen Bestvina and Michael Handel.
\newblock Train tracks and automorphisms of free groups.
\newblock {\em Ann. of Math. (2)}, 135(1):1--51, 1992.

\bibitem[BH99]{theBible}
Martin~R. Bridson and Andr\'{e} Haefliger.
\newblock {\em Metric spaces of non-positive curvature}, volume 319 of {\em
  Grundlehren der Mathematischen Wissenschaften [Fundamental Principles of
  Mathematical Sciences]}.
\newblock Springer-Verlag, Berlin, 1999.

\bibitem[BHS17]{BehrstockHagenSisto}
Jason Behrstock, Mark~F. Hagen, and Alessandro Sisto.
\newblock Hierarchically hyperbolic spaces, {I}: {C}urve complexes for cubical
  groups.
\newblock {\em Geom. Topol.}, 21(3):1731--1804, 2017.

\bibitem[BHS19]{BehrstockHagenSistoII}
Jason Behrstock, Mark Hagen, and Alessandro Sisto.
\newblock Hierarchically hyperbolic spaces {II}: {C}ombination theorems and the
  distance formula.
\newblock {\em Pacific J. Math.}, 299(2):257--338, 2019.

\bibitem[BJ96]{BassJiang}
Hyman Bass and Renfang Jiang.
\newblock Automorphism groups of tree actions and of graphs of groups.
\newblock {\em J. Pure Appl. Algebra}, 112(2):109--155, 1996.

\bibitem[BM10]{BradyMcCammond}
Tom Brady and Jon McCammond.
\newblock Braids, posets and orthoschemes.
\newblock {\em Algebr. Geom. Topol.}, 10(4):2277--2314, 2010.

\bibitem[Bow98]{Bowditch}
Brian~H. Bowditch.
\newblock Cut points and canonical splittings of hyperbolic groups.
\newblock {\em Acta Math.}, 180(2):145--186, 1998.

\bibitem[BS19]{BradySoroko}
Noel Brady and Ignat Soroko.
\newblock Dehn functions of subgroups of right-angled {A}rtin groups.
\newblock {\em Geom. Dedicata}, 200:197--239, 2019.

\bibitem[{Bur}16]{Healy}
Brendan {Burns Healy}.
\newblock {Acylindrical Hyperbolicity of Out($W_n$)}.
\newblock {\em arXiv e-prints}, page arXiv:1610.08005, Oct 2016.

\bibitem[BV12]{BridsonVogtmann}
Martin~R. Bridson and Karen Vogtmann.
\newblock The {D}ehn functions of {$Out(F_n)$} and {$Aut(F_n)$}.
\newblock {\em Ann. Inst. Fourier (Grenoble)}, 62(5):1811--1817, 2012.

\bibitem[CL95]{CohenLustig}
Marshall~M. Cohen and Martin Lustig.
\newblock Very small group actions on {${\bf R}$}-trees and {D}ehn twist
  automorphisms.
\newblock {\em Topology}, 34(3):575--617, 1995.

\bibitem[Coo87]{Cooper}
Daryl Cooper.
\newblock Automorphisms of free groups have finitely generated fixed point
  sets.
\newblock {\em J. Algebra}, 111(2):453--456, 1987.

\bibitem[CT94]{CollinsTurner}
D.~J. Collins and E.~C. Turner.
\newblock Efficient representatives for automorphisms of free products.
\newblock {\em Michigan Math. J.}, 41(3):443--464, 1994.

\bibitem[Cul84]{Culler}
Marc Culler.
\newblock Finite groups of outer automorphisms of a free group.
\newblock In {\em Contributions to group theory}, volume~33 of {\em Contemp.
  Math.}, pages 197--207. Amer. Math. Soc., Providence, RI, 1984.

\bibitem[Cun15]{Cunningham}
Charles Cunningham.
\newblock {\em On the automorphism groups of universal right-angled Coxeter
  groups}.
\newblock PhD thesis, Tufts University, 2015.

\bibitem[CV86]{CullerVogtmann}
Marc Culler and Karen Vogtmann.
\newblock Moduli of graphs and automorphisms of free groups.
\newblock {\em Invent. Math.}, 84(1):91--119, 1986.

\bibitem[{Das}18]{Das}
Saikat {Das}.
\newblock {Thickness of $\mathsf{Out}(A_1*...*A_n)$}.
\newblock Available at arXiv:1811.00435 [Math.GR], Nov 2018.

\bibitem[Dun85]{DunwoodyAccessible}
M.~J. Dunwoody.
\newblock The accessibility of finitely presented groups.
\newblock {\em Invent. Math.}, 81(3):449--457, 1985.

\bibitem[Dun98]{Dunwoody}
M.~J. Dunwoody.
\newblock Folding sequences.
\newblock In {\em The {E}pstein birthday schrift}, volume~1 of {\em Geom.
  Topol. Monogr.}, pages 139--158. Geom. Topol. Publ., Coventry, 1998.

\bibitem[FH11]{FeighnHandel}
Mark Feighn and Michael Handel.
\newblock The recognition theorem for {${\rm Out}(F_n)$}.
\newblock {\em Groups Geom. Dyn.}, 5(1):39--106, 2011.

\bibitem[FH18]{FeighnHandelAlg}
Mark Feighn and Michael Handel.
\newblock Algorithmic constructions of relative train track maps and {CT}s.
\newblock {\em Groups Geom. Dyn.}, 12(3):1159--1238, 2018.

\bibitem[FM15]{FrancavigliaMartino}
Stefano Francaviglia and Armando Martino.
\newblock Stretching factors, metrics and train tracks for free products.
\newblock {\em Illinois J. Math.}, 59(4):859--899, 2015.

\bibitem[FR41]{FouxeRabinovitch}
D.~I. Fouxe-Rabinovitch.
\newblock \"{U}ber die {A}utomorphismengruppen der freien {P}rodukte. {II}.
\newblock {\em Rec. Math. [Mat. Sbornik] N. S.}, 9 (51):183--220, 1941.

\bibitem[Ger94]{GerstenFreeByZ}
S.~M. Gersten.
\newblock The automorphism group of a free group is not a {${\rm CAT}(0)$}
  group.
\newblock {\em Proc. Amer. Math. Soc.}, 121(4):999--1002, 1994.

\bibitem[GH20]{GenevoisHorbez}
Anthony {Genevois} and Camille {Horbez}.
\newblock {Acylindrical hyperbolicity of automorphism groups of
  infinitely-ended groups}.
\newblock Available at arXiv:2002.01388, February 2020.

\bibitem[GJ00]{GloverJensen}
Henry~H. Glover and Craig~A. Jensen.
\newblock Geometry for palindromic automorphism groups of free groups.
\newblock {\em Comment. Math. Helv.}, 75(4):644--667, 2000.

\bibitem[GJLL98]{GaboriauJaegerLevittLustig}
Damien Gaboriau, Andre Jaeger, Gilbert Levitt, and Martin Lustig.
\newblock An index for counting fixed points of automorphisms of free groups.
\newblock {\em Duke Math. J.}, 93(3):425--452, 1998.

\bibitem[GL07a]{GuirardelLevittDeformationSpaces}
Vincent Guirardel and Gilbert Levitt.
\newblock Deformation spaces of trees.
\newblock {\em Groups Geom. Dyn.}, 1(2):135--181, 2007.

\bibitem[GL07b]{GuirardelLevitt}
Vincent Guirardel and Gilbert Levitt.
\newblock The outer space of a free product.
\newblock {\em Proc. Lond. Math. Soc. (3)}, 94(3):695--714, 2007.

\bibitem[{Gue}20]{Guerch}
Yassine {Guerch}.
\newblock {Automorphismes du groupe des automorphismes d'un groupe de Coxeter
  universel}.
\newblock Available at arXiv:2002.02223 [Math.GR], February 2020.

\bibitem[{Hae}15]{Haettel}
Thomas {Haettel}.
\newblock {Virtually cocompactly cubulated Artin-Tits groups}.
\newblock Available at arXiv:1509.08711 [Math.GR], September 2015.

\bibitem[HJP16]{HuangJankiewiczPrzytycki}
Jingyin Huang, Kasia Jankiewicz, and Piotr Przytycki.
\newblock Cocompactly cubulated 2-dimensional {A}rtin groups.
\newblock {\em Comment. Math. Helv.}, 91(3):519--542, 2016.

\bibitem[HM13]{HandelMosher}
Michael Handel and Lee Mosher.
\newblock Lipschitz retraction and distortion for subgroups of {${\rm
  Out}(F_n)$}.
\newblock {\em Geom. Topol.}, 17(3):1535--1579, 2013.

\bibitem[KPS73]{KPS}
A.~Karrass, A.~Pietrowski, and D.~Solitar.
\newblock Finite and infinite cyclic extensions of free groups.
\newblock {\em J. Austral. Math. Soc.}, 16:458--466, 1973.
\newblock Collection of articles dedicated to the memory of Hanna Neumann, IV.

\bibitem[KV93]{KrsticVogtmann}
Sava Krsti\'{c} and Karen Vogtmann.
\newblock Equivariant outer space and automorphisms of free-by-finite groups.
\newblock {\em Comment. Math. Helv.}, 68(2):216--262, 1993.

\bibitem[KWM05]{KapovichWeidmannMiasnikov}
Ilya Kapovich, Richard Weidmann, and Alexei Miasnikov.
\newblock Foldings, graphs of groups and the membership problem.
\newblock {\em Internat. J. Algebra Comput.}, 15(1):95--128, 2005.

\bibitem[Lym]{SomeNewCAT0FreeByCyclicGroups}
Rylee~Alanza Lyman.
\newblock Some new {CAT(0)} free-by-cyclic groups.
\newblock In preparation.

\bibitem[Lym19]{TrainTracksOrbigraphs}
Rylee~Alanza Lyman.
\newblock Train tracks, orbigraphs and {CAT(0)} free-by-cyclic groups.
\newblock Available at arXiv:1909.03097 [Math.GR], 2019.

\bibitem[Mac64]{Macbeath}
A.~M. Macbeath.
\newblock Groups of homeomorphisms of a simply connected space.
\newblock {\em Ann. of Math. (2)}, 79:473--488, 1964.

\bibitem[MM96]{McCulloughMiller}
Darryl McCullough and Andy Miller.
\newblock Symmetric automorphisms of free products.
\newblock {\em Mem. Amer. Math. Soc.}, 122(582):viii+97, 1996.

\bibitem[MM99]{MasurMinsky}
Howard~A. Masur and Yair~N. Minsky.
\newblock Geometry of the complex of curves. {I}. {H}yperbolicity.
\newblock {\em Invent. Math.}, 138(1):103--149, 1999.

\bibitem[Nie24]{Nielsen}
Jakob Nielsen.
\newblock Die {I}somorphismengruppe der freien {G}ruppen.
\newblock {\em Math. Ann.}, 91(3-4):169--209, 1924.

\bibitem[Pau91]{Paulin}
Fr\'{e}d\'{e}ric Paulin.
\newblock Outer automorphisms of hyperbolic groups and small actions on {${\bf
  R}$}-trees.
\newblock In {\em Arboreal group theory ({B}erkeley, {CA}, 1988)}, volume~19 of
  {\em Math. Sci. Res. Inst. Publ.}, pages 331--343. Springer, New York, 1991.

\bibitem[Pau97]{Paulin2}
Fr\'{e}d\'{e}ric Paulin.
\newblock Sur les automorphismes ext\'{e}rieurs des groupes hyperboliques.
\newblock {\em Ann. Sci. \'{E}cole Norm. Sup. (4)}, 30(2):147--167, 1997.

\bibitem[QR18]{QingRafi}
Yulan {Qing} and Kasra {Rafi}.
\newblock {Quasi-geodesics in Out(F\_n) and their shadows in sub-factors}.
\newblock {\em arXiv e-prints}, page arXiv:1806.10121, June 2018.

\bibitem[RW15]{RodenhausenWade}
Moritz Rodenhausen and Richard~D. Wade.
\newblock Centralisers of {D}ehn twist automorphisms of free groups.
\newblock {\em Math. Proc. Cambridge Philos. Soc.}, 159(1):89--114, 2015.

\bibitem[Sen81]{Seneta}
E.~Seneta.
\newblock {\em Nonnegative matrices and {M}arkov chains}.
\newblock Springer Series in Statistics. Springer-Verlag, New York, second
  edition, 1981.

\bibitem[Ser03]{Trees}
Jean-Pierre Serre.
\newblock {\em Trees}.
\newblock Springer Monographs in Mathematics. Springer-Verlag, Berlin, 2003.
\newblock Translated from the French original by John Stillwell, Corrected 2nd
  printing of the 1980 English translation.

\bibitem[Sta83]{Stallings}
John~R. Stallings.
\newblock Topology of finite graphs.
\newblock {\em Invent. Math.}, 71(3):551--565, 1983.

\bibitem[Sta91]{StallingsGTrees}
John~R. Stallings.
\newblock Foldings of {$G$}-trees.
\newblock In {\em Arboreal group theory ({B}erkeley, {CA}, 1988)}, volume~19 of
  {\em Math. Sci. Res. Inst. Publ.}, pages 355--368. Springer, New York, 1991.

\bibitem[SW79]{ScottWall}
Peter Scott and Terry Wall.
\newblock Topological methods in group theory.
\newblock In {\em Homological group theory ({P}roc. {S}ympos., {D}urham,
  1977)}, volume~36 of {\em London Math. Soc. Lecture Note Ser.}, pages
  137--203. Cambridge Univ. Press, Cambridge-New York, 1979.

\bibitem[Swa67]{Swan}
Richard~G. Swan.
\newblock Representations of polycyclic groups.
\newblock {\em Proc. Amer. Math. Soc.}, 18:573--574, 1967.

\bibitem[Syk04]{Sykiotis}
Mihalis Sykiotis.
\newblock Stable representatives for symmetric automorphisms of groups and the
  general form of the {S}cott conjecture.
\newblock {\em Trans. Amer. Math. Soc.}, 356(6):2405--2441, 2004.

\bibitem[Thu88]{Thurston}
William~P. Thurston.
\newblock On the geometry and dynamics of diffeomorphisms of surfaces.
\newblock {\em Bull. Amer. Math. Soc. (N.S.)}, 19(2):417--431, 1988.

\bibitem[Thu14]{ThurstonEntropy}
William~P. Thurston.
\newblock Entropy in dimension one.
\newblock In {\em Frontiers in complex dynamics}, volume~51 of {\em Princeton
  Math. Ser.}, pages 339--384. Princeton Univ. Press, Princeton, NJ, 2014.

\bibitem[Tit77]{Tits}
J.~Tits.
\newblock A ``theorem of {L}ie-{K}olchin'' for trees.
\newblock In {\em Contributions to algebra (collection of papers dedicated to
  {E}llis {K}olchin)}, pages 377--388. Academic Press [Harcourt Brace
  Jovanovich, Publishers], New York-London, 1977.

\bibitem[Var19]{Varghese}
Olga Varghese.
\newblock The automorphism group of the universal coxeter group.
\newblock (in press), 2019.

\bibitem[Vog95]{Vogtmann}
Karen Vogtmann.
\newblock End invariants of the group of outer automorphisms of a free group.
\newblock {\em Topology}, 34(3):533--545, 1995.

\end{thebibliography}
\bibliographystyle{alpha}
\end{document}